\documentclass[UTF-8,reqno]{amsart}
\usepackage{enumerate}
\usepackage{amssymb,url,color,  tikz, booktabs}
\usepackage{graphics,graphicx}

\usepackage{mathrsfs}
\usetikzlibrary{arrows, shapes.geometric}
\usetikzlibrary{decorations.pathreplacing}
\tikzstyle{arrow} = [ultrathick,>=stealth]
\tikzstyle{block} = [rectangle, minimum width=3cm, minimum height=1cm, align=flush center, draw=black, thick]

\usepackage{changes}
\def\rep{\replaced}

\makeatletter
\newcommand\RSloop{\@ifnextchar\bgroup\RSloopa\RSloopb}
\makeatother
\newcommand\RSloopa[1]{\bgroup\RSloop#1\relax\egroup\RSloop}
\newcommand\RSloopb[1]%
  {\ifx\relax#1%
   \else
     \ifcsname RS:#1\endcsname
       \csname RS:#1\endcsname
     \else
       \GenericError{(RS)}{RS Error: operator #1 undefined}{}{}%
     \fi
   \expandafter\RSloop
   \fi
  }
\newcommand\X{0}
\newcommand\RS[1]%
  {\begin{tikzpicture}
     [every node/.style=
       {circle,draw,fill,minimum size=1.5pt,inner sep=0pt,outer sep=0pt},
      line cap=round
     ]
   \coordinate(\X) at (0,0);
   \RSloop{#1}\relax
   \end{tikzpicture}
  }
\newcommand\RSdef[1]{\expandafter\def\csname RS:#1\endcsname}
\newlength\RSu
\RSdef{i}{\draw (\X) -- +(90:\RSu) node{};}
\RSdef{l}{\draw (\X) -- +(110:\RSu) node{};}
\RSdef{r}{\draw (\X) -- +(70:\RSu) node{};}
\RSdef{I}{\draw (\X) -- +(90:\RSu) coordinate(\X I);\edef\X{\X I}}
\RSdef{L}{\draw (\X) -- +(120:\RSu) coordinate(\X L);\edef\X{\X L}}
\RSdef{R}{\draw (\X) -- +(60:\RSu) coordinate(\X R);\edef\X{\X R}}

\newcommand{\tone}[1]{#1^{\scalebox{0.8}{\RS{lr}}}}
\newcommand{\ttwo}[1]{#1^{\scalebox{0.8}{\RS{rLrl}}}}
\newcommand{\tthree}[1]{#1^{\scalebox{0.8}{\RS{rLrLrl}}}}
\newcommand{\tfour}[1]{#1^{\scalebox{0.8}{\RS{{Lrl}{Rrl}}}}}
\newcommand{\tze}[1]{#1^{\scalebox{0.8}{\RS{r}}}}

\usepackage{color}
\usepackage[colorlinks=true]{hyperref}
\hypersetup{
    linkcolor=blue,          
    citecolor=red,        
    filecolor=blue,      
    urlcolor=cyan
}

\definecolor{darkergreen}{rgb}{0.0, 0.5, 0.0}

\long\def\zhu#1{{\color{red}\footnotesize Zhu:\ #1}}

\numberwithin{equation}{section}
\def\theequation{\arabic{section}.\arabic{equation}}
\newcommand{\be}{\begin{eqnarray}}
\newcommand{\ee}{\end{eqnarray}}
\newcommand{\ce}{\begin{eqnarray*}}
\newcommand{\de}{\end{eqnarray*}}
\newtheorem{theorem}{Theorem}[section]
\newtheorem{lemma}[theorem]{Lemma}
\newtheorem{remark}[theorem]{Remark}
\newtheorem{definition}[theorem]{Definition}
\newtheorem{proposition}[theorem]{Proposition}
\newtheorem{Examples}[theorem]{Example}
\newtheorem{corollary}[theorem]{Corollary}

\newcommand{\VV}{\mathscr{V}}

\def\eps{\varepsilon}

\def\e{\mathrm{e}}

\def\u{\mathbf{u}}

\def\p{\partial}

\def\[{{\Big[}}
\def\]{{\Big]}}
\def\<{{\langle}}
\def\>{{\rangle}}
\def\({{\Big(}}
\def\){{\Big)}}

\def\bx{{\mathbf{x}}}
\def\tr{\mathrm {tr}}

\def\dif{{\mathord{{\rm d}}}}

\def\no{\nonumber}
\def\={&\!\!=\!\!&}

\def\bB{{\mathbf B}}
\def\bC{{\mathbf C}}

\def\cS{{\mathcal S}}

\def\mA{{\mathbb A}}
\def\mB{{\mathbb B}}

\def\mE{{\mathbb E}}

\def\mH{{\mathbb H}}
\def\mI{{\mathbb I}}

\def\mL{{\mathbb L}}

\def\mN{{\mathbb N}}

\def\mP{{\mathbb P}}

\def\mR{{\mathbb R}}
\def\mS{{\mathbb S}}

\def\mZ{{\mathbb Z}}

\def\bB{{\mathbf B}}

\def\1{{\mathbf{1}}}

\def\sA{{\mathscr A}}

\def\sC{{\mathscr C}}

\def\sF{{\mathscr F}}

\def\sI{{\mathscr I}}

\def\sL{{\mathscr L}}

\def\sV{{\mathscr V}}
\def\sW{{\mathscr W}}

\def\sY{{\mathscr Y}}
\def\E{\mathbb E}

\def\geq{\geqslant}
\def\leq{\leqslant}

\def\eps{\varepsilon}

\def\e{\mathrm{e}}

\def\u{\mathbf{u}}

\def\p{\partial}

\def\[{{\Big[}}
\def\]{{\Big]}}
\def\<{{\langle}}
\def\>{{\rangle}}
\def\({{\Big(}}
\def\){{\Big)}}

\def\bx{{\mathbf{x}}}
\def\tr{\mathrm {tr}}

\def\dif{{\mathord{{\rm d}}}}

\def\no{\nonumber}
\def\={&\!\!=\!\!&}
\def\bt{\begin{theorem}}
\def\et{\end{theorem}}
\def\bl{\begin{lemma}}
\def\el{\end{lemma}}
\def\br{\begin{remark}}
\def\er{\end{remark}}
\def\bx{\begin{Examples}}
\def\ex{\end{Examples}}
\def\bd{\begin{definition}}
\def\ed{\end{definition}}
\def\bp{\begin{proposition}}
\def\ep{\end{proposition}}
\def\bc{\begin{corollary}}
\def\ec{\end{corollary}}

\def\geq{\geqslant}
\def\leq{\leqslant}

 \def\R{\mathbb R}
 \def\R{\mathbb R}

\def\<{\langle} \def\>{\rangle}

\newcommand{\Prec}{\prec\!\!\!\prec}
\allowdisplaybreaks

\begin{document}

\title{Singular HJB equations with applications to KPZ on the real line}

\author{Xicheng Zhang}
\address[X. Zhang]{School of Mathematics and Statistics, Wuhan University, Wuhan,
Hubei 430072, P.R.China
}
\email{XichengZhang@gmail.com}

\author{Rongchan Zhu}
\address[R. Zhu]{Department of Mathematics, Beijing Institute of Technology, Beijing 100081, China; Fakult\"at f\"ur Mathematik, Universit\"at Bielefeld, D-33501 Bielefeld, Germany}
\email{zhurongchan@126.com}

\author{Xiangchan Zhu}
\address[X. Zhu]{ Academy of Mathematics and Systems Science,
Chinese Academy of Sciences, Beijing 100190, China; Fakult\"at f\"ur Mathematik, Universit\"at Bielefeld, D-33501 Bielefeld, Germany}
\email{zhuxiangchan@126.com}

\begin{abstract}

This paper is devoted to studying the Hamilton-Jacobi-Bellman  equations with distribution-valued coefficients, which is not well-defined in the classical sense 
and shall be understood by using paracontrolled distribution method introduced in \cite{GIP15}.  By a new characterization of weighted H\"older space and  
Zvonkin's transformation we prove some new a priori estimates, and therefore, establish the global well-posedness for singular HJB equations.
As an application, the global well-posedness for KPZ equations on the real line in polynomial weighted H\"older spaces is obtained without using Cole-Hopf's transformation. In particular, we solve the conjecture posed in \cite[Remark 1.1]{PR18}.

\end{abstract}


\date{\today}

\maketitle

\tableofcontents

\section{Introduction}

In this paper we are concerned with 
the following singular Hamilton-Jacobi-Bellman equation in $\mR^d$ (abbreviated as HJB):
\begin{equation} 
\label{eq:1}
\sL u:=\left( \partial_t - \Delta \right) u =  H(\nabla u)+ b \cdot \nabla u + f, \quad u (0) = u_0,
\end{equation}
where $H:\mR^d\to\mR$ is a locally Lipschitz function of at most quadratic growth, and
for some $\alpha\in(\frac{1}{2},\frac{2}{3})$ and $\kappa\in(0,1)$,
$$ 
b \in L^\infty_T \bC^{- \alpha} (\rho_{\kappa}), \quad f \in L^\infty_T\bC^{-\alpha}   (\rho_{\kappa}).
$$
Here  $\rho_\kappa(x) :=\<x\>^{-\kappa}:=(1 + | x |^2)^{-\kappa / 2}$ and $\bC^{- \alpha} (\rho_{\kappa})$ stands for the weighted H\"older (or Besov) space (see Section \ref{sec:2.1}). 

\medskip

It is well known that HJB equation appears originally in optimal control theory, whose solution represents the value function of an optimal control problem 
(see \cite{Kry80, YZ99, FS06}). Let us consider the following stochastic optimal control problem:
$$
V(t,x):=\inf_{\gamma}\mE\left[\int^T_t L(s,X^\gamma_s(x),\gamma(s))\dif s+\psi(X^\gamma_T(x))\right],
$$
where the infimum is taken for all controls $\gamma$ being in some class of adapted processes, 
$L$ is the cost function, $\psi$ is the final bequest value, 
and $X^\gamma_t(x)=X^\gamma_t$ is the state process and solves the following SDE:
$$
\dif X^\gamma_t=(b(t,X^\gamma_t)+\gamma_t)\dif t+\sqrt{2}\dif W_t,\ X^\gamma_0=x,
$$
where $W$ is a $d$-dimensional Brownian motion.
Let 
$$
H(t,x,Q):=\inf_{v\in\mR^d}(v\cdot Q+L(t,x,v)).
$$ 
By the dynamical programming principle, $V$ solves the following backward HJB equation:
$$
\p_t V+\Delta V+b\cdot\nabla_x V+H(\nabla V)=0,\ V(T,x)=\psi(x).
$$
Moreover, by the verification theorem, the optimal control $\gamma$ is then given by $\nabla V(t, X^*_t)$, where $X^*_t$ solves the following SDE:
$$
\dif X^*_t=(b(t,X^*_t)+\nabla V(t, X^*_t))\dif t+\sqrt{2}\dif W_t,\ X^*_0=x.
$$
Thus the study of singular HJBs provides us a possibility to study 
the singular stochastic control problem. Here the singularity means that $b$ could be a distribution.

\medskip

Another main motivation of studying HJB \eqref{eq:1} is to solve the following Kardar-Parisi-Zhang (KPZ) equation on the real line:
\begin{equation}\label{1:kpz}
\sL h=``(\p_x h)^{2}\mbox{''}+\xi,\quad h(0)=h_0, 
\end{equation}
where  $\xi$ is a Gaussian space-time white noise on $\mathbb{R}^+\times \R$. The KPZ equation was introduced in \cite{KPZ86} as a model for the growth of interface represented by a height function $h$. In \cite{KPZ86} the authors predicted that under a famous $1-2-3$ scaling the height function must converge to a scale invariant random field which is called KPZ fixed point (see \cite{C12,Qua12,MQR17} and reference therein). Such conjecture is called the strong KPZ universality conjecture. A weaker form of universality which is now called the weak universality conjecture states that the KPZ equation is itself a universal description of the fluctuations of weakly asymmetric growth models (see e.g. \cite{BG97,HQ15,HX18} and reference therein).

\medskip
The main difficulty in solving \eqref{1:kpz}  comes from the singularity of space-time white noise and the nonlinearity, 
which makes $\p_x h$ is not a function and $(\p_x h)^2$ cannot be understood in the classical sense. This problem can be avoided by using the Cole-Hopf transform (see \cite{KPZ86, BCJL94, BG97}), i.e. $w:=\e^h$ formally solves the stochastic heat equation
\begin{align}\label{eq:li}\sL w=w\xi,\quad w(0)=\e^{h_0},\end{align}
which can be understood by It\^o's integration (\cite{Wal86}). 
In \cite{BCJL94, BG97} the solutions to \eqref{1:kpz} are defined by $\log w$ with $w$ being 
the solutions to \eqref{eq:li}. But it remained  unclear whether the Cole-Hopf solution solves the original KPZ equation.

\medskip 

The first rigorous result on solving the original KPZ equation \eqref{1:kpz} 
on the torus is due to Hairer by using rough path theory \cite{Hai13}. Later Hairer  introduced the theory of regularity structures in \cite{Hai14} and Gubinelli, Imkeller and Perkowski proposed paracontrolled distribution method in \cite{GIP15, GP17}, which makes it possible to study a large class of PDEs driven by singular noise. The key ideas of these theories are 
to use the structure of solutions to give a meaning to the not classically well-defined terms. These terms are well-defined with the help of renormalization for the higher order terms of noise. 
More precisely, $(\p_xh)^2$ can be formally understood as a subtraction of an infinite correction term: $(\p_xh)^2-\infty$.  
By a renormalization and decomposition procedure, one can reduce KPZ equation \eqref{1:kpz} 
to an HJB equation \eqref{eq:1} together with other linear equations (see Section \ref{sec:kpz} for more details).

\medskip

Most of the well-known works in the field of singular SPDEs are considered in the finite volume case. Since the main interest for the KPZ equation comes from large scale behavior, it is natural to consider the KPZ equation on the real line. In general the space-time white noise on the infinite volume stays in weighted Besov spaces, and so does the solution. 
This prevents to apply the fixed point argument to construct local solutions. The first work to overcome this difficulty was achieved by Hairer and Labb\'e \cite{HL15, HL18} for the linear rough heat equation by using the exponential weight. For non-linear equation a priori estimate is a natural tool and has been used successfully in the dynamical $\Phi^4_d$ model by Mourrat and Weber \cite{MW17,MW17a} and Gubinelli and Hofmanov\'a \cite{GH18}, which rely on the damping term $-\phi^3$. In \cite{PR18} a priori estimate and a paracontrolled solution to KPZ equation have been obtained for \eqref{1:kpz} by using Cole-Hopf's transform.  Using the probabilistic
notion of energy solutions \cite{GJ14, GJ13,GP18} or studying the assoicated generator and Kolmogorov equation \cite{GP18a} it is possible to give a meaning of the
KPZ equation on $\mR$, but this essentially depends on the invariant measure and is restricted to the initial data, which is absolutely continuous w.r.t. the stationary measure.
In \cite{CWZZ18} martingale solutions have been constructed for geometric stochastic heat equations on infinite volume by using Dirichlet form approach, 
which also relies on the integration by parts formula for the invariant measure.

\medskip
For \eqref{eq:1} we have similar difficulty as \eqref{1:kpz}. 
Since $ b,f \in L_T^\infty \bC^{- \alpha} (\rho_{\kappa})$ and $\alpha>1/2$, the best regularity space for $u$ is $L^\infty_T\bC^{2-\alpha}$ 
by Schauder's estimate. As a result, the transport term $b\cdot\nabla u$
is not well-defined in the classical sense. We need to use regularity structure theory or paracontrolled distribution method to give a meaning to  equation \eqref{eq:1}. The main aim of this paper is to use PDE arguments and paracontrolled distribution method to obtain the global well-posedness of \eqref{eq:1}. 
Notice that for general $H$, we cannot use Cole-Hopf's transform to transform \eqref{eq:1} into a linear equation. 

\subsection{Main results} 

Our goal in the study of the present problem is to make some progress 
in establishing global bounds for singular SPDE's in which  strong damping is not at hand.
As mentioned above, to define $b\cdot\nabla u$ we need to do renormalizations by probabilistic calculations. It is not the main aim of this paper to discuss the renormalization terms as this has been done extensively in the references cited above. For the main result, we suppose that  
the definition of $b\circ \nabla \sI b\in L_T^\infty\bC^{1-2\alpha}(\rho_{2\kappa})$ and 
$b\circ \nabla \sI f\in L_T^\infty\bC^{1-2\alpha}(\rho_{2\kappa})$ are given with $\sI=\sL^{-1}$, i.e. $(b,f)\in \mathbb{B}^\alpha_T(\rho_\kappa)$ (see Section \ref{ssec:para} and Section \ref{sec:2.5}), which in general could be realized by probabilistic calculation (see Section \ref{sec:kpz} for examples). Under this assumption we are mainly concerned with the analysis of the deterministic system in the following.

\medskip


The following result is a special case of main Theorem \ref{existence}.

\bt\label{main} 
Let $\alpha\in(\frac12,\frac23)$ and $\kappa$ be small enough so that $\delta=2(\frac9{2-3\alpha}+1)\kappa<1$, $\bar\alpha=\alpha+\kappa^{1/4}\in(\frac12,\frac23)$.
Suppose that for some $\zeta\in[0,2)$ and $c>0$,
\begin{equation*}
|\p_QH(Q)|\leq c(1+|Q|)\mbox{ and }|H(Q)|\leq c(|Q|^\zeta+1)\ \mbox{for $d\geq 2$},
\end{equation*}
and  
\begin{align*}
\left\{
\begin{aligned}
&\tfrac{1-\alpha}{2}>\eta>\tfrac{2\zeta\delta}{2-\zeta}\vee[2\kappa^{1/4}+2\delta],\ &d\geq2;\\
&\tfrac{1-\alpha}{2}>\eta>2\left[\tfrac{2(3-2\bar\alpha)\delta}{1-\bar\alpha}\vee(\kappa^{1/4}+2\delta)\right],&d=1.
\end{aligned}
\right.
\end{align*}
For any renormailzied pair $(b,f)\in\mB^\alpha_T({\rho_\kappa})$ and 
initial value $u_0\in \bC^\gamma(\rho_{(\gamma-1-\alpha)\delta})$ with $1+\alpha<\gamma<2$,
there exists a unique paracontrolled solution $u\in\mathbb{S}_T^{2-\alpha-\kappa^{1/4}}(\rho_\eta)$ to HJB equation \eqref{eq:1}.
\et

The definition of paracontrolled solutions are given in Definition \ref{def:para1} and Section \ref{Zvonkin}.

As the main application, we obtain well-posedness of \eqref{1:kpz}. The regularity of the space-time white noise $\xi$ is more rough than the coefficient $f$ given in \eqref{eq:1}. To apply Theorem \ref{main} we need to introduce some random distributions and use  Schauder estimate to transform \eqref{1:kpz} to \eqref{eq:1}.  This is the usual way being done for KPZ equation (cf. \cite{Hai13,GP17,PR18}). We use $Y$ to denote the stationary solution to the linear equation
$(\p_t-\Delta)Y=\xi,$ and  $\tone{Y}, \ttwo{Y}$ are random distributions defined in Section \ref{sec:kpz}. 

\bt\label{th:kpz} Let $\kappa>0$ be small enough, $\delta:=40\kappa<1$. 
For $h_0=Y(0)+\widetilde h(0) $ with $\widetilde h(0)\in \bC^{\frac32+2\eps}(\rho_{\eps\delta})$ for $\eps>0$, there exists a unique paracontrolled solution to \eqref{1:kpz} in the sense that $h-Y-\tone{Y}-\ttwo{Y}:=\widetilde h\in \mS_T^{\frac32-2\kappa^{1/4}}(\rho_\eta)$ is a unique paracontrolled solution to \eqref{e:h1} for
$2[(100\kappa)\vee ({\kappa^{1/4}}+80\kappa)]<\eta<\frac14.$

\et

This result improves the weight for the solution space obtained \cite{PR18} and is proved in Theorem \ref{th:k}.


\subsection{Sketch of proof and the structure of the paper}
In Section \ref{sec:2} we first introduce the basic notations and the spaces used throughout the paper.  The regularization effect of heat semigroups and  paracontrolled calculus are recalled in Section \ref{sec:2.2} and Section \ref{ssec:para}, respectively. The conditions for the coefficient $(b,f)$ are discussed in Section \ref{sec:2.5}.

\medskip

The bulk of our argument is contained in Sections \ref{linear}-\ref{Zvonkin} and we now proceed to explain the strategy. 
We separate \eqref{eq:1} as the following two equations:
\begin{equation} \label{eq:2}\left( \partial_t - \Delta \right) w = b \cdot \nabla w  + f, \qquad w(0) = w_0\end{equation}
\begin{equation} \label{eq:3}\left( \partial_t - \Delta \right) u = b \cdot \nabla u  + H(\nabla w+\nabla u), \qquad u(0) = u_0.\end{equation}
In Section \ref{linear} we first establish Schauder estimate for \eqref{eq:2} with sublinear weights (see Theorem \ref{Th33}). 
This solves the conjecture proposed in \cite[Remark 1.1]{PR18}. The difficulty to study \eqref{eq:2} lies in the loss of weight for $b$ part on the right hand side.  It is possible to use  the technique in \cite{HL18} to solve the problem. However, by the technique in \cite{HL18} the solution will stay in the Besov space with  exponential weight, which seems not easy to be used to obtain a uniform $L^\infty(\rho_\delta)$ estimate for solution to \eqref{eq:3}. The key idea is to use a new characterization of the weighted H\"older space (see Lemma \ref{cha}) to localize the problem with coefficient in unweighted Besov spaces.  To this end, we first establish the Schauder estimate with the coefficient in unweighted Besov space in Section \ref{sub:Schauder}. Here we want to emphasize that
the estimate depends polynomially on the norm of the coefficient compared to the exponential dependence by the usual Gronwall type argument. To obtain this, we add a new damping term $\lambda w$ to \eqref{eq:2}, for which a uniform estimate is easy to be established by choosing $\lambda$ large enough. Then by a classical maximum principle, we obtain the Schauder estimate for the solutions to  \eqref{eq:2} depending polynomially on the coefficient.  In Section \ref{sec:3.3} we establish global well-posedness of equation \eqref{eq:2} and   a uniform estimate of solution to \eqref{eq:2} in Besov space with sublinear weight.
\begin{figure}
	\begin{tikzpicture}
	\node (lnd) [block] {Equation \eqref{eq:1}\\containing $b, f$ and nonlinear term $H$};
	\node (13) [block, below of=lnd, xshift=2cm, yshift=-1.5cm] {Equation \eqref{eq:3}\\without $f$};
	\node (12) [block, left of=13, xshift=-4cm, inner sep = 0pt] {Equation \eqref{eq:2}\\without nonlinear term};
	
	\node (14) [block, below of=13, yshift=-1.5cm]  {Section \ref{s:HJB}\\
		Well-posedness of \eqref{eq:5}\\every term is a function};
	\node (15) [block, below of=12, yshift=-1.5cm]  {Section \ref{linear}\\
		Well-posedness of \eqref{eq:2}\\
		Solution: sublinear growth};
	
	\node (16) [block, below of=lnd, yshift=-6cm]  {Section \ref{Zvonkin}\\Well-posedness of \eqref{eq:1}};

	\draw[->,line width=0.4mm] (lnd) --node[anchor=east] {Separation} (13);
	\draw[->,line width=0.4mm] (lnd) -- (12);
	
	\draw[->,line width=0.4mm] (13) --node[anchor=west] {Zvonkin transform} (14);
	\draw[->,line width=0.4mm] (12) --node[anchor=west] {localization} (15);
	
	\draw[->,line width=0.4mm] (14) -- (16);
	\draw[->,line width=0.4mm] (15) -- (16);
	\end{tikzpicture}	
\end{figure} 

\medskip

We then study \eqref{eq:3} in Sections \ref{s:HJB} and \ref{Zvonkin}. Compared to \eqref{eq:1} the distribution-valued 
$f$ has been changed to  function-valued. But we still have a singular transport term $b\cdot \nabla u$ with distribution-valued $b$ in \eqref{eq:3}.
 In the classical PDE theory (see \cite{LSU68}) we may  use De-Giorgi's method    to   obtain better regularity. However, the singularity of $b$ makes it not easy.
Instead, we use Zvonkin's transform to transform \eqref{eq:3} to the following general  HJB equation (see Section \ref{Zvonkin})
\begin{align}\label{eq:5}
\p_t v=\tr(a\cdot\nabla^2 v)+B\cdot\nabla v+\widetilde{H}(v,\nabla v),\ v(0)=\varphi,
\end{align}
where $a\in L^\infty_T\bC^{1-\alpha} $ is symmetric, uniformly elliptic,
$B\in \mL^\infty_T(\rho_{\delta_1})$ for some $\delta_1\in(0,1]$.
All the coefficients of \eqref{eq:5} are function-valued with the cost that \eqref{eq:5} is given as a non-divergence form.  To be more precise, we use \cite{GH18} to decompose $b$ into a less regular term $b_>$ in the unweighted Besov space and a function-valued term $b_\leq$. Then we use Zvonkin's transform to kill $b_>$. The idea comes from Zvonkin's transform for SDEs, but our Zvonkin's transform is different from the normal one and it is the first time to be used for dealing with nonlinear PDE \eqref{eq:3}. We  emphasize that we need to construct a $C^1$-diffeomorphism by solving a linear equation similar as \eqref{eq:2} with $b_>$ as the coefficient.


\medskip

Section \ref{s:HJB} is devoted to the global well-posedness of equation \eqref{eq:5} (see Theorem \ref{Th42}). We first establish a maximum principle in Section \ref{sec:4.1} by Feymann-Kac formula. For the subcritical case, the global estimate follows from $L^\infty(\rho_\delta)$-estimate and $L^p$ theory of PDEs. For the critical case, the proof is more involved.  We can only treat $d=1$ case. In this case by taking spatial derivative on both sides, we obtain a divergence PDE. Then  the  $L^\infty(\rho_\delta)$-bound and energy estimate yield the $\mH^{2,p}_T(\rho_\eta)$-estimate of the solution to equation \eqref{eq:5}. By using this and Zvonkin's transform we finally establish global estimate for solutions to \eqref{eq:3} and well-posedness of \eqref{eq:1} in Section \ref{Zvonkin}. 

Now we use the above picture to see our steps to solve the problem.

 Section \ref{sec:kpz} is devoted to the application to the KPZ equation and the proof of Theorem \ref{th:kpz}. Finally in Appendix  \ref{sec:7.1} we give the uniqueness of solutions to \eqref{eq:1} based on the exponential weight approach developed in \cite{HL18}. Appendix \ref{sec:7.2} is then devoted to an exponential moment estimate for SDEs used in Section \ref{s:HJB}. 

\subsection{Conventions and notations}
Throughout this paper, we use $C$ or $c$ with or without subscripts to denote an unrelated constant, whose value
may change in different places. We also use $:=$ as a way of definition. By $A\lesssim_C B$ and $A\asymp_C B$
or simply $A\lesssim B$ and $A\asymp B$, we mean that for some constant $C\geq 1$,
$$
A\leq C B,\ \ C^{-1} B\leq A\leq CB.
$$
For convenience, we list some commonly used notations and definitions below.
$$
\begin{tabular}{c|c}\toprule
$\sC^\alpha(\rho)$: weighted H\"older space (Def. \ref{Def23}) & $\sC^\alpha:=\sC^\alpha(1)$ \\ \midrule
$\bB^\alpha_{p,q}(\rho)$: weighted Besov space (Def. \ref{Def25}) & $\bB^\alpha_{p,q}:=\bB^\alpha_{p,q}(1)$\\ \midrule
$\bC^\alpha(\rho)$: weighted H\"older-Zygmund space (Def. \ref{Def25})&  $\bC^\alpha:=\bC^\alpha(1)$ \\ \midrule
$\mS^\alpha_T(\rho)$: Paracontrolled solution space \eqref{SS0}  &$\mS^\alpha_T:=\mS^\alpha_T(1)$ \\ \midrule
$\mB^\alpha_T(\rho)$: Space of renormalized pair (Def. \ref{Def216}) & $\mB^\alpha_T:=\mB^\alpha_T(1)$\\ \midrule
$f\prec g, f\succ g, f\circ g$: Paraproduct (Sec. \ref{ssec:para})  &  $f\succcurlyeq g:=f\succ g+f\circ g$ \\ \midrule
$f\Prec g$: Modified paraproduct (Sec. \ref{ssec:para})  &  $\sL_\lambda:=\p_t-\Delta+\lambda$ \\ \midrule
${\rm com}(f,g,h):=(f\prec g)\circ h-f(g\circ h)$ (Sec. \ref{ssec:para})  & $\sI_\lambda:=(\p_t-\Delta+\lambda)^{-1}$\\ \midrule
$\sV_> f$,\ $\sV_\leq f$: Localization operator (Sec. \ref{ssec:para})  & $\sL:=\sL_0$,\ $\sI:=\sI_0$\\ \midrule
$P_tf(x):=(4\pi t)^{-d/2}\int_{\mR^d}f(y)\e^{-|x-y|^2/(4t)}\dif y$ & $B_r:=\{x:|x|\leq r\}$\\ \midrule
$\sI_{s}^tf(x):=\int_s^tP_{t-r}f(r,x)\dif r$  & $\<x\>:=(1+|x|^2)^{1/2}$\\\midrule 
Commutator: $[\sA_1,\sA_2]f:=\sA_1(\sA_2 f)-\sA_2(\sA_1f)$ & $\mN_0:=\mN\cup\{0\}$ \\\midrule 
\bottomrule
\end{tabular}
$$

\section{Preliminaries}\label{sec:2}

\subsection{Weighted Besov spaces}\label{sec:2.1}
We first recall the following definition about the admissible weight
introduced in \cite{Tri06}.
\bd
A $C^\infty$-smooth function $\rho:\mR^d\to(0,\infty)$ is called an admissible weight if
for each $j\in\mN$, there is a constant $C_j>0$ such that
$$
|\nabla^j\rho(x)|\leq C_j\rho(x),\ \ \forall x\in\mR^d,
$$
and for some $C,\beta>0$,
$$
\rho(x)\leq C\rho(y)(1+|x-y|)^\beta,\ \ \forall x,y\in\mR^d.
$$
The set of all the admissible weights is denoted by $\sW$.
\ed
\bx
Let $\rho_\delta(x)=\<x\>^{-\delta}=(1+|x|^2)^{-\delta/2}$, where $\delta\in\mR$. It is easy to see that $\rho_\delta\in\sW$.
Such a weight is called polynomial weight.
\ex
 We introduce the following weighted H\"older space for later use.
\bd (Weighted H\"older spaces)\label{Def23}
Let $\rho\in\sW$ and $k\in\mN_0$. For $\alpha\in[0,1)$, we define the weighted H\"older space $\sC^{k+\alpha}(\rho)$ by the norm
\begin{align*}
\|f\|_{\sC^{k+\alpha}(\rho)}:=\sum_{j=0}^k\|\nabla^j(\rho f)\|_{L^\infty}+\sup_{x\not=y}\frac{|\nabla^k(\rho f)(x)-\nabla^k(\rho f)(y)|}{|x-y|^\alpha}<\infty.
\end{align*}
\ed
\br
\rm
By the properties of admissible weights and elementary calculations, it is easy to see that  for some $C=C(d,\rho)\geq 1$,
\begin{align}
\|f\|_{\sC^{k+\alpha}(\rho)}&\asymp_C\sum_{j=0}^k\|\rho\nabla^jf \|_{L^\infty}
+\sup_{|x-y|\leq 1}\frac{|(\rho\nabla^k f)(x)-(\rho\nabla^k f)(y)|}{|x-y|^\alpha}\no
\\&\asymp_C\sum_{j=0}^k\|\rho\nabla^jf \|_{L^\infty}+\sup_{|x-y|\leq 1}\frac{\rho(x)|\nabla^k f(x)-\nabla^k f(y)|}{|x-y|^\alpha}\label{eq}.
\end{align}
\er

Let $\mathcal{S}(\R^{d})$ be the space of Schwartz functions on $\R^{d}$ and $\mathcal{S}'(\R^{d})$ the space of tempered distributions, 
which is the dual space of $\mathcal{S}(\R^{d})$. 
The Fourier transform of $f\in \mathcal{S}'(\R^{d})$ is defined through
$$
\widehat f(z):=(2\pi)^{-d/2}\int_{\R^{d}}f(x)\e^{-i z\cdot x}\dif x.
$$
For $j\geq -1$, let $\Delta_j$ be the usual block operator used in the Littlewood-Paley decomposition so that
for any $f\in\cS'(\mR^d)$ ({\cite{BCD11}}), 
$$
\Delta_j f\in\cS,\ \ {\rm supp}(\widehat{\Delta_j f})\subset B_{2^{j+2}}\setminus B_{2^{j-1}},\ j\in\mN_0,
$$
and
$$
{\rm supp}(\widehat{\Delta_{-1} f})\subset B_1,\ \ f=\sum_{j\geq -1}\Delta_j f.
$$
We also introduce the following weighted Besov spaces (cf. \cite{Tri06}):
\bd\label{Def25}
Let $\rho\in\sW$ and $p,q\in[1,\infty]$ and $\alpha\in\mR$. The weighted Besov space $\bB^\alpha_{p,q}(\rho)$ is defined by
$$
\bB^\alpha_{p,q}(\rho):=\left\{f\in\cS'(\mR^d): 
\|f\|_{\bB^\alpha_{p,q}(\rho)}:=\left(\sum_j 2^{\alpha jq}\|\Delta_j f\|_{L^p(\rho)}^q\right)^{1/q}<\infty\right\},
$$
where 
$$
\|f\|_{L^p(\rho)}:=\|\rho f\|_p:=\left(\int_{\mR^d}|\rho(x)f(x)|^p\dif x\right)^{1/p}.
$$
The weighted H\"older-Zygmund space is defined by
$$
\bC^\alpha(\rho):=\bB^\alpha_{\infty,\infty}(\rho).
$$
\ed

\br\rm
Let $\rho\in\sW$. For any $0<\beta\notin\mN$ and $\alpha\in\mR$, $p,q\in[1,\infty]$, it is well known that 
(see \cite[Theorem 6.5, Theorem 6.9]{Tri06}, \cite[page99]{BCD11})
\begin{align}\label{DM9}
\|f\|_{\bC^\beta(\rho)}&\asymp \|f\|_{\sC^\beta(\rho)},\ \|f\|_{\bB^\alpha_{p,q}(\rho)}\asymp\|f\rho\|_{\bB^\alpha_{p,q}}.
\end{align}

\er

For $T>0$, $\alpha\in\mR$ and an admissible weight $\rho\in\sW$, let $L^\infty_T\bC^{\alpha}(\rho)$ be the space of space-time distributions with finite norm
$$
\|f\|_{L^\infty_T\bC^{\alpha}(\rho)}:=\sup_{0\leq t\leq T} \| f(t)\|_{\bC^{\alpha}(\rho)}<\infty.
$$
For $\alpha\in(0,1)$ we denote by $C_T^{\alpha}L^\infty(\rho)$  the space of $\alpha$-H\"older continuous
mappings $f: [0,T]\to  L^\infty(\rho) $ with finite norm
$$
\|f\|_{C_T^{\alpha}L^\infty(\rho)}:=\sup_{0\leq t\leq T} \|f(t)\|_{L^\infty(\rho)}+\sup_{0\leq s\neq t\leq T} \frac{\|f(t)-f(s)\|_{L^\infty(\rho)}}{|t-s|^{\alpha}}.
$$
The following space will be used frequently: for $\alpha\in(0,2)$,
\begin{align}\label{SS0}
\mS^\alpha_T(\rho):=\Big\{f: \|f\|_{\mS^{\alpha}_T(\rho)}:=&\|f\|_{L^\infty_T\bC^\alpha(\rho)}+\|f\|_{C_T^{\alpha/2}L^\infty(\rho)}<\infty\Big\}.
\end{align}
We have the following simple fact  (see \cite[Lemma 2.11]{PR18}): for $\alpha\in (0,1)$, 
\begin{align}\label{*}
\|\nabla f\|_{\mS^\alpha_T(\rho)}\lesssim \|f\|_{\mS^{\alpha+1}_T(\rho)}.
\end{align}
Moreover, by interpolation it is easy to see that for $0<\kappa<\alpha$,
\begin{align*}
\|f\|_{C_T^{\kappa/2}\bC^{\alpha-\kappa}(\rho)}\lesssim \|f\|_{\mS_T^\alpha(\rho)}.
\end{align*}
For $p\in[1,\infty]$, $k\in\mN_0$ and $T>0$, we also need the following Sobolev space:
$$
\mH^{k,p}_T:=\Big\{f: \|f\|_{\mH^{k,p}_T}:=\|f\|_{\mL^p_T}+\|\nabla^k f\|_{\mL^p_T}<\infty\Big\},
$$
where, with the usual modification when $p=\infty$,
$$
\|f\|_{\mL^p_T}:=\left(\int^T_0\!\!\!\int_{\mR^d} |f(t,x)|^p\dif x\dif t\right)^{\frac 1p}.
$$ 
For an admissible weight $\rho$, we also introduce the weighted Sobolev space
$$
\mH^{k,p}_T(\rho):=\Big\{f: \|f\|_{\mH^{k,p}_T(\rho)}:=\|f\rho\|_{\mH^{k,p}_T}<\infty\Big\},
$$
and local space $\mH^{k,p}_{\mathrm{loc}}$:
$$
\mH^{k,p}_{\mathrm{loc}}:=\Big\{f: f\chi_R\in\mH^{k,p}_{T},\ \ \forall T, R>0\Big\},
$$
where $\chi_R$ is the usual cutoff function. 

The following interpolation inequality will be used frequently, which are easy consequence of H\"older's inequality and the corresponding definition.
 (see \cite[Lemma A.3]{GH18a} for a discrete version).
\bl\label{Le32}
Let $\rho\in\sW$ and $\theta\in[0,1]$. Let $\alpha,\alpha_1,\alpha_2\in\mR$ and
$\delta,\delta_1,\delta_2\in \R$ satisfy
$$
\delta=\theta\delta_1+(1-\theta)\delta_2,\ \alpha=\theta \alpha_1+(1-\theta)\alpha_2,
$$
and $p,q,p_1,q_1,p_2,q_2\in[1,\infty]$ satisfy
$$
\tfrac{1}{p}=\tfrac{\theta}{p_1}+\tfrac{1-\theta}{p_2},\ \ \tfrac{1}{q}=\tfrac{\theta}{q_1}+\tfrac{1-\theta}{q_2}.
$$
Then we have
\begin{align}\label{DQ1}
\|f\|_{\bB^\alpha_{p,q}(\rho^\delta)}\leq \|f\|_{\bB^{\alpha_1}_{p_1,q_1}(\rho^{\delta_1})}^\theta\|f\|_{\bB^{\alpha_2}_{p_2,q_2}(\rho^{\delta_2})}^{1-\theta}.
\end{align}
Moreover, for any $0<\alpha<\beta<2$ with $\theta=\alpha/\beta$, we also have
\begin{align}\label{AM1}
\|f\|_{\mS^\alpha_T(\rho^\delta)}\lesssim \|f\|_{\mS^\beta_T(\rho^{\delta_1})}^{\theta}
\|f\|_{\mL^\infty_T(\rho^{\delta_2})}^{1-\theta}.
\end{align}
\el

\subsection{Estimates of Gaussian heat semigroups}\label{sec:2.2}
For $t>0$, let $P_t$ be the Gaussian heat semigroup defined by
$$
P_t f(x):=(4\pi t)^{-d/2}\int_{\mR^d}\e^{-|x-y|^2/(4t)}f(y)\dif y.
$$
Let $\rho$ be an admissible weight. It is well know that there is a constant $C=C(\rho,d)>0$ such that (see \cite[Lemma 2.10]{MW17})
	\begin{align}\label{Es19}
	\|\Delta_j P_t f\|_{L^\infty(\rho)}\lesssim_C \e^{-2^{2j}t}\|\Delta_jf\|_{L^\infty(\rho)},\ j\geq -1, t\geq 0.
	\end{align}

We have the following estimates about the Gaussian heat semigroup.

\bl\label{lem:2.8}
Let $\rho$ be an admissible weight. 
\begin{enumerate}[(i)]
\item For any $\theta>0$ and $\alpha\in\mR$, there is a constant $C=C(\rho,d,\alpha,\theta)>0$ such that
\begin{align}\label{E1}
\|P_t f\|_{\bC^{\theta+\alpha}(\rho)}\lesssim_C t^{-\theta/2}\|f\|_{\bC^\alpha(\rho)},\ t>0.
\end{align}
\item For any $m\in\mN_0$ and $\theta<m$, there is a constant $C=C(\rho,d,m,\theta)>0$ such that
\begin{align}\label{E4}
\|\nabla^m P_t f\|_{L^\infty(\rho)}\lesssim_C t^{(\theta-m)/2}\|f\|_{\bC^\theta(\rho)},\ t>0.
\end{align}
\item For any $0<\theta<2$, there  is a constant $C=C(\rho,d,\theta)>0$ such that
\begin{align}\label{E2}
\|P_t f-f\|_{L^\infty(\rho)}\lesssim_C t^{\theta/2}\|f\|_{\bC^{\theta}(\rho)},\ t>0.
\end{align}
\end{enumerate}
\el
\begin{proof} 
(i) By the definition and \eqref{Es19}, we have
\begin{align*}
\|P_t f\|_{\bC^{\theta+\alpha}(\rho)}
&=\sup_j 2^{(\theta+\alpha)j}\|\Delta_j P_t f\|_{L^\infty(\rho)}
\lesssim\sup_j 2^{(\theta+\alpha)j}\e^{-2^{2j}t}\|\Delta_jf\|_{L^\infty(\rho)}
\\&\leq\sup_j 2^{\theta j}\e^{-2^{2j}t}\|f\|_{\bC^\alpha(\rho)}\lesssim t^{-\theta/2}\|f\|_{\bC^\alpha(\rho)}.
\end{align*}
(ii) For $m\in\mN_0$ and $\theta<m$, by \eqref{Es19} we have
\begin{align*}
\|\nabla^m P_t f\|_{L^\infty(\rho)}&\leq\sum_j\|\nabla^m \Delta_j P_t f\|_{L^\infty(\rho)}
\lesssim\sum_j 2^{mj}\e^{-2^{2j}t}\|\Delta_jf\|_{L^\infty(\rho)}
\no\\&\lesssim\sum_j(2^{mj}\e^{-2^{2j}t}2^{-\theta j})\|f\|_{\bC^{\theta}(\rho)}
\lesssim t^{(\theta-m)/2}\|f\|_{\bC^{\theta}(\rho)}.
\end{align*}
(iii) By \eqref{E4}, we have
	$$
	\|P_t f-f\|_{L^\infty(\rho)}=\left\|\int^t_0\Delta P_s f\dif s\right\|_{L^\infty(\rho)}
	\lesssim \int^t_0s^{-1+\theta/2}\|f\|_{\bC^{\theta}(\rho)}\dif s
	\lesssim t^{\theta/2}\|f\|_{\bC^{\theta}(\rho)}.
	$$
	The proof is complete.
\end{proof}

For given $\lambda\geq 0$ and $f\in L^\infty(\mR_+; L^\infty(\mR^d))$, we consider the following heat equation:
$$
\sL_\lambda u:=(\p_t-\Delta+\lambda)u=f,\ \ u(0)=0.
$$
The unique solution of this equation is given by
$$
u(t,x)=\int^t_0\e^{-\lambda(t-s)}P_{t-s}f(s,x)\dif s=:\sI_\lambda f(t,x).
$$
In other words, $\sI_\lambda$ is the inverse of $\sL_\lambda$. 

The following Schauder estimate is well known for $q=\infty$ and $\theta=2$ (see \cite{GH18}).
\bl\label{Le11}
(Schauder estimates in weighted space) Let $\rho\in\sW$ and
$$
\alpha\in(0,1],\ \ \theta\in(\alpha,2].
$$
For any $q\in[\frac{2}{2-\theta},\infty]$, there is a constant $C=C(\rho, d,\alpha,\theta,q)>0$
such that for all $\lambda, T\geq 0$ and $f\in L^q_T\bC^{-\alpha}(\rho)$,
\begin{align}\label{EG1}
\|\sI_\lambda f\|_{\mS^{\theta-\alpha}_T(\rho)}
\lesssim_C(\lambda\vee 1)^{\frac{\theta}{2}+\frac{1}{q}-1}\|f\|_{L^q_T\bC^{-\alpha}(\rho)}.
\end{align}
\el
\begin{proof}
	Let $q\in[\frac{2}{2-\theta},\infty]$ and $\frac{1}{p}+\frac{1}{q}=1$. For $t\in(0,T]$,  by \eqref{Es19} and H\"older's inequality, we have
	\begin{align*}
	2^{j(\theta-\alpha)}\|\Delta_j\sI_\lambda f(t)\|_{L^\infty(\rho)}&\lesssim 2^{j(\theta-\alpha)}\int^t_0\e^{-(\lambda+2^{2j})(t-s)}\|\Delta_jf(s)\|_{L^\infty(\rho)}\dif s\\
	&\lesssim 2^{j\theta}\left(\int^t_0\e^{-p(\lambda+2^{2j})(t-s)}\dif s\right)^{\frac{1}{p}}
	\left(\int^t_0\|f(s)\|^q_{\bC^{-\alpha}(\rho)}\dif s\right)^{\frac{1}{q}}\\
	&\lesssim2^{j\theta}\left(\int^t_0\e^{-p(\lambda+2^{2j}) s}\dif s\right)^{\frac{1}{p}}\|f\|_{L^q_T\bC^{-\alpha}(\rho)}
	\\&\lesssim 2^{j\theta}(2^{2j}+\lambda)^{-\frac{1}{p}}\|f\|_{L^q_T\bC^{-\alpha}(\rho)}
	\lesssim(\lambda\vee 1)^{\frac{\theta}{2}-\frac{1}{p}}\|f\|_{L^q_T\bC^{-\alpha}(\rho)},
	\end{align*}
	which implies by the definition of Besov space
	\begin{align}\label{EG01}
	\|\sI_\lambda f\|_{L^\infty_T\bC^{\theta-\alpha}(\rho)}\lesssim_C (\lambda\vee 1)^{\frac{\theta}{2}+\frac{1}{q}-1}\|f\|_{L^q_T\bC^{-\alpha}(\rho)}.
	\end{align}
	On the other hand, let $u=\sI_\lambda f$. For $0\leq t_1< t_2\leq T$, we have
	\begin{align*}
	u(t_2)-u(t_1)&=\int_0^{t_1}(\e^{-\lambda(t_2-s)}-\e^{-\lambda(t_1-s)}) P_{t_2-s}f(s)\dif s\\
	&\quad+(P_{t_2-t_1}-I)\sI_\lambda f(t_1)+\int_{t_1}^{t_2}\e^{-\lambda(t_2-s)}P_{t_2-s}f(s) \dif s\\
	&=:I_1+I_2+I_3.
	\end{align*}
	For $I_1$, by \eqref{E1} and H\"older's inequality, we have
	\begin{align*}
	\|I_1\|_{L^\infty(\rho)}&\leq |\e^{-\lambda(t_2-t_1)}-1|\int_0^{t_1}\e^{-\lambda(t_1-s)}\|P_{t_2-s}f(s)\|_{L^\infty(\rho)} ds\\
	&\leq \Big((\lambda(t_2-t_1))\wedge1\Big)\int_0^{t_1}\e^{-\lambda(t_1-s)}(t_2-s)^{-\frac{\alpha}{2}}\|f(s)\|_{\bC^{-\alpha}(\rho)} ds\\
	&\leq (\lambda(t_2-t_1))^{\frac{\theta}{2}}(t_2-t_1)^{-\frac{\alpha}{2}}\left(\int_0^{t_1}
	\e^{-\lambda(t_1-s)p}ds\right)^{1/p}\|f\|_{L^q_T\bC^{-\alpha}(\rho)}\\
	&\lesssim (t_2-t_1)^{\frac{\theta-\alpha}{2}}\lambda^{\frac{\theta}{2}-\frac{1}{p}}\|f\|_{L^q_T\bC^{-\alpha}(\rho)}.
	\end{align*}
	For $I_2$, by \eqref{E2} and \eqref{EG01} we have
	\begin{align*}
	\|I_2\|_{L^\infty(\rho)}&\leq (t_2-t_1)^{\frac{\theta-\alpha}{2}}\|\sI_\lambda f\|_{L^\infty_T\bC^{\theta-\alpha}(\rho)}\\
	&\lesssim (t_2-t_1)^{\frac{\theta-\alpha}{2}}(\lambda\vee1)^{\frac{\theta}{2}-\frac{1}{p}}\|f\|_{L^q_T\bC^{-\alpha}(\rho)}.
	\end{align*}
	For $I_3$, by \eqref{E4} and the change of variable, we have
	\begin{align*}
	\|I_3\|_{L^\infty(\rho)}&\lesssim\lambda^{\frac{\alpha}{2}-\frac{1}{p}}
	\left(\int_0^{\lambda(t_2-t_1)}\e^{-sp}s^{-\frac{\alpha p}{2}}\dif s\right)^{\frac{1}{p}}\|f\|_{L^q_T\bC^{-\alpha}(\rho)}\\
	&\lesssim (t_2-t_1)^{\frac{\theta-\alpha}{2}}\lambda^{-1+\frac{\theta}{2}+\frac{1}{q}}\|f\|_{L^q_T\bC^{-\alpha}(\rho)} ,
	\end{align*}
	where we used $\e^{-sp}s^{-\frac{\alpha p}{2}}\leq s^{\frac{(\theta-\alpha)p}{2}-1}$ for all $s>0$.
	Therefore,
	\begin{align}\label{EG11}
	\|\sI_\lambda f\|_{C^{(\theta-\alpha)/2}_TL^\infty(\rho)}\lesssim_C
	(\lambda\vee 1)^{\frac{\theta}{2}+\frac{1}{q}-1}\|f\|_{L^q_T\bC^{-\alpha}(\rho)},
	\end{align}
which together with \eqref{EG01} yields \eqref{EG1}.
\end{proof}

\subsection{Paracontrolled calculus}
\label{ssec:para}

In this subsection we recall some basic ingredients in the paracontrolled calculus developed by Bony \cite{Bon81} and \cite{GIP15}. 
The first important fact is that the product $fg$ of two distributions $f\in \bC^\alpha$ and $g\in \bC^\beta$ is well defined 
if and only if $\alpha+\beta>0$. In terms of Littlewood-Paley's block operator $\Delta_j$, 
the product $fg$ of two distributions $f$ and $g$ can be formally decomposed as
$$
fg=f\prec g+f\circ g+f\succ g,
$$
where 
$$
f\prec g=g\succ f:=\sum_{j\geq-1}\sum_{i<j-1}\Delta_if\Delta_jg, \quad f\circ g:=\sum_{|i-j|\leq1}\Delta_if\Delta_jg.
$$

In the following we collect some important estimates from \cite{GH18} about the paraproducts in weighted Besov spaces, 
that will be used below.

\begin{lemma}\label{lem:para} 
	Let $\rho_{1},\rho_{2}$ be two admissible weights. We have for any  $\beta\in\R$,
	\begin{equation}\label{GZ0}
	\|f\prec g\|_{\bC^\beta(\rho_{1}\rho_{2})}\lesssim\|f\|_{L^\infty(\rho_{1})}\|g\|_{\bC^{\beta}(\rho_{2})},
	\end{equation}
	and for any $\alpha<0$ and $\beta\in\mR$,
	\begin{equation}\label{GZ1}
	\|f\prec g\|_{\bC^{\alpha+\beta}(\rho_{1}\rho_{2})}\lesssim\|f\|_{\bC^{\alpha}(\rho_{1})}\|g\|_{\bC^{\beta}(\rho_{2})}.
	\end{equation}
	Moreover, for any $\alpha,\beta\in\mR$ with  $\alpha+\beta>0$,
	\begin{equation}\label{GZ2}
	\|f\circ g\|_{\bC^{\alpha+\beta}(\rho_{1}\rho_{2})}\lesssim\|f\|_{\bC^{\alpha}(\rho_{1})}\|g\|_{\bC^{\beta}(\rho_{2})}.
	\end{equation}
	In particular, if $\alpha+\beta>0$, then
	\begin{equation}\label{GZ3}
	\|f g\|_{\bC^{\alpha\wedge\beta}(\rho_{1}\rho_{2})}\lesssim\|f\|_{\bC^{\alpha}(\rho_{1})}\|g\|_{\bC^{\beta}(\rho_{2})}.
	\end{equation}
\end{lemma}
\begin{proof}
See \cite[Lemma 2.14]{GH18}.
\end{proof}

\begin{lemma}\label{lem:com2}
	Let $\rho_{1}, \rho_{2}, \rho_{3}$ be three admissible weights. For any $\alpha\in (0,1)$ and $\beta,\gamma\in \R$ with 
	$\alpha+\beta+\gamma>0$ and $\beta+\gamma<0$, there exists a bounded trilinear operator $\mathrm{com}$ 
	on $\bC^\alpha(\rho_{1})\times \bC^\beta(\rho_{2})\times \bC^\gamma(\rho_{3})$ such that
	\begin{align}\label{FA1}
	\|\mathrm{com}(f,g,h)\|_{\bC^{\alpha+\beta+\gamma}(\rho_{1}\rho_{2}\rho_{3})}\lesssim 
	\|f\|_{\bC^\alpha(\rho_{1})}\|g\|_{\bC^\beta(\rho_{2})}\|h\|_{\bC^\gamma(\rho_{3})},
	\end{align}
	where
	$$
	\mathrm{com}(f,g,h):=(f\prec g)\circ h - f(g\circ h).
	$$
\end{lemma}
\begin{proof}
	See \cite[Lemma 2.16]{GH18}.
\end{proof}
Moreover, we will make use of the time-mollified paraproducts as introduced in \cite[Section 5]{GIP15}. 
Let $Q:\R\to\R_{+}$ be a smooth function with support in $[-1,1]$ and $\int_{\R}Q(s)\mathrm{d}s=1$. 
For $T>0$ and $j\geq -1$, we define an operator $Q_{j}: L^\infty_T\bC^{\alpha}(\rho)\to L^\infty_T\bC^{\alpha}(\rho)$ by
$$
Q_{j}f(t):=\int_{\R}2^{2j}Q(2^{2j}(t-s))f( (s\wedge T)\vee 0)\mathrm{d} s,
$$
and the modified paraproduct of $f,g\in L^\infty_T\bC^{\alpha}(\rho)$ by
$$
f\Prec g := \sum_{j\geq -1}(S_{j-1}Q_{j}f)\Delta_{j} g\ \mbox{ with } S_jf=\sum_{i\leq j-1}\Delta_if.
$$
Note that for $\alpha\leq 0$, $\beta\in\mR$ and $\rho_1,\rho_2\in\sW$,
\begin{equation}\label{GZ11}
	\|f\Prec g\|_{L^\infty_T\bC^{\alpha+\beta}(\rho_{1}\rho_{2})}\lesssim\|f\|_{L^\infty_T\bC^{\alpha}(\rho_{1})}\|g\|_{L^\infty_T\bC^{\beta}(\rho_{2})}.
\end{equation}

\begin{lemma}\label{lem:5.1}
	Let $\rho_{1},\rho_{2}$ be two admissible  weights. For any $\alpha\in (0,1)$ and $\beta\in \R$, there is a constant
	$C=C(\rho_1,\rho_2,d,\alpha,\beta)>0$ such that for all $\lambda\geq 0$ and $T>0$,
	\begin{align}
	\big\|[\sL_\lambda,f\Prec] g\big\|_{L^\infty_T\bC^{\alpha+\beta-2}(\rho_{1}\rho_{2})}&\lesssim_C  \|f\|_{\mS^{\alpha}_T(\rho_{1})}\|g\|_{L^\infty_T\bC^{\beta}(\rho_{2})},\label{GA3}
	\end{align}
	and
	\begin{align}\label{GA4}
	\|f\prec g-f\Prec g\|_{L^\infty_T\bC^{\alpha+\beta}(\rho_{1}\rho_{2})}\lesssim_C 
	\|f\|_{C^{\alpha/2}_TL^\infty(\rho_{1})}\|g\|_{L^\infty_T\bC^{\beta}(\rho_{2})}.
	\end{align}
Moreover, for any $\eps>0$, we also have for some $C=C(\eps,\rho_1,\rho_2,d,\alpha,\beta)$,
	\begin{align}
	\|[\nabla \sI_\lambda, f \prec ]g\|_{L^\infty_T\bC^{\alpha+\beta+1-\eps}(\rho_1\rho_2)}
	&\lesssim_C\|f\|_{\mS^{\alpha}_T(\rho_1)}\|g\|_{L^\infty_T\bC^\beta(\rho_2)}.\label{GA44}
	\end{align}
\end{lemma}
\begin{proof}
	The estimates \eqref{GA3} and \eqref{GA4} can be found in \cite[Lemma 2.17]{GH18}. We only prove \eqref{GA44}.
	Without loss of generality, we assume $\lambda=0$.
	Recalling $\sI f(t)=\int^t_0 P_{t-s} f(s)\dif s$ and by definition, we have
	\begin{align*}
	[\nabla \sI, f \prec ]g(t)&=\int^t_0P_{t-s}\nabla(f(s)\prec g(s))\dif s-f(t)\prec\int^t_0\nabla P_{t-s} g(s)\dif s
	\\&=\int_0^t P_{t-s} (\nabla f(s)\prec g(s))\dif s+\int_0^t[P_{t-s},f(s)\prec] \nabla g(s) \dif s
	\\&+\int_0^t(f(s)-f(t))\prec P_{t-s}\nabla g(s)\dif s=:I_1(t)+I_2(t)+I_3(t).
	\end{align*}
For $I_1$, by \eqref{EG01} with $\theta=2$ and $q=\infty$ and \eqref{GZ1}, we have
	\begin{align*}
	\|I_1\|_{L^\infty_T\bC^{\alpha+\beta+1}(\rho_1\rho_2)}
	&\lesssim \|\nabla f\prec g\|_{L^\infty_T\bC^{\alpha+\beta-1}(\rho_1\rho_2)}
	\lesssim \|f\|_{L^\infty_T\bC^\alpha(\rho_1)}\|g\|_{L^\infty_T\bC^\beta(\rho_2)}.
	\end{align*}
	For $I_2$, by a modification of \cite[Lemma A.1]{CC18} we have
	\begin{align*}
	\|I_2(t)\|_{\bC^{\alpha+\beta+1-\eps}(\rho_1\rho_2)}
	&\lesssim \int_0^t (t-s)^{-1+\frac{\eps}{2}}\|f(s)\|_{\bC^\alpha(\rho_1)}\|g(s)\|_{\bC^\beta(\rho_2)}\dif s
	\\&\lesssim\|f\|_{L^\infty_T\bC^\alpha(\rho_1)}\|g\|_{L^\infty_T\bC^\beta(\rho_2)}.
	\end{align*}
	For $I_3$, by \eqref{GZ0} and \eqref{E1} we have
	\begin{align*}
       \|I_3(t)\|_{\bC^{\alpha+\beta+1-\eps}(\rho_1\rho_2)}
       &\lesssim\int_0^t\|f(s)-f(t)\|_{L^\infty(\rho_1)}\|\nabla P_{t-s}g(s)\|_{\bC^{\alpha+\beta+1-\eps}(\rho_2)}\dif s
\\&       \lesssim \|f\|_{C_T^{\alpha/2}L^\infty(\rho_1)}\|g\|_{L^\infty_T\bC^\beta(\rho_2)}\int^t_0(t-s)^{-1+\frac{\eps}{2}}\dif s.
	\end{align*}
	The proof is complete.
\end{proof}

Finally we recall the localization operators from \cite{GH18}. Let  $\sum_{k\geq-1} w_k = 1$ be a smooth dyadic  partition of unity on $\R^d$,
where $w_{-1}$ is supported in a ball containing zero and each $w_k$ for $k\geq0$ is supported on the annulus of size $2^k$.
Let $(v_{m})_{m\geq -1}$ be a smooth dyadic partition of unity on $[0,\infty)$ such that  $v_{-1}$ is 
supported in a ball containing zero and each $v_m$ for $m\geq0$ is supported on the annulus of size $2^m$.
For a given sequence $(L_{k,m})_{k,m\geq -1}$  we define localization operators $\VV_{>},
\VV_{\leqslant}$  as in \cite{GH18}
\begin{equation}\label{dec}
\VV_{>} f = \sum_{k,m}w_k v_{m}\sum_{j>L_{k,m}}\Delta_j f, \qquad
\VV_{\leqslant} f = \sum_{k,m} w_k v_{m}\sum_{j\leq L_{k,m}}\Delta_j f.
\end{equation}

\begin{lemma}\label{lem:local2}
	Let $\rho$ be an admissible weight. For given $L>0, T>0$,
	there exists a (universal) choice of parameters $(L_{k,m})_{k,m\geq -1}$ such that for all $\alpha,\beta,\kappa\in\mR$ with 
	$\alpha+\kappa>0$\footnote{Here the condition is slightly different from \cite[Lemma 2.6]{GH18}, but the proof follows along the same line.}, $\delta>0$ and $0\leq t\leq T$,
	$$ 
	\| \VV_{>} f \|_{L^\infty_T\bC^{- \alpha - \delta} (\rho^{\beta-\delta})} \lesssim 2^{- \delta L} \| f \|_{L^\infty_T\bC^{- \alpha}(\rho^{\beta})},
	$$
	$$
	\| \VV_{\leqslant} f \|_{L^\infty_T\bC^{\kappa} (\rho^{\alpha+\beta+\kappa})} \lesssim 2^{(\alpha + \kappa)L} \| f \|_{L^\infty_T\bC^{- \alpha} (\rho^{\beta})}, 
	$$
	where the proportional constant depends on $\alpha,\beta,\delta,\kappa$ but is independent of $f$.
\end{lemma}
\begin{proof}
	See \cite[Lemma 2.6]{GH18}.
\end{proof}


{
\subsection{Renormalized pairs}\label{sec:2.5}

In this subsection we introduce the  renormalized pairs, which is one important part in Gubinelli-Imkeller-Perkowski's 
pracontrolled theory.
Fix $\alpha\in(\frac12,\frac23)$ and an admissible weight $\rho\in\sW$. For $T>0$, let $b=(b_1,\cdots, b_d)$ and $f$ be $(d+1)$-distributions in $L^\infty_T\bC^{-\alpha}(\rho)$.
First of all, we introduce two important quantities for later use
\begin{align}\label{AAb}
\ell^b_T(\rho):=\sup_{\lambda\geq 0}\|b\circ\nabla\sI_\lambda b\|_{L^\infty_T\bC^{1-2\alpha}(\rho^2)}+\|b\|^2_{L_T^\infty\bC^{-\alpha}(\rho)}+1,
\end{align}
and for $q\in[1,\infty]$,
\begin{align}\label{AA9}
\mA^{b,f}_{T,q}(\rho):=\sup_{\lambda\geq 0}\|b\circ\nabla\sI_\lambda f\|_{L^q_T\bC^{1-2\alpha}(\rho^2)}+\|b\|_{L_T^\infty\bC^{-\alpha}(\rho)}\|f\|_{L^q_T\bC^{-\alpha}(\rho)}.
\end{align}
By \eqref{GZ2}, except for $\alpha<\frac12$, in general, $b(t)\circ\nabla\sI_\lambda f(t)$ is not well-defined since by Schauder's estimate,
we only  have (see Lemma \ref{Le11})
$$
\nabla\sI_\lambda f\in L^\infty_T\bC^{1-\alpha}(\rho).
$$ 
However, in the probabilistic sense, it is possible to give a meaning for $b\circ\nabla\sI_\lambda f$ when $b, f$ belong to 
the chaos of Gaussian noise (see Section \ref{sec:kpz} below). This motivates us to introduce the following notion.
\bd\label{Def216}
We call the above $(b,f)\in L^\infty_T\bC^{-\alpha}(\rho)$ a renormalized pair if there exist $b_n, f_n\in L^\infty_T\sC^\infty(\rho)$ 
with $\sup_{n\in\mN}\big(\ell^{b_n}_T(\rho)+\mA^{b_n,f_n}_{T,\infty}(\rho)\big)<\infty$ and such that 
$(b_n,f_n)$ converges to $(b,f)$ in $L_T^\infty\bC^{-\alpha}(\rho)$, and  for each $\lambda\geq 0$, there are functions $g_\lambda, h_\lambda\in L_T^\infty\bC^{1-2\alpha}(\rho^2)$ such that
\begin{align}\label{Lim1}
\lim_{n\to\infty}\|b_n\circ\nabla\sI_\lambda f_n-g_\lambda\|_{L_T^\infty\bC^{1-2\alpha}(\rho^2)}=0
\end{align}
and
\begin{align}\label{Lim2}
\lim_{n\to\infty}\|b_n\circ\nabla\sI_\lambda b_n-h_\lambda\|_{L_T^\infty\bC^{1-2\alpha}(\rho^2)}=0.
\end{align}
For notational convenience, we shall write
$$
g_\lambda=:b\circ\nabla\sI_\lambda f,\ h_\lambda=:b\circ\nabla\sI_\lambda b.
$$ 
The set of all the above renormalized pair is denoted by $\mB^\alpha_T(\rho)$.
\ed

\br\rm
(i) Let $b\in \mL^\infty_T(\rho)$ and $f\in L^\infty_T\bC^{-\alpha}(\rho)$. Let $b_n(t,x):=b(t,\cdot)*\rho_n(x)$ and $f_n(t,x):=f(t,\cdot)*\rho_n(x)$ be the mollifying approximation.
By definition and \eqref{GZ2}, it is easy to see that $(b,f)\in\mB^\alpha_T(\rho)$.
Moreover, if $(b,f)\in\mB^\alpha_T(\rho)$ and $b'\in\mL^\infty_T(\rho)$, then $(b+b', f)\in\mB^\alpha_T(\rho)$. 

(ii) To make the convergence hold in \eqref{Lim1} and \eqref{Lim2}, we may need to subtract some terms containing renormalization constants in the approximation $b_n\circ\nabla \sI_\lambda f_n$ and $b_n\circ\nabla \sI_\lambda b_n$. In Definition \ref{Def216}, we suppose the renormalization constants are zero for simplicity, since in application we can choose symmetric mollifiers for approximation, 
which makes the renormalization constant disappear.  In general we only use the uniform bounds 
$\sup_{n\in\mN}\big(\ell^{b_n}_T(\rho)+\mA^{b_n,f_n}_{T,\infty}(\rho)\big)<\infty$ and the convergence \eqref{Lim1}, 
\eqref{Lim2} and the renormalization constants do not affect our analysis and calculations. 
\er
To eliminate the parameter $\lambda$ in \eqref{Lim1} and \eqref{Lim2}, the following lemma is useful.
\bl\label{lem:lambda} Let $\sI^{t}_s(f)=\int_s^t P_{t-r}f(r)\dif r$. For any $t>0$, we have
\begin{equation}\label{I}
\sup_{\lambda\geq 0}\|b(t)\circ\nabla \sI_\lambda f(t)\|_{\bC^{1-2\alpha}(\rho)}\leq 2\sup_{s\in [0,t]}\| b(t)\circ \nabla \sI^{t}_s(f)\|_{\bC^{1-2\alpha}(\rho)}.
\end{equation}
\el
\begin{proof}
	Note that by integration by parts formula,
	\begin{align*}
	&\int_0^t e^{-\lambda(t-s)}P_{t-s}f(s)\dif s
	=\int^t_0P_{t-s}f(s)\dif s-\lambda\int^t_0\e^{-\lambda(t-s)}\int^s_0P_{t-r}f(r)\dif r\dif s\\
	&\qquad=\e^{-\lambda t}\int^t_0P_{t-s}f(s)\dif s+\lambda\int^t_0\e^{-\lambda(t-s)}\int^t_{s}P_{t-r}f(r)\dif r\dif s.
	\end{align*}
	Thus,
	\begin{align*}
	&b(t)\circ\nabla \sI_\lambda f(t)=
	\e^{-\lambda t}b(t)\circ\nabla \sI^t_0f+\lambda\int^t_0\e^{-\lambda(t-s)}b(t)\circ\nabla \sI^t_s(f)\dif s.
	\end{align*} From this we get the desired estimate.
\end{proof}

The following localized property about the operation $\circ$ is also useful.
\bl\label{le:loc}
Let $T>0$, $\rho,\bar\rho\in\sW$, $\eps\in(0,1)$ and $\alpha\in(\frac12,\frac23)$.
Suppose that
$$
\phi\in \bC^{\alpha+\eps}(\bar\rho\rho^{-2}),\  \psi\in\mS^{\alpha+\eps}_T,\ (b,f)\in\mB^\alpha_T(\rho).
$$ 
Then there is a constant $C>0$ depending only on $T, \eps,\alpha,d,\rho,\bar\rho$ 
such that for all $\lambda\geq 0$ and $t\in[0,T]$,
\begin{align}
&\|((b\phi)\circ\nabla\sI_\lambda (f\psi))(t)\|_{\bC^{1-2\alpha}(\bar\rho)}
\lesssim_C\|\phi\|_{\bC^{\alpha+\eps}(\bar\rho\rho^{-2})}\|\psi\|_{\mS^{\alpha+\eps}_t}\mA^{b,f}_{t,\infty}(\rho).\label{EK0}
\end{align}
\el
\begin{proof}
We only prove the estimate \eqref{EK0}. For simplicity, we drop the time variable. By using paraproduct, we have
	\begin{align*}
	(b\phi)\circ\nabla \sI_\lambda (f\psi)
	&=(b\phi)\circ\nabla \sI_\lambda (\psi\succcurlyeq f)+(b\phi)\circ\nabla \sI_\lambda (\psi\prec f)
	\\&=(b\phi)\circ\nabla \sI_\lambda (\psi\succcurlyeq f)+(b\phi)\circ[\nabla \sI_\lambda ,\psi\prec]f
	\\&\quad+\textrm{com}(\psi,\nabla \sI_\lambda f,b\phi)+\psi((b\phi)\circ\nabla \sI_\lambda f)
	\\&=(b\phi)\circ\nabla \sI_\lambda (\psi\succcurlyeq f)+(b\phi)\circ[\nabla \sI_\lambda ,\psi\prec]f
	\\&\quad+\textrm{com}(\psi,\nabla \sI_\lambda f,b\phi)+\psi((\phi\succcurlyeq b)\circ\nabla \sI_\lambda f)
	\\&\quad+\psi\textrm{com}(\phi,b,\nabla \sI_\lambda f)+\psi\phi( b\circ\nabla \sI_\lambda f).
	\end{align*}
	Let $\eps>0$ being small enough.  We estimate each term as following.
	\begin{enumerate}[$\bullet$]
\item By \eqref{GZ2}, \eqref{EG1} and \eqref{GZ1}, we have
\begin{align*}
\|(b\phi)\circ\nabla \sI_\lambda (\psi\succcurlyeq f)\|_{\bC^0(\bar\rho)}
&\lesssim\|b\phi\|_{\bC^{-\alpha}(\bar\rho\rho^{-1})}\|\nabla \sI_\lambda (\psi\succcurlyeq f)\|_{L^\infty_t\bC^{\alpha+\eps}(\rho)}
\\&\lesssim\|b\phi\|_{\bC^{-\alpha}(\bar\rho\rho^{-1})}\|\psi\succ f+\psi\circ f\|_{L^\infty_t\bC^{\alpha-1+\eps}(\rho)}
\\&\lesssim\|\phi\|_{\bC^{\alpha+\eps}(\bar\rho\rho^{-2})}\|b\|_{\bC^{-\alpha}(\rho)}\|f\|_{L^\infty_t\bC^{-\alpha}(\rho)}\|\psi\|_{L^\infty_t\bC^{\alpha+\eps}}.
	\end{align*}
\item By \eqref{GZ2}, \eqref{GZ3} and \eqref{GA44}, we have
	\begin{align*}
	\|(b\phi)\circ[\nabla \sI_\lambda ,\psi\prec]f\|_{\bC^0(\bar\rho)}
&\lesssim\|b\phi\|_{\bC^{-\alpha}(\bar\rho\rho^{-1})}\|[\nabla \sI_\lambda ,\psi\prec]f\|_{L_t^\infty\bC^{\alpha+\eps}(\rho)}
\\&\lesssim\|\phi\|_{\bC^{\alpha+\eps}(\bar\rho\rho^{-2})}\|b\|_{\bC^{-\alpha}(\rho)}\|\psi\|_{\mS_t^{2\alpha-1+2\eps}}\|f\|_{L_t^\infty\bC^{-\alpha}(\rho)}.
	\end{align*}
\item By \eqref{FA1}, \eqref{EG1} and \eqref{GZ3}, we have
	\begin{align*}
\|\textrm{com}(\psi,\nabla \sI_\lambda f,b\phi)\|_{\bC^0(\bar\rho)}
&\lesssim\|\psi\|_{\bC^{2\alpha-1+\eps}}\|\nabla \sI_\lambda f\|_{L^\infty_t\bC^{1-\alpha}(\rho)}\|b\phi\|_{\bC^{-\alpha}(\bar\rho \rho^{-1})}
\\&\lesssim\|\psi\|_{\bC^{2\alpha-1+\eps}}\|f\|_{L^\infty_t\bC^{-\alpha}(\rho)}\|b\|_{\bC^{-\alpha}(\rho)}\|\phi\|_{\bC^{\alpha+\eps}(\bar\rho\rho^{-2})}.
	\end{align*}
\item  By \eqref{GZ3}, \eqref{GZ2}, \eqref{EG1} and \eqref{GZ1}, we have
	\begin{align*}
\|\psi((\phi\succcurlyeq b)\circ\nabla \sI_\lambda f)\|_{\bC^0(\bar\rho)}
&\lesssim\|\psi\|_{L^\infty}\|\phi\succcurlyeq b\|_{\bC^{\alpha-1+\eps}(\bar\rho\rho^{-1})}\|\nabla \sI_\lambda f\|_{\bC^{1-\alpha}(\rho)}
	\\&\lesssim\|\psi\|_{L^\infty}\|\phi\|_{\bC^{\alpha+\eps}(\bar\rho\rho^{-2})}\|b\|_{\bC^{-\alpha}(\rho)}\|f\|_{L_t^\infty\bC^{-\alpha}(\rho)}.
	\end{align*}
\item  By \eqref{GZ3} and \eqref{FA1}, we have
	\begin{align*}
\|\psi\textrm{com}(\phi,b,\nabla \sI_\lambda f)\|_{\bC^0(\bar\rho)}
\lesssim\|\psi\|_{L^\infty}\|\phi\|_{\bC^{2\alpha-1+\eps}(\bar\rho\rho^{-2})}\|b\|_{\bC^{-\alpha}(\rho)}\|f\|_{L^\infty_t\bC^{-\alpha}(\rho)}.
	\end{align*}
\item  By \eqref{GZ3}, we have
	\begin{align*}
\|\psi\phi( b\circ\nabla \sI_\lambda f)\|_{\bC^{1-2\alpha}(\bar\rho)}
\lesssim\|\psi\phi\|_{\bC^{2\alpha-1+\eps}(\bar\rho\rho^{-2})}\|b\circ\nabla \sI_\lambda f\|_{\bC^{1-2\alpha}(\rho^2)}.
	\end{align*}
\end{enumerate}
	Combining the above calculations, we obtain the desired estimate. 
\end{proof}
}
\section{A study of linear parabolic equation in weighted H\"older spaces}\label{linear}

In this section we consider the following linear parabolic equation:
\begin{align}\label{PDE7}
\sL_\lambda u=(\p_t-\Delta+\lambda) u=b\cdot\nabla u+f,\quad u(0)=u_0,
\end{align}
where $\lambda\geq 0$, $b=(b_1,\cdots, b_d)$ is a vector-valued distribution and $f$ is a scalar-valued distribution. 
Suppose that for some $\alpha\in(\frac{1}{2},\frac{2}{3})$ and admissible weight $\rho\in\sW$,
\begin{align}\label{DW9}
(b, f)\in  \mB^\alpha_T(\rho),\ \ T>0.
\end{align}
The aim of this section is to show the well-posedness of PDE \eqref{PDE7} under \eqref{DW9}. We first give the definition of the paracontrolled solutions to \eqref{PDE7}. We then establish the Schauder estimate with the coefficient in unweighted Besov space  by choosing $\lambda$ large enough. Then by a classical maximum principle, we obtain the Schauder estimate for \eqref{PDE7} depending polynomially on the coefficient.  In Section \ref{sec:3.3} we establish global well-posedness of equation \eqref{PDE7} under \eqref{DW9} and  obtain a uniform estimate of solution to \eqref{PDE7} in Besov spaces 
with sublinear weights.

\subsection{Paracontrolled solutions}
To introduce the paracontrolled solution of PDE \eqref{PDE7}, by Bony's decomposition,
we make the following paracontrolled ansatz as in \cite{GIP15}:
\begin{align}\label{DT11}
u=
\nabla u\Prec  \sI_\lambda b+u^\sharp+\sI_\lambda f,
\end{align}
where $u^\sharp$ solves the following PDE in weak sense
\begin{align}\label{DT110}
\sL_\lambda u^\sharp=&\nabla u\prec b-\nabla u\Prec b+\nabla u\succ b+b\circ\nabla u-[\sL_\lambda, \nabla u\Prec]\sI_\lambda b,
\\u^\sharp(0)=&u_0.\no
\end{align}
Note that $b\circ\nabla u$ does not make a sense, whose meaning is given as follows:
By \eqref{DT11}, we can write
\begin{align}
b\circ\nabla u&=b\circ\nabla(\nabla u\Prec  \sI_\lambda b)+b\circ\nabla u^\sharp+b\circ\nabla\sI_\lambda f\no\\
&=b\circ\nabla(\nabla u\prec \sI_\lambda b)+\textrm{com}_1+b\circ\nabla u^\sharp+b\circ\nabla\sI_\lambda f\no\\
&=b\circ(\nabla^2 u\prec \sI_\lambda b)+(b\circ\nabla \sI_\lambda b)\cdot \nabla u+\textrm{com}\no\\
&\quad+\textrm{com}_1+b\circ\nabla u^\sharp+b\circ\nabla\sI_\lambda f,\label{FQ2}
\end{align}
where
$$\textrm{com}_1:=b\circ\nabla [\nabla u\Prec \sI_\lambda b-\nabla u\prec \sI_\lambda b]$$
and
$$
\textrm{com}:=\mathrm{com}(\nabla u, \nabla \sI_\lambda b,b).
$$
\begin{definition}\label{def:para1} 
Let $\rho,\bar\rho\in\sW$ be two bounded admissible weights and $\eps\geq 0$.
For given $(b,f)\in\mB^\alpha_T(\rho)$, with notation \eqref{SS0}, a pair of functions 
\begin{align}\label{Reg}
(u, u^\sharp)\in \mS^{2-\alpha}_T(\bar\rho)\times \mS_T^{3-2\alpha}(\rho^{2+\eps}\bar\rho)
\end{align}
is called a paracontrolled solution of PDE \eqref{PDE7}  if
$(u,u^\sharp)$ satisfies \eqref{DT11} and \eqref{DT110} with $b\circ\nabla u$ given by \eqref{FQ2},
in the analytic weak sense.
\end{definition}

\br\rm
Under \eqref{Reg},
from the proof of Lemma \ref{Le32} below, 
each term in \eqref{FQ2} is well-defined.
Moreover, for $b, f\in L^\infty_T\sC^2(\rho)$ with $\rho(x)=\<x\>^{-1}$, it is well known
that PDE \eqref{PDE7} has a unique classical solution.
From Definition \ref{def:para1}, it is not hard to see that classical solutions are paracontrolled solutions. 
\er

The following lemma makes the above definition more transparent.

\bl\label{Le32}
Let $T, \eps\geq 0$ and $(u,u^\sharp)$ be a paracontrolled solution of \eqref{PDE7} in the sense of Definition \ref{def:para1}. 
For any $\gamma,\beta\in(\alpha,2-2\alpha]$, there is a constant $C>0$ depending only on $T,\eps,\alpha,\gamma,\beta,d, \rho,\bar\rho$ 
such that for all $\lambda\geq 0$ and $t\in[0,T]$,
\begin{align}
\|(b\circ\nabla u)(t)\|_{\bC^{1-2\alpha}(\rho^{2+\eps}\bar\rho)}&\lesssim_C 
\ell^b_t(\rho)\|u\|_{\mS^{\alpha+\gamma}_t(\bar\rho)}
+\sqrt{\ell^b_t(\rho)}\|u^\sharp(t)\|_{\bC^{\beta+1}(\rho^{1+\eps}\bar\rho)}\no\\
&\quad+\|(b\circ\nabla\sI_\lambda f)(t)\|_{\bC^{1-2\alpha}(\rho^{2+\eps}\bar\rho)}.\label{BV1}
\end{align}
\el
\begin{proof}
	 Below we drop the time variable $t$ and fix 
	$$
	\gamma,\beta\in(\alpha,2-2\alpha].
	$$
	Recall $1-2\alpha<0$. We now estimate each term in \eqref{FQ2} as following.
	\begin{enumerate}[$\bullet$]
		\item Since $\gamma>\alpha$, by \eqref{GZ1}, \eqref{GZ2} and \eqref{EG1}, we have
		\begin{align*}
		\|b\circ(\nabla^2 u\prec \sI_\lambda b)\|_{\bC^{1-2\alpha}(\rho^2\bar\rho)}
		&\lesssim
		\|b\|_{\bC^{-\alpha}(\rho)}\|\nabla^2 u\prec \sI_\lambda b\|_{\bC^{\gamma}(\rho\bar\rho)}\\
		&\lesssim  \|b\|_{\bC^{-\alpha}(\rho)}\|\nabla^2 u\|_{\bC^{\gamma+\alpha-2}(\bar\rho)}
		\|\sI_\lambda b\|_{\bC^{2-\alpha}(\rho)}\\
		&\lesssim  \|b\|^2_{L^\infty_t\bC^{-\alpha}(\rho)}\|u\|_{\bC^{\gamma+\alpha}(\bar\rho)}
		\lesssim \ell^b_t(\rho)\|u\|_{\bC^{\alpha+\gamma}(\bar\rho)}.
		\end{align*}
		\item By \eqref{GZ3}, we have
		\begin{align*}
		\|\nabla u(b\circ\nabla \sI_\lambda b)\|_{\bC^{1-2\alpha}(\rho^2\bar\rho)}
		&\lesssim  \|\nabla u\|_{\bC^{\gamma+\alpha-1}(\bar\rho)}\|b\circ\nabla \sI_\lambda b\|_{\bC^{1-2\alpha}(\rho^2)}\\
		&\lesssim \ell^b_t(\rho)\|u\|_{\bC^{\alpha+\gamma}(\bar\rho)}.
		\end{align*}
		\item Since 
		$\gamma>\alpha$, by \eqref{FA1} and \eqref{EG1}, we have
		\begin{align*}
		\|{\rm com}\|_{\bC^{1-2\alpha}(\rho^2\bar\rho)}&\lesssim \|b\|_{\bC^{-\alpha}(\rho)}
		\|\nabla u\|_{\bC^{\gamma+\alpha-1}(\bar\rho)}\|\nabla \sI_\lambda b\|_{\bC^{1-\alpha}(\rho)}\\
		&\lesssim\|b\|^2_{L^\infty_t\bC^{-\alpha}(\rho)}\|u\|_{\bC^{\gamma+\alpha}(\bar\rho)}
		\lesssim \ell^b_t(\rho)\|u\|_{\bC^{\alpha+\gamma}(\bar\rho)}.
		\end{align*}
		\item By Lemma \ref{lem:para}, \eqref{*} \eqref{GA4} and \eqref{EG1}, we have
		\begin{align*}
		\|\textrm{com}_1\|_{\bC^{1-2\alpha}(\rho^2\bar\rho)}
		&\lesssim \|b\|_{\bC^{-\alpha}(\rho)}
		\|\nabla u\Prec \sI_\lambda b-\nabla u\prec \sI_\lambda b\|_{\bC^{\gamma+1}(\rho\bar\rho)}\\
		&\lesssim \|b\|_{\bC^{-\alpha}(\rho)}
		\|\nabla u\|_{C^{(\gamma+\alpha-1)/2}_tL^\infty(\bar\rho)}\|\sI_\lambda b\|_{L^\infty_t\bC^{2-\alpha}(\rho)}\\
		&\lesssim \|b\|_{L^\infty_t\bC^{-\alpha}(\rho)}^2\|u\|_{\mS^{\alpha+\gamma}_t(\bar\rho)}
		\lesssim \ell^b_t(\rho)\|u\|_{\mS^{\alpha+\gamma}_t(\bar\rho)}.
		\end{align*}
		\item Since $\beta>\alpha$, by \eqref{GZ2}, we have
		\begin{align*}
		\|b\circ\nabla u^\sharp\|_{L^\infty(\rho^{2+\eps}\bar\rho)}
		&\lesssim\|b\|_{\bC^{-\alpha}(\rho)}\|\nabla u^\sharp\|_{\bC^{\beta}(\rho^{1+\eps}\bar\rho)}
		\leq\sqrt{\ell^b_t(\rho)}\|u^\sharp\|_{\bC^{\beta+1}(\rho^{1+\eps}\bar\rho)}.
		\end{align*}
	\end{enumerate}
	Combining the above calculations and by \eqref{FQ2}, we obtain the estimate. 	
\end{proof}


\subsection{Schauder's estimate for paracontrolled solutions without weights}
\label{sub:Schauder}

In this section we assume $(b,f)\in\mB^\alpha_T$, and for simplicity, we shall write
$$
\ell^b_T=\ell^b_T(1),\ \ \mA^{b,f}_{T,q}=\mA^{b,f}_{T,q}(1).
$$

\bl\label{Le11a} Assume $u_0=0$.
For any  $\theta\in (1+\tfrac{3\alpha}{2},2)$, $q\in(\frac{2}{2-\theta},\infty)$ and $T>0$, there exist constants $c_0, c_1>0$ only depending on $\theta,\alpha,d,q,T$ 
such that for all $\lambda\geq c_0(\ell^b_T)^{1/(1-\frac{\theta}{2}-\frac{1}{q})}$ and any paracontrolled solution $u_\lambda=u$ to PDE \eqref{PDE7},
\begin{align}\label{PP7}
\|u_\lambda\|_{\mS^{\theta-\alpha}_T}\leq c_1\mA^{b,f}_{T,q}.
\end{align}
Moreover,  there is a constant $c_2>0$ such that for all $\lambda\geq 0$,
\begin{align}\label{PP6}
\|u_\lambda\|_{\mS^{2-\alpha}_T}+\|u^\sharp_\lambda\|_{\mS^{3-2\alpha}_T}
\leq c_2(\ell^b_T)^{\frac{4}{2-3\alpha}}\Big(\|u_\lambda\|_{\mL_T^\infty}+\mA^{b,f}_{T,\infty}\Big).
\end{align}
\el
\begin{proof}
	Below we fix
	$$
	\theta\in(1+\tfrac{3\alpha}{2},2],\ \ q\in[\tfrac{2}{2-\theta},\infty],\ \gamma,\beta\in(\alpha, \theta-2\alpha].
	$$ 
	By \eqref{EG1}, \eqref{GZ1} and \eqref{GZ0}, we clearly have
	\begin{align}\label{GA11}
	\begin{split}
	(\lambda\vee 1)^{1-\frac{\theta}{2}-\frac{1}{q}}\|u\|_{\mS^{\theta-\alpha}_T}
	&\lesssim\|b\prec\nabla u+b\succ\nabla u+b\circ\nabla u+f\|_{L^q_T\bC^{-\alpha}}
	\\&\lesssim \|b\|_{L^\infty_T\bC^{-\alpha}}\|\nabla u\|_{L^q_TL^\infty}
	+\|b\circ\nabla u\|_{L^q_T\bC^{-\alpha}}+\mA^{b,f}_{T,q},
	\end{split}
	\end{align}
	and by Lemma \ref{lem:5.1},
	\begin{align*}
	(\lambda\vee 1)^{1-\frac{\theta}{2}-\frac{1}{q}}\|u^\sharp\|_{\mS^{\theta+\gamma-1}_T}
	&\lesssim \|\nabla u\prec b-\nabla u\Prec b\|_{L^\infty_T\bC^{\gamma-1}}+\|\nabla u\succ b\|_{L^\infty_T\bC^{\gamma-1}}\\
	&\quad+\|[\sL_\lambda, \nabla u\Prec]\sI_\lambda b\|_{L^\infty_T\bC^{\gamma-1}}+\|b\circ\nabla u\|_{L^q_T\bC^{\gamma-1}}\\
	&\lesssim \|u\|_{\mS^{\gamma+\alpha}_T}\|b\|_{L^\infty_T\bC^{-\alpha}}+\|b\circ\nabla u\|_{L^q_T\bC^{1-2\alpha}},
	\end{align*}
	where we used \eqref{*}, \eqref{GA4}, \eqref{GA44} and \eqref{GZ1} in the second inequality. 
	Moreover, by \eqref{BV1}, we also have
$$
\|b\circ\nabla u\|_{L^q_T\bC^{1-2\alpha}}\lesssim 
\ell^b_T\|u\|_{\mS^{\gamma+\alpha}_T}
+\sqrt{\ell^b_T}\|u^\sharp\|_{L^q_T\bC^{\beta+1}}+\mA^{b,f}_{T,q}.
$$
Thus, we obtain that for all $\lambda\geq 0$,
	\begin{align}\label{GA111}
	\begin{split}
	&(\lambda\vee 1)^{1-\frac{\theta}{2}-\frac{1}{q}}\left(\|u\|_{\mS^{\theta-\alpha}_T}+\|u^\sharp\|_{\mS^{\theta+\gamma-1}_T}\right)
	\\&\lesssim \ell^b_T\|u\|_{\mS^{\gamma+\alpha}_T}
	+\sqrt{\ell^b_T}\|u^\sharp\|_{L^\infty_T\bC^{\beta+1}}+\mA^{b,f}_{T,q}.
	\end{split}
	\end{align}
	In particular, letting $\gamma=\theta-2\alpha$ and $\beta=2\theta-2\alpha-2$, we get
	for some $c=c(\theta,\alpha,d,q,T)$,
	\begin{align*}
	(\lambda\vee 1)^{1-\frac{\theta}{2}-\frac{1}{q}}\left(\|u\|_{\mS^{\theta-\alpha}_T}+\|u^\sharp\|_{\mS^{2\theta-2\alpha-1}_T}\right)
	\lesssim_c\ell^b_T\Big(\|u\|_{\mS^{\theta-\alpha}_T}+\|u^\sharp\|_{\mS^{2\theta-2\alpha-1}_T}\Big)
	+\mA^{b,f}_{T,q}.
	\end{align*}
	Choosing $\lambda$ such that $\lambda^{1-\frac{\theta}{2}-\frac{1}{q}}\geq c\ell^b_T$, we obtain \eqref{PP7}.
	
	On the other hand, letting $\theta=2$ and $q=\infty$ in \eqref{GA111}, 
	we obtain that for any $\gamma,\beta\in(\alpha, 2-2\alpha]$,
	\begin{align}\label{Z1}
	\|u\|_{\mS^{2-\alpha}_T}+\|u^\sharp\|_{\mS^{1+\gamma}_T}
	&\lesssim \ell^b_T\| u\|_{\mS^{\gamma+\alpha}_T}
	+\sqrt{\ell^b_T}\|u^\sharp\|_{L^\infty_T\bC^{\beta+1}}+\mA^{b,f}_{T,\infty}.
	\end{align}
	If $\alpha<\beta<\gamma<2-2\alpha$, then by \eqref{AM1} and Young's inequality, we have for any $\eps\in(0,1)$,
	\begin{align}\label{HM1}
	\begin{split}
	\|u\|_{\mS^{2-\alpha}_T}+\|u^\sharp\|_{\mS^{1+\gamma}_T}
	&\leq \eps\Big(\|u\|_{\mS^{2-\alpha}_T}+\|u^\sharp\|_{\mS_T^{1+\gamma}}\Big)
	+C_\eps(\ell^b_T)^{\frac{2-\alpha}{2-\gamma-2\alpha}}\|u\|_{\mL^\infty_T}
	\\&\quad+C_\eps (\ell^b_T)^{\frac{1+\gamma}{2(\gamma-\beta)}}\|u^\sharp\|_{\mL^\infty_T}+C\mA^{b,f}_{T,\infty}.
	\end{split}
	\end{align}
	Note that by \eqref{DT11},
	\begin{align*}
	\|u^\sharp\|_{\mL^\infty_T}
	&=\|u-\nabla u\Prec  \sI_\lambda b-\sI_\lambda f\|_{\mL^\infty_T}
	\\&\lesssim \|u\|_{\mL^\infty_T}(1+\|b\|_{L^\infty_T\bC^{-\alpha}})+\|f\|_{L_T^\infty\bC^{-\alpha}}\lesssim \|u\|_{\mL^\infty_T}\sqrt{\ell^b_T}+\mA^{b,f}_{T,\infty}.
	\end{align*}
	Substituting it into \eqref{HM1} and taking $\eps=1/2$, we obtain
	\begin{align*}
	\|u\|_{\mS^{2-\alpha}_T}+\|u^\sharp\|_{\mS^{1+\gamma}_T}\lesssim(\ell^b_T)^{\frac{2-\alpha}{2-\gamma-2\alpha}\vee(\frac{1+\gamma}{2(\gamma-\beta)}+\frac{1}{2})}
	\Big(\|u\|_{\mL^\infty_T}+\mA^{b,f}_{T,\infty}\Big),
	\end{align*}
	which, by choosing $\gamma=2/3$ and $\beta$ close to $\alpha$, yields that
	$$
	\|u\|_{\mS^{2-\alpha}_T}+\|u^\sharp\|_{\mS^{5/3}_T}\lesssim(\ell^b_T)^{\frac{8-3\alpha}{2(2-3\alpha)}}
	\Big(\|u\|_{\mL_T^\infty}+\mA^{b,f}_{T,\infty}\Big).
	$$
	Moreover, by \eqref{Z1} with $\gamma=2-2\alpha$ and $\beta=2/3$, we get
	$$
	\|u^\sharp\|_{\mS^{3-2\alpha}_T}
	\lesssim \ell^b_T\| u\|_{\mS^{2-\alpha}_T}
	+\sqrt{\ell^b_T}\|u^\sharp\|_{\mS^{5/3}_T}+\mA^{b,f}_{T,\infty}
	\lesssim (\ell^b_T)^{\frac{4}{2-3\alpha}}\Big(\|u\|_{\mL_T^\infty}+\mA^{b,f}_{T,\infty}\Big).
	$$
	The proof is complete. 
\end{proof}

\bt\label{Th12} 
Let $T>0$ and $u_0=0$. For any $(b,f)\in\mB^\alpha_T$, there is a unique paracontrolled solution $u$ to PDE \eqref{PDE7}
in the sense of Definition \ref{def:para1}. Moreover, 
there are $q>1$ large enough only depending on $\alpha$ and $c_1,c_2>0$ such that
$$
\|u\|_{\mL^\infty_T}\leq c_1(\ell^b_T)^{\frac{5}{2-3\alpha}}\mA^{b,f}_{T,q},\ \ 
\|u\|_{\mS^{2-\alpha}_T}+\|u^\sharp\|_{\mS^{3-2\alpha}_T}\leq c_2(\ell^b_T)^{\frac{9}{2-3\alpha}}\mA^{b,f}_{T,\infty}.
$$
\et
\begin{proof}
	We first assume that
	$$
	b,f\in L^\infty_T \sC^2,\ \ \forall T>0.
	$$
	Fix $\lambda\geq 0$. For any $\lambda'>0$, it is well known that there is a unique classical solution $w$ to the following PDE:
	\begin{align}\label{EQ1}
	\p_tw=\Delta w-(\lambda'+\lambda) w+b\cdot\nabla w+f,\ \ w(0)=0.
	\end{align}
	In particular, for any $\theta\in(1+\frac{3}{2}\alpha,2)$ and $q\in(\frac{2}{2-\theta},\infty)$,
	by \eqref{PP7}, we have for $\lambda'\geq c_0(\ell^b_T)^{1/(1-\frac{\theta}{2}-\frac{1}{q})}$, 
	$$
	\|w\|_{\mL^\infty_T}\leq\|w\|_{L_T^\infty\bC^{\theta-\alpha}}\leq c_1 \cdot\mA^{b,f}_{T,q}.
	$$
	Now let $u$ be the unique classical solution to PDE \eqref{PDE7} with $u_0=0$.
	Let $\bar u=u-w$. Then $\bar u$ solves the following PDE:
	$$
	\p_t\bar u=\Delta \bar u-\lambda\bar{u}+b\cdot\nabla \bar u+\lambda' w,\ \ \bar u(0)=0.
	$$ 
	By the classical maximum principle, we have
	$$
	\|\bar u\|_{\mL^\infty_T}\leq \lambda' T\|w\|_{\mL^\infty_T}.
	$$
	Hence, by taking $\theta$ close to $1+\frac{3\alpha}{2}$ and $q$ large enough, we obtain 
	$$
	\|u\|_{\mL^\infty_T}\leq (\lambda' T+1)\|w\|_{\mL^\infty_T}\lesssim (\ell^b_T)^{1/(1-\frac{\theta}{2}-\frac1q)}\cdot\mA^{b,f}_{T,q}
	\lesssim (\ell^b_T)^{\frac{5}{2-3\alpha}}\cdot\mA^{b,f}_{T,q},
	$$
	which together with \eqref{PP6} yields
	\begin{align}\label{PP9}
	\|u\|_{\mS^{2-\alpha}_T}+\|u^\sharp\|_{\mS^{3-2\alpha}_T}\leq c_2(\ell^b_T)^{\frac{9}{2-3\alpha}}\mA^{b,f}_{T,\infty}.
	\end{align}
	{\bf (Existence)} Let $b_n$ and $f_n$ be the smoothing approximations of $b$ and $f$ in $\mB^\alpha_T$.
	We consider the following approximation equation:
	$$
	\p_t u_n=\Delta u_n-\lambda u_n+b_n\cdot\nabla u_n+f_n,\quad u_n(0)=0.
	$$
	By the assumption and \eqref{PP9}, we have the following uniform estimate:
	$$
	\sup_{n\in\mN}\Big(\|u_n\|_{\mS^{2-\alpha}_T}+\|u^\sharp_n\|_{\mS^{3-2\alpha}_T}\Big)\lesssim1.
	$$
	Using this uniform estimate and by a standard compact and weak convergence method, 
	we can show the existence of a paracontrolled solution (see \cite{GH18}).
	\\
	\\
	{\bf (Uniqueness)} Let $u_1$ and $u_2$ be two paracontrolled solution of PDE \eqref{PDE7}.
	Let $\bar u:=u_1-u_2$. Clearly, $\bar u$ is a paracontrolled solution of 
	$$
	\p_t\bar u=\Delta\bar u-\lambda\bar u+b\cdot\nabla\bar u,\ \ u(0)=0.
	$$
Let $\theta\in(1+\alpha,2)$ and $q=\frac{2}{2-\theta}$. By \eqref{EG1}, we have
\begin{align}\label{GFa}
\|\bar u\|^q_{\mS^{\theta-\alpha}_T}\leq C\int^T_0\|(b\cdot\nabla\bar u)(t)\|^q_{\bC^{-\alpha}}\dif t.
\end{align}
On the other hand, by \eqref{GZ0}, \eqref{GZ1} and Lemma \ref{Le32} we have
\begin{align*}
\|(b\cdot\nabla\bar u)(t)\|_{\bC^{-\alpha}}
&\leq \|(b\prec\nabla\bar u)(t)\|_{\bC^{-\alpha}}+\|(b\succ\nabla\bar u)(t)\|_{\bC^{-\alpha}}+\|(b\circ\nabla\bar u)(t)\|_{\bC^{-\alpha}}
\\&\lesssim \|b(t)\|_{\bC^{-\alpha}}\|\nabla\bar u(t)\|_{L^\infty}+\|(b\circ\nabla\bar u)(t)\|_{\bC^{1-2\alpha}}
\\&\lesssim \|\nabla\bar u(t)\|_{L^\infty}+
\|\bar u\|_{\mS^{2-\alpha}_t}+\|\bar u^\sharp\|_{L^\infty_t\bC^{3-2\alpha}}
\stackrel{\eqref{PP6}}{\lesssim} \|\nabla\bar u\|_{\mL^\infty_t}+\|\bar u\|_{\mL^\infty_t}.
\end{align*}
Substituting this into \eqref{GFa} and by $\theta-\alpha>1$, we obtain
\begin{align*}
\|\bar u\|^q_{L^\infty_T\bC^{\theta-\alpha}}\leq C\int^T_0\|\bar u\|^q_{L_t^\infty\bC^{\theta-\alpha}}\dif t,
\end{align*}
which in turn implies that $\bar u=0$. The uniqueness is proven. 
\end{proof}
\br\rm
The polynomial dependence on  $\ell^b_T$ in Theorem \ref{Th12} is important to establish the Schauder estimate in sublinear weighted H\"older space
since it together with a new characterization for weighted H\"older spaces in Lemma \ref{cha} below can be used to solve 
the problem of weight loss (see \cite[Remark 1.1]{PR18}).
\er
\subsection{Schauder estimate for paracontrolled solutions with weights}\label{sec:3.3}

In this section we show the well-posednness of PDE \eqref{PDE7} in weighted H\"older spaces. Recall that for $\delta\in\mR$,
$$
\rho_\delta(x):=(1+|x|^2)^{-\delta/2}=:\<x\>^{-\delta}.
$$
Now we give the main result of this section. 
\bt\label{Th33}
Let $\alpha\in(\frac{1}{2},\frac{2}{3})$ and $\vartheta:=\frac{9}{2-3\alpha}$. Choose $\kappa>0$ so that
$$
\delta:=(2\vartheta+2)\kappa\leq 1,\ \ \delta_0:=(\frac{55}{27}\vartheta+4)\kappa.
$$ 
For any $T>0$, $(b,f)\in\mB^\alpha_T(\rho_\kappa)$ and
$u_0\in\cup_{\eps>0}\bC^{1+\alpha+\eps}$, there is a unique paracontrolled solution $(u, u^\sharp)$ 
to PDE \eqref{PDE7} in the sense of Definition \ref{def:para1} with
\begin{align}\label{MN1}
\|u\|_{\mS^{2-\alpha}_T(\rho_\delta)}+\|u^\sharp\|_{\mS^{3-2\alpha}_T(\rho_{\delta_0})}
\lesssim_C\mA^{b,f}_{T,\infty}(\rho_\kappa),
\end{align}
where $C=C(T,d,\alpha,\kappa,\ell^b_T(\rho_\kappa))>0$.
\et

To prove the result we introduce the following notations.
Let $\chi\in C^\infty_c(\mR^d)$ with 
$$
\chi(x)=1,\ \ |x|\leq 1/8,\ \ \ \chi(x)=0,\  \ |x|>1/4,
$$
and for $r>0$ and $z\in\mR^d$,
$$
\chi^z_r(x):=\chi((x-z)/r),\ \ \phi^z_r(x):=\chi^z_{r(1+|z|)}(x).
$$
To show the existence of a paracontrolled solution,
we need the following simple characterization of weighted H\"older spaces.
\bl\label{cha}
Let $\alpha\geq 0$ and $r\in(0,1]$. For any $\delta,\kappa\in\mR$, there is a constant $C=C(r,\alpha,d,\delta,\kappa)>0$ such that
\begin{align}\label{GD2}
\|f\|_{\sC^\alpha(\rho_\delta\rho_\kappa)}
\asymp_C\sup_{z}\left(\rho_\delta(z)\|\phi^z_r f\|_{\sC^\alpha(\rho_\kappa)}\right).
\end{align}
\el
\begin{proof}
	Without loss of generality, we assume $\kappa=0$.
	In fact, we clearly have
	\begin{align*}
	\sup_{z}\left(\rho_\delta(z)\|\phi^z_r f\|_{\sC^\alpha(\rho_\kappa)}\right)
	\asymp\sup_{z}\left(\rho_\delta(z)\|\phi^z_r\rho_\kappa f\|_{\sC^\alpha}\right)
	\asymp\|\rho_\delta\rho_\kappa f\|_{\sC^\alpha}.
	\end{align*}
	By interpolation theorem  (see e.g. \cite[Theorem 3.11.8, Theorem 6.2.4]{BL76}), it suffices to prove \eqref{GD2} for $\alpha\in\mN_0$.
	We first consider the case $\alpha=0$. We use $B_r(z)$  to denote the ball with radius $r$ centered at $z$. 
For any $\delta\in\mR$, since for $x\in B_{(1+|z|)/2}(z)$,
	$$
	\rho^{-1}_\delta(x)\leq 2^{|\delta|}(1+|x|)^{\delta}\leq 4^{|\delta|}(1+|z|)^\delta=4^{|\delta|}\rho^{-1}_\delta(z),
	$$
	we have
	$$
	\rho_\delta(z)\phi^z_r(x) |f(x)|\leq 4^{|\delta|}\rho_\delta(x) |f(x)|\leq 4^{|\delta|}\|\rho_\delta f\|_{L^\infty}.
	$$
	Hence,
	\begin{align}\label{GQ1}
	\sup_z\left(\rho_\delta(z)\|\phi^z_rf\|_{L^\infty}\right)\leq 4^{|\delta|}\|\rho_\delta f\|_{L^\infty}=4^{|\delta|}\|f\|_{L^\infty(\rho_\delta)}.
	\end{align}
	On the other hand, since $\phi^x_r(x)=1$, we clearly have
	\begin{align}\label{GQ11}
	\|\rho_\delta f\|_{L^\infty}=\sup_x|\rho_\delta(x) \phi^x_r(x)f(x)|\leq\sup_z\left(\rho_\delta(z)\|\phi^z_rf\|_{L^\infty}\right).
	\end{align}
	For $\alpha=1$,  note that by \eqref{GQ1},
	\begin{align*}
	&\sup_z\left(\rho_\delta(z)\|\nabla(\phi^z_r f)\|_{L^\infty}\right)
	\leq \sup_z\left(\rho_\delta(z)(\|\nabla\phi^z_r f\|_{L^\infty}+\|\phi^z_r\nabla f\|_{L^\infty})\right)\\
	&\qquad\qquad\lesssim \sup_z\left(\rho_\delta(z)(\|\phi^z_{2r} f\|_{L^\infty}+\|\phi^z_r\nabla f\|_{L^\infty})\right)\\
	&\qquad\qquad\lesssim \|\rho_\delta f\|_{L^\infty}+\|\rho_\delta\nabla f\|_{L^\infty}\lesssim \|f\|_{\sC^1(\rho_\delta)}.
	\end{align*}
	Moreover, by \eqref{GQ11},
	$$
	\|\rho_\delta \nabla f\|_{L^\infty}\lesssim\sup_z\left(\rho_\delta(z)\|\phi^z_{r/2}\nabla f\|_{L^\infty}\right)
	\lesssim\sup_z\left(\rho_\delta(z)\|\phi^z_r f\|_{\sC^1}\right).
	$$
	Thus \eqref{GD2} holds for $\alpha=1$. For $\alpha=2,\cdots$, it follows by similar calculations.
\end{proof}
The key point of using $\phi^z_r$ is the following simple fact that for any $m\in\mN_0$,
\begin{align}\label{DC9}
\|\nabla \phi_r^z\|_{\sC^m}\lesssim (1+|z|)^{-1}\Rightarrow \sup_z\|\nabla \phi_r^z\|_{\sC^m(\rho^{-1}_1)}<\infty,
\end{align}
where we used $1+|x|\lesssim 1+|z|$ on the support of $\phi_r^z$. 
This provides an extra weight and helps us to obtain the a-priori estimate for the solutions in Besov space with polynomial weight. 

Now we can give

\begin{proof}[Proof of Theorem \ref{Th33}] {\bf (Existence).}
	 Without loss of generality we may assume $\lambda=0$ and $u_0=0$.
	 In fact, for general initial data $u_0\in\cup_{\eps>0}\bC^{1+\alpha+\eps}$, 
	 by considering $\bar u=u-u_0$, we can reduce the nonzero initial value to zero initial value with $f$ replaced by 
	 $\bar f=f+\Delta u_0+b\cdot\nabla u_0\in \bC^{-\alpha}(\rho_\kappa)$. In this case, by Lemma \ref{lem:para},
	 \begin{align*}
	 \|b\circ\nabla \sI(\Delta u_0)\|_{L^\infty_T\bC^\eps(\rho_\kappa)}\lesssim1,
	 \end{align*}
	 and by  Lemma \ref{le:loc} with $\psi=\nabla u_0, f=b, \phi=1, \bar{\rho}=\rho_{2\kappa}, \rho=\rho_\kappa$,
	 \begin{align*}
	 &\|b\circ\nabla \sI(b\cdot \nabla u_0)\|_{L^\infty_T\bC^{1-2\alpha}(\rho_{2\kappa})}
	 \lesssim1.
	 \end{align*}
	 Hence, we still have
	 $$
	 (b,\bar f)\in\mB^\alpha_T(\rho_\kappa).
	 $$

	Now, let $T>0$ and $b_n, f_n\in L^\infty_T\sC^\infty(\rho_\kappa)$ be as in the definition of $\mB^\alpha_T(\rho_\kappa)$.
	For every $n$,  define
	$$
	\bar b_n(t,x):=b_n(t,x)\chi_n(x),\ \ \bar f_n(t,x):=f_n(t,x)\chi_n(x),
	$$
	with $\chi_n$ being the usual cut-off functions. 
	It is well known that there is a unique classical solution $ u_n\in L^\infty_T\sC^2$ solving \eqref{PDE7} with $(b,f)=(\bar b_n,\bar f_n)$. 
Our main aim is to show that there is a constant $C>0$ independent of $n$ such that
\begin{align}\label{Un}
\|u_n\|_{\mS^{2-\alpha}_T(\rho_\delta)}+\|u^\sharp_n\|_{\mS^{3-2\alpha}_T(\rho_{\delta_0})}
	\lesssim_C \mA^{\bar b_n,\bar f_n}_{T,\infty}(\rho_\kappa)
\end{align}
	On the other hand, by  \eqref{EK0} with $\bar\rho=\rho^2=\rho_{2\kappa}$ and $\phi=\psi=\chi_n$, 
	we also have for some $C$ independent of $n$,
	$$
	\mA^{\bar b_n,\bar f_n}_{T,\infty}(\rho_\kappa)\lesssim_C\mA^{b_n,f_n}_{T,\infty}(\rho_\kappa),\quad \ell^{\bar b_n}_{T}(\rho_\kappa)\lesssim_C\ell_T^{b_n}(\rho_\kappa).
	$$
Hence,
$$
\sup_n\Big(\|u_n\|_{\mS^{2-\alpha}_T(\rho_\delta)}+\|u^\sharp_n\|_{\mS^{3-2\alpha}_T(\rho_{\delta_0})}\Big)<\infty.
$$
Thus, by a standard compact argument, we can show the existence of a paracontrolled solution (see \cite{GH18}).
	
	In the following, we devote to proving \eqref{Un}. For simplicity, we drop the bar and subscript $n$ and assume $b,f\in L^\infty_T\sC^2$. We fix $0<r<1/2$. 
	Note that $\phi^z_{2r}=1$ on the support of $\phi^z_r$. 
	For each $z\in\mR^d$,
	it is easy to see that $u_z:=u\phi^z_r$ satisfies the following PDE:
	$$
	\p_t u_z=\Delta u_z+b_z\cdot\nabla u_z+F_z,\ \ u_z(0)=0,
	$$
	where $b_z:=b\phi^z_{2r}$ and 
	$$
	F_z:=f\phi^z_r-2\nabla u\cdot\nabla\phi^z_r-u\Delta\phi^z_r-b\cdot\nabla\phi^z_ru.
	$$
	Let $q$ be the same as in Theorem \ref{Th12}.
	By Theorem \ref{Th12}, there are two constants $c_1,c_2>0$ such that for all $z\in\mR^d$,
	\begin{align}\label{HG7}
	\|u_z\|_{\mS^{2-\alpha}_T}
	\leq c_1(\ell^{b_z}_T)^{\vartheta}\mA^{b_z,F_z}_{T,\infty},\ \|u_z\|_{\mL^\infty_T}\leq c_2(\ell^{b_z}_T)^{\vartheta}\mA^{b_z,F_z}_{T,q}.
	\end{align}
	Let $\eps>0$ be small enough. By the definition of $F_z$, using $\phi_{2r}^z\nabla \phi_r^zu=\nabla \phi_r^zu $  and \eqref{GZ3}, we have
	\begin{align}
	\|F_z\|_{\bC^{-\alpha}}&\leq \|f\phi^z_r\|_{\bC^{-\alpha}}+2\|\nabla u\cdot\nabla\phi^z_r\|_{L^\infty}
	+\|u\Delta\phi^z_r\|_{L^\infty}+\|b\cdot\nabla\phi^z_ru\|_{\bC^{-\alpha}}
	\no\\&\lesssim\|f\|_{\bC^{-\alpha}(\rho_\kappa)}\|\phi^z_r\|_{\bC^{\alpha+\eps}(\rho^{-1}_\kappa)}+\|\nabla u\|_{L^\infty(\rho_1)}
	\|\nabla\phi^z_r\|_{L^\infty(\rho^{-1}_1)}
	\no\\&\quad+\|u\|_{L^\infty(\rho_1)}\|\Delta\phi^z_r\|_{L^\infty(\rho^{-1}_1)}
	+\|b\|_{\bC^{-\alpha}(\rho_\kappa)}\|\nabla\phi^z_ru\|_{\bC^{\alpha+\eps}(\rho_\kappa^{-1})}
	\no\\&\lesssim\|f\|_{\bC^{-\alpha}(\rho_\kappa)}\|\phi^z_r\|_{\sC^1(\rho^{-1}_\kappa)}+\|u\|_{\sC^1(\rho_1)}\|\nabla\phi^z_r\|_{\sC^1(\rho^{-1}_1)}\label{DB9}
	\\&\quad	+\|b\|_{\bC^{-\alpha}(\rho_\kappa)}\|u\|_{\sC^1(\rho_1)}\|\nabla\phi^z_r\|_{\sC^1(\rho^{-1}_1)}
	\|\phi^z_{2r}\|_{\sC^1(\rho_\kappa^{-1})},\no
	\end{align}
	and also,
	\begin{align*}
	\|(b_z\circ\nabla \sI_\lambda F_z)\|_{\bC^{1-2\alpha}}&\leq
	\|b_z\circ\nabla \sI_\lambda({f}\phi^z_r)\|_{\bC^{1-2\alpha}}+\|b_z\circ\nabla \sI_\lambda(b\cdot\nabla \phi^z_ru )\|_{\bC^{1-2\alpha}}
	\\&+\|b_z\circ\nabla \sI_\lambda(u\Delta\phi^z_r+2\nabla u\cdot\nabla\phi^z_r)\|_{L^\infty}=:I_1^z+I_2^z+I_3^z.
	\end{align*}
	For $I^z_1$, by \eqref{EK0} with $\bar\rho\equiv 1$, $\rho=\rho_\kappa$ and $\psi=\phi^z_r$, we have
	\begin{align*}
	I^z_1\lesssim \|\phi^z_{2r}\|_{\bC^{\alpha+\eps}(\rho_\kappa^{-2})}\|\phi^z_r\|_{\bC^{\alpha+\eps}}\mA^{b,f}_{t,\infty}(\rho_\kappa)
	\lesssim \|\phi^z_{2r}\|_{\sC^1(\rho_\kappa^{-2})}\mA^{b,f}_{t,\infty}(\rho_\kappa).
	\end{align*}
	For $I^z_2$, by \eqref{EK0} with $\bar\rho\equiv 1$, $\rho=\rho_\kappa$ and $\psi=\nabla \phi^z_r u$, we have
	\begin{align*}
	I^z_2&\lesssim \|\phi^z_{2r}\|_{\bC^{\alpha+\eps}(\rho_\kappa^{-2})}\|\nabla \phi^z_r u\|_{\mS_t^{\alpha+\eps}}\mA^{b,b}_{t,\infty}(\rho_\kappa)
	\\&\lesssim \|\phi^z_{2r}\|_{\sC^1(\rho_\kappa^{-2})}\|\nabla \phi^z_r\|_{\sC^1(\rho^{-1}_1)}\|u\|_{\mS_t^1(\rho_1)}\ell^{b}_t(\rho_\kappa).
	\end{align*}
	For $I^z_3$, as in \eqref{DB9}, we have
	\begin{align*}
	I^z_3&\lesssim \|b_z\|_{\bC^{-\alpha}}\|\nabla \sI_\lambda(u\Delta\phi^z_r+2\nabla u\cdot\nabla\phi^z_r)\|_{\bC^{\alpha+\eps}}
	\\&\lesssim \|b\|_{\bC^{-\alpha}(\rho_\kappa)}\|\phi^z_{2r}\|_{\bC^{\alpha+\eps}(\rho^{-1}_\kappa)}\|u\Delta\phi^z_r+2\nabla u\cdot\nabla\phi^z_r\|_{L_t^\infty L^\infty}
	\\&\lesssim \|b\|_{\bC^{-\alpha}(\rho_\kappa)}\|\phi^z_{2r}\|_{\sC^1(\rho^{-1}_\kappa)}
	\|u\|_{L_t^\infty\sC^1(\rho_1)}\|\nabla\phi^z_r\|_{\sC^1(\rho^{-1}_1)}.
	\end{align*}
	Combining the above calculations, by the definition of $\mA^{b_z,F_z}_{T,q}$ and \eqref{DC9}, we get
	\begin{align*}
	\mA^{b_z,F_z}_{T,q}
	&=\sup_\lambda\|b_z\circ \nabla \sI_\lambda F_z\|_{L^q_T\bC^{1-2\alpha}}+\|b_z\|_{L^\infty_T\bC^{-\alpha}}\|F_z\|_{L^q_T\bC^{-\alpha}}
	\\&\lesssim\Big(\|\phi^z_{2r}\|_{\sC^1(\rho_\kappa^{-2})}+\|\phi_{2r}^z\|_{\sC^1(\rho_\kappa^{-1})}(\|\phi_r^z\|_{\sC^1(\rho_\kappa^{-1})}+\|\phi_{2r}^z\|_{\sC^1(\rho_\kappa^{-1})})\Big)
	\\\qquad\qquad&\qquad\times\left(\mA^{b,f}_{T,\infty}(\rho_\kappa)+\ell_T^b(\rho_\kappa)\left(\int^T_0\|u\|_{\mS^{2\alpha}_t(\rho_1)}^q\dif t\right)^{1/q}\right),
	\end{align*}
	where we have used
	\begin{align*}
	\|b_z\|_{L^\infty_T\bC^{-\alpha}}\lesssim \|b\|_{L_T^\infty\bC^{-\alpha}(\rho_\kappa)}\|\phi_{2r}^z\|_{\sC^1(\rho_\kappa^{-1})},
	\end{align*}
	and by \eqref{DB9} and \eqref{DC9},
	\begin{align*}
\|F_z\|_{L^q_T\bC^{-\alpha}}&\lesssim 
	\|f\|_{L_T^q\bC^{-\alpha}(\rho_\kappa)}\|\phi_r^z\|_{\sC^1(\rho_\kappa^{-1})}\\&+\Big(1+\ell^b_T(\rho_\kappa)
	\|\phi_{2r}^z\|_{\sC^1(\rho_\kappa^{-1})}\Big)\left(\int_0^T\|u(t)\|_{\sC^1(\rho_1)}^q\dif t\right)^{1/q}.
	\end{align*}
	By Lemma \ref{cha}, we have
	$$\sup_z\rho_\kappa(z)\|\phi_{2r}^z\|_{\sC^1(\rho_\kappa^{-1})}\lesssim1.$$
	On the other hand, by Lemma \ref{cha} and \eqref{EK0} with $\bar\rho=1, \bar\rho=\rho_\kappa$, we have
	 \begin{align*}\sup_z\rho^2_{\kappa}(z)\ell_T^{b_z}
	 \lesssim \sup_z\rho^2_{\kappa}(z)(\|\phi_{2r}^z\|_{\sC^1(\rho_\kappa^{-2})}+\|\phi_{2r}^z\|_{\sC^1(\rho_\kappa^{-1})}^2)\ell_T^b(\rho_\kappa)
	 \lesssim\ell_T^b(\rho_\kappa),\end{align*}
	 which together with the above  estimate implies that for $\delta=(2\vartheta+2)\kappa\leq 1$,
	\begin{align*}
	&\sup_z\rho_\delta(z)(\ell_T^{b_z})^{\vartheta}\mA^{b_z,F_z}_{T,q}
	\leq \Big(\sup_z\rho_\kappa^2(z)\ell^{b_z}_T\Big)^{\vartheta}\sup_z\rho^2_\kappa(z)\mA^{b_z,F_z}_{T,q}
	\\&\quad\leq \big(\ell^{b}_T(\rho_\kappa)\big)^{\vartheta+1}
	\left(\mA^{b,f}_{T,\infty}(\rho_\kappa)+\left(\int^T_0\|u\|_{\mS^{2\alpha}_t(\rho_1)}^q\dif t\right)^{1/q}\right).
	\end{align*}
	Note that by \eqref{AM1} and Young's inequality,
	$$
	\|u\|_{\mS^{2\alpha}_t(\rho_1)}
	\leq \eps\|u\|_{\mS^{2-\alpha}_t(\rho_1)}+C_\eps\|u\|_{\mL^\infty_t(\rho_1)}.
	$$
	Hence, multiplying both sides of \eqref{HG7} by $\rho_\delta(z)$ 
	we arrive at
	$$
	\|u\|_{\mS^{2-\alpha}_T(\rho_\delta)}
	\leq \eps\|u\|_{\mS^{2-\alpha}_T(\rho_\delta)}+C_\eps\|u\|_{\mL^\infty_T(\rho_1)}+C_\eps\mA^{b,f}_{T,\infty}(\rho_\kappa),
	$$
	and 
	\begin{align*}
	\|u\|_{\mL^\infty_T(\rho_\delta)}
	&\lesssim \mA^{b,f}_{T,\infty}(\rho_\kappa)+\left(\int^T_0\|u\|^q_{\mS^{2-\alpha}_t(\rho_1)}
	\dif t\right)^{1/q}.
	\end{align*}
	Both of the above two estimates implies that
	$$
	\|u\|_{\mL^\infty_T(\rho_1)}\leq \|u\|_{\mL^\infty_T(\rho_\delta)}\lesssim \mA^{b,f}_{T,\infty}(\rho_\kappa)+\left(\int^T_0\|u\|^q_{\mL^\infty_t(\rho_1)}\dif t\right)^{1/q}.
	$$
	Finally, we use Gronwall's inequality to conclude the first estimate in \eqref{Un}.
	\medskip
	\\
	 By \eqref{DT11}, \eqref{GZ11} and \eqref{EG01} we have for weight $\rho, \bar\rho\in \sW$
	\begin{align}\label{b:sharp1}
	\|u^\sharp\|_{L^\infty_T\bC^{2-\alpha}(\rho\bar\rho)}
	&\lesssim\|u\|_{L^\infty_T\bC^{2-\alpha}(\rho\bar\rho)}+\|\nabla u\Prec\sI_\lambda b\|_{L^\infty_T\bC^{2-\alpha}(\rho\bar\rho)}
	+\|\sI_\lambda f\|_{L^\infty_T\bC^{2-\alpha}(\rho\bar\rho)}\no\\
	&\lesssim\|u\|_{L^\infty_T\bC^{2-\alpha}(\bar\rho)}+\|\nabla u\|_{\mL^\infty_T(\bar\rho)}\|b\|_{L^\infty_T\bC^{-\alpha}(\rho)}
	+\|f\|_{L^\infty_T\bC^{-\alpha}(\rho)}\no\\
	&\lesssim\sqrt{\ell^b_T(\rho)}\|u\|_{L^\infty_T\bC^{2-\alpha}(\bar\rho)}+\|f\|_{L^\infty_T\bC^{-\alpha}(\rho)}.
	\end{align}
	Next we estimate each term on the right hand side of \eqref{DT110} by using Lemma \ref{lem:para}.
	\begin{enumerate}[$\bullet$]
		\item By \eqref{GA4}, \eqref{*} we have
		$$
		\|\nabla u\prec b-\nabla u\Prec b\|_{L^\infty_T\bC^{1-2\alpha}(\rho\bar\rho)}\lesssim 
		\|u\|_{\mS_T^{2-\alpha}(\bar\rho)}\|b_{L^\infty_T\bC^{-\alpha}(\rho)}.
		$$
		\item By \eqref{GZ1}, we have
		$$
		\|\nabla u\succ b\|_{L_T^\infty\bC^{1-2\alpha}(\rho\bar\rho)}\lesssim 
		\|u\|_{L_T^\infty\bC^{2-\alpha}(\bar\rho)}\|b\|_{L_T^\infty\bC^{-\alpha}(\rho)}.
		$$
		\item By \eqref{GA3} and \eqref{EG01} we have
		$$
		\|[\sL,\nabla u \Prec]\sI b\|_{L_T^\infty\bC^{1-2\alpha}(\rho\bar\rho)}\lesssim 
		\|u\|_{\mS_T^{2-\alpha}(\bar\rho)}\|b\|_{L_T^\infty\bC^{-\alpha}(\rho)}.
		$$
		\item By Lemma \ref{Le32} with $\gamma=2-2\alpha$, $\beta\in(\alpha,2-2\alpha)$, we have
		\begin{align*}
		\|b\circ\nabla u\|_{L_T^\infty\bC^{1-2\alpha}(\rho^{2+\eps}\bar\rho)}\lesssim \|u\|_{\mS_T^{2-\alpha}(\bar\rho)}
		+\|u^\sharp\|_{L_T^\infty\bC^{\beta+1}(\rho^{1+\eps}\bar\rho)}+\mA^{b,f}_{T,\infty}(\rho).
		\end{align*}
	\end{enumerate}
	Combining the above calculations and by \eqref{DT110} and \eqref{EG1} with $\theta=2$ and $q=\infty$, we obtain
	\begin{align}\label{CC8}
	\|u^\sharp\|_{\mathbb{S}_T^{3-2\alpha}(\rho^{2+\eps}\bar\rho)}&\lesssim \|u\|_{\mS_T^{2-\alpha}(\bar\rho)}
	+\|u^\sharp\|_{L_T^\infty\bC^{\beta+1}(\rho^{1+\eps}\bar\rho)}+\mA^{b,f}_{T,\infty}(\rho).
	\end{align}
	On the other hand, for $\eps>\frac{2\alpha-1}{2-3\alpha}$, one can choose $\beta$ close to $\alpha$ so that
	$$
	\theta:=\tfrac{\eps}{1+\eps}=\tfrac{\alpha+\beta-1}{1-\alpha}.
	$$
	Thus by interpolation inequality \eqref{DQ1}, Young's inequality and \eqref{b:sharp1}, for any $\delta>0$,
	\begin{align*}
	\|u^\sharp\|_{L_T^\infty\bC^{\beta+1}(\rho^{1+\eps}\bar\rho)}
	&\lesssim \|u^\sharp\|_{L_T^\infty\bC^{3-2\alpha}(\rho^{2+\eps}\bar\rho)}^{\theta}
	\|u^\sharp\|_{L^\infty_T\bC^{2-\alpha}(\rho\bar\rho)}^{1-\theta}
	\\&\leq \delta\|u^\sharp\|_{L_T^\infty\bC^{3-2\alpha}(\rho^{2+\eps}\bar\rho)}
	+C_\delta\Big(\|u\|_{\mS_T^{2-\alpha}(\bar\rho)}+\mA^{b,f}_{T,\infty}(\rho)\Big).
	\end{align*}
	Substituting this into \eqref{CC8}, we obtain the second estimate by taking $\rho=\rho_\kappa, \bar\rho=\rho_\delta$.
	
{\bf (Uniqueness).} It follows by Theorem \ref{Th72} in the appendix.
\end{proof}

\section{Hamilton-Jacobi-Bellman equations}\label{s:HJB}

In this section we consider the following HJB equation:
\begin{align}\label{HJB}
\p_t v=\tr(a\cdot\nabla^2 v)+B\cdot\nabla v+H(v,\nabla v),\ v(0)=v_0,
\end{align}
where $a:\mR_+\times\mR^d\to\mR^d\otimes\mR^d$ is a symmetric matrix-valued measurable function,
and $B:\mR_+\times\mR^d\to\mR^d$ is a vector-valued measurable function, and
$$
H(t,x,v,Q): \mR^+\times\mR^d\times\mR\times\mR^d\to\mR
$$
is a real-valued measurable function, and continuous in $v,Q$ for each $t,x$.

 For instance, for any $\zeta\in[1,2]$, the equation
\begin{align}\label{eq:22}
\sL v= |\nabla v|^\zeta +B \cdot \nabla v + f
\end{align}
is a typical HJB equation. 
Note that for $\lambda>0$, if we define
$$
 v_\lambda(t,x):=v(\lambda^2 t,\lambda x),\ B_\lambda(t,x):=\lambda B(\lambda^2 t,\lambda x),\ f_\lambda(t,x):=\lambda^2f(\lambda^2 t,\lambda x),
$$
then
$$
\sL  v_\lambda=\lambda^{2-\zeta} |\nabla v_\lambda|^\zeta +B_\lambda \cdot \nabla v_\lambda + f_\lambda.
$$
In particular, if $\zeta=2$, then the nonlinear term has the same order as the Laplacian term in scaling level. 
In this case, we shall say HJB \eqref{eq:22} being {\it critical}.
While for  $\zeta<2$, the nonlinear term can be controlled well by the Laplacian term. 
In this case, we shall say HJB \eqref{eq:22} being {\it subcritical}\footnote{Here the critical and subcritical conditions  are different from the meaning in \cite{Hai14}}.

Throughout this section we shall use the following polynomial weight function 
$$
\rho_\delta(x):=\<x\>^{-\delta}=(1+|x|^2)^{-\delta/2}\Rightarrow\rho^\gamma_\delta=\rho_{\gamma\delta},\ \delta,\gamma\in\mR,
$$
and make the following elliptic assumption on $a$:
\begin{enumerate}[{\bf (H$^\alpha_1$)}]
	\item $a:\mR_+\times\mR^d\to\mR^d\otimes\mR^d$ is a symmetric $d\times d$-matrix-valued  measurable function and satisfies that for some $c_0\in(0,1)$,
	\begin{align}\label{Uni}
	c_0|\xi|^2\leq \sum_{i,j=1}^da_{ij}(t,x)\xi_i\xi_j\leq c_0^{-1}|\xi|^2,\ \ \forall \xi\in\mR^d,
	\end{align}
	and for some $\alpha\in(0,1)$ and $c_1\geq 1$,
	$$
	|a(t,x)-a(t,y)|\leq c_1|x-y|^\alpha.
	$$
\end{enumerate}
About the nonlinear term $H$, we separately consider two cases: subscritical case for all $d\in\mN$ and critical case only for $d=1$, and assume
\begin{enumerate}[{\bf (H$^{\delta,\zeta}_{\rm sub}$)}]
	\item Suppose that for some $\delta,\zeta\in[0,2)$ and $c_2>0$,
	\begin{align}\label{HH}
	|H(t,x,v,Q)|\lesssim_{c_2}\<x\>^\delta+|Q|^\zeta.
	\end{align}
\end{enumerate}
\begin{enumerate}[{\bf (H$^{\delta,\beta}_{\rm crit}$)}]
	\item Suppose that $d=1$ and for some $\delta\in[0,2)$ and $c_2>0$,
\begin{align}\label{HH2}
	|H(t,x,v,Q)|\lesssim_{c_2}\<x\>^\delta+|Q|^2,\ \ |\p_vH(t,x,v,Q)|\lesssim_{c_2}\<x\>^\delta+|v|+|Q|,
\end{align}
	and for some $\beta\in(0,1]$ and all $|x-y|\leq 1$,
\begin{align}\label{HH1}
|H(t,x,v,Q)-H(t,y,v,Q)|\lesssim_{ c_2}|x-y|^\beta (\<x\>^\delta+\<y\>^\delta+|v|^2+|Q|^2).
\end{align}
\end{enumerate}

We introduce the following definition of strong solution to HJB equation \eqref{HJB}.
\bd\label{Def41}
We call a function $v\in \cap_{p\geq 2}\mH^{2,p}_{loc}$ strong solution to \eqref{HJB} if
for all $\psi\in C^\infty_c(\mR^d)$ and $t\geq 0$,
$$
\<v(t),\psi\>=\<v_0,\psi\>+\int^t_0\Big\<\big(\tr(a\cdot\nabla^2 v)+B\cdot\nabla v+H(v,\nabla v)\big)(s),\psi\Big\>\dif s,
$$
where $\<v_0,\psi\>:=\int v_0\psi$.
In particular, for all $t\geq 0$ and Lebesgue almost all $x\in\mR^d$,
$$
v(t,x)=v_0(x)+\int^t_0\Big(\tr(a\cdot\nabla^2 v)+B\cdot\nabla v+H(v,\nabla v)\Big)(s,x)\dif s.
$$
\ed
The aim of this section is to establish the following strong well-posedness for HJB equation \eqref{HJB}.
For simplicity of notation, we introduce the following parameter set for saying the dependence of a constant:
$$
\Theta:=(T, d,\alpha,\beta,\zeta,\delta,c_0,c_1,c_2).
$$
\bt\label{Th42}
Let $T>0$, $\delta\in(0,2)$ and $\alpha,\beta,\delta_1\in(0,1]$. Suppose that {\bf (H$^\alpha_1$)}, $B\in \mL_T^\infty(\rho_{\delta_1})$
 and {\bf (H$^{\delta,\zeta}_{\rm sub}$)} or {\bf (H$^{\delta,\beta}_{\rm crit}$)} hold. We let
\begin{align}\label{ETA}
\left\{
\begin{aligned}
&\eta>\tfrac{\zeta\delta}{2-\zeta}\vee[2\delta_1+\delta],\ &\mbox{\rm under {\bf (H$^{\delta,\zeta}_{\rm sub}$)}};\\
&\eta>2\left(\tfrac{(1+2\beta)\delta}{\beta}\vee(\delta_1+\delta)\right),&\mbox{\rm under {\bf (H$^{\delta,\beta}_{\rm crit}$)}}.
\end{aligned}
\right.
\end{align}
{\bf (Existence)} 
For any initial value $v_0\in \sC^2(\rho_{\delta})$, there are $p_0$ large enough
and strong solution $v$ for HJB equation \eqref{HJB}, which satisfies the following estimate: for any $p\geq p_0$,
there is a constant $C=C(\Theta, p,\eta,\delta_1,\|B\|_{\mL^\infty_T(\rho_{\delta_1})},\|v_0\|_{\sC^2(\rho_\delta)})>0$ such that
\begin{align}\label{DW1}
\|v\|_{\mL^\infty_T(\rho_{\delta})}+\|\p_tv\|_{\mL^{p}_T(\rho_\eta)}+\|v\|_{\mH^{2,p}_T(\rho_\eta)}\leq C.
\end{align}
In particular, for any $0\leq \eps'<\eps\leq 2$,
$$
\|v\|_{C^{\eps'/2}_T\bC^{2-\eps}(\rho_\eta)}\leq C.
$$
{\bf (Uniqueness)} 
If, in addition, for some $C>0$,
\begin{align}\label{WQ5}
|\p_v H(t,x,v,Q)|^{1/2}+|\p_Q H(t,x,v,Q)|\lesssim_C \<x\>+|v|^{1/\delta}+|Q|^{1/\eta},
\end{align}
then there is a unique strong solution with regularity \eqref{DW1}.
\et

\br\rm
When $a\in L^\infty_T\sC^1$, the above regularity result could be obtained by De-Giorgi's iteration method
since it can be written as the divergence form (cf. \cite{LSU68}).
However, for H\"older diffusion $a$ as we need, it seems not be studied in the literature. 
Besides, the unbounded $B$ and $H$
also cause many difficulties for obtaining the global estimates, which is crucial for  a-priori estimate such as \eqref{eq:2} and KPZ equation.
We believe that the above theorem is of its own interest.
\er

In the following we first establish a maximum principle in Section \ref{sec:4.1}. The subcritical case is treated in Section \ref{sec:4.2} by using $L^\infty(\rho_\delta)$-estimate and $L^p$-theory for PDEs. For the critical case,  we take spatial derivative on both sides and obtain a PDE of divergence  form. 
Then using the  $L^\infty(\rho_\delta)$-bound and energy estimate we obtain the $\mH^{2,p}_T(\rho_\eta)$-estimate in Section \ref{sec:4.3}. 
\subsection{Maximum principle in weighted spaces}\label{sec:4.1}
We first show the following maximum principle in weighted spaces.
\bt(Maximum principle)\label{Th43}
Let $T>0$ and $\delta\in(0,2)$. Suppose \eqref{Uni} and for some $c_2, c_3>0$,
$$
|H(t,x,v,Q)|\leq c_2 \<x\>^\delta+c_3|Q|^2,\ \ B\in \mL^\infty_T(\rho_1).
$$
For any $v_0\in L^\infty(\rho_{\delta})$, there is a function $C(r)=C_\Theta(r)>0$ with $C(0)=0$ such that
for any strong solution $v\in \cap_{p\geq2} \mH^{2,p}_{loc}\cap \mL^\infty_T(\rho_\delta)$ of \eqref{HJB} with initial value $v_0$,
\begin{align}\label{Max}
\|v\|_{\mL^\infty_T(\rho_{\delta})}\leq C(c_2+\|v_0\|_{L^\infty(\rho_{\delta})}).
\end{align}
\et

\begin{proof}
	We use a probabilistic method.
	For $\lambda>0$, define
	$$
	w(t,x):=\e^{\lambda v(t,x)}.
	$$
	By the chain rule, it is easy to see that $w$ satisfies
	$$
	\p_t w=\tr(a\cdot\nabla^2 w)+B\cdot\nabla w+\lambda w \Big(H(v,\nabla v)-\lambda \tr(a\cdot\nabla v\otimes\nabla v)\Big).
	$$
	For simplicity of notations, we write
	$$
	F_\delta(x):=c_2\<x\>^\delta,\ U_\lambda:=\lambda w \Big(H(v,\nabla v)-\lambda \tr(a\cdot\nabla v\otimes\nabla v)-F_\delta\Big).
	$$
	Next we reverse the time variable. For a space-time function $f$, we set
	$$
	f^T(t,x):=f(T-t,x).
	$$
	It is easy to see that $w^T(t,x)=w(T-t,x)$ solves the following backward equation:
	\begin{align}\label{GG4}
	\p_t w^T+\tr(a^T\cdot\nabla^2 w^T)+B^T\cdot\nabla w^T+U_\lambda^T+\lambda w^TF_\delta =0,
	\end{align}
	with subjected to the final condition
	\begin{align}\label{GG5}
	w^T(T,x)=w(0,x)=\e^{\lambda v_0(x)}.
	\end{align}
	Under \eqref{Uni} and $B\in \mL^\infty_T(\rho_1)$, for each $(t,x)\in[0,T]\times\mR^d$,
	it is well known that the following SDE has a (probabilistically) weak solution starting from $x$ at time $t$ (see \cite[page 87, Theorem 1]{Kry80})
	$$
	X^{t,x}_s=x+\int^s_t\sqrt{2a^T}(r,X^{t,x}_r)\dif W_r+\int^s_tB^T(r,X^{t,x}_r)\dif r,\ \ \forall s\in[t,T],
	$$
	where $W$ is a $d$-dimensional Brownian motion on some stochastic basis $(\Omega',\mathcal{F}',\mathbb{P})$.
	For $R>0$, define a stopping time
	$$
	\tau_R:=\inf\{s\geq t: |X^{t,x}_s|> R\}.
	$$
	It is well known that the following Krylov estimate holds (\cite[page 52, Theorem 2]{Kry80}): for any $p\geq d+1$,
	$$
	\mE\left(\int^{T\wedge\tau_R}_t f(s, X^{t,x}_s)\dif s\right)\leq C_R\left(\int^T_t\!\!\!\int_{B_R}|f(s, x)|^p\dif x\dif s\right)^{1/p}.
	$$
	Since $v\in  \cap_{p\geq 2}\mH^{2,p}_{loc}\cap \mL^\infty_T(\rho_\delta)$, it is easy to see that
	$$
	w^T\in\cap_{p\geq 2}\mH^{2,p}_{loc},\ \ \p_t w^T\in\cap_{p\geq 2}\mL^p_{loc}.
	$$
	Thus, for each fixed $(t,x)$, by generalized It\^o's formula (see \cite[page 122, Theorem 1]{Kry80}), we have
	\begin{align*}
	\dif_s w^T(s,X^{t,x}_s)&=(\p_sw^T+\tr(a^T\cdot \nabla^2w^T)+B^T\cdot\nabla w^T)(s,X^{t,x}_s)\dif s\\
	&\quad+(\sqrt{2a^T}\cdot\nabla w^T)(s,X^{t,x}_s)\dif W_s,
	\end{align*}
	and by \eqref{GG4} and \eqref{GG5},
	\begin{align*}
	&\e^{\int^{t'}_{t} \lambda F_\delta(X^{t,x}_s)\dif s} w^T(t',X^{t,x}_{t'})\\
	&=w^T(t,x)+\int^{t'}_t\e^{\int^s_{t} \lambda F_\delta(X^{t,x}_r)\dif r}\dif_s w^T(s,X^{t,x}_s)\\
	&\quad+\int^{t'}_t\e^{\int^s_{t} \lambda F_\delta(X^{t,x}_r)\dif r}(\lambda F_\delta w^T)(s,X^{t,x}_s)\dif s\\
	&=w^T(t,x)-\int^{t'}_t\e^{\int^s_{t} \lambda F_\delta(X^{t,x}_r)\dif r}U^T_\lambda(s,X^{t,x}_s)\dif s+M_{t'},
	\end{align*}
	where
	$$
	M_{t'}:=\int^{t'}_t\e^{\int^s_{t} \lambda F_\delta(X^{t,x}_r)\dif r}(\sqrt{2a^T}\cdot\nabla w^T)(s,X^{t,x}_s)\dif W_s.
	$$
	By \eqref{Uni} and $|H(v,Q)|\leq F_\delta+c_3|Q|^2$, one can choose $\lambda=c_3/c_0$ so that
	$$
	U^T_\lambda\leq \lambda w \Big(c_3|\nabla v|^2-\lambda c_0 |\nabla v|^2\Big)=0.
	$$
	Hence, for $\lambda=(c_3/c_0)\vee1$,
	\begin{align*}
	\e^{\lambda v(T-t,x)}=w^T(t,x)\leq\e^{\int^{t'}_{t} \lambda F_\delta(X^{t,x}_s)\dif s}w^T(t',X^{t,x}_{t'})-M_{t'}.
	\end{align*}
	Since $t'\mapsto M_{t'\wedge\tau_R}$ is a martingale, we have
	$$
	\e^{\lambda v(T-t,x)}\leq \mE\left(\e^{\int^{T\wedge\tau_R}_{t} \lambda F_\delta(X^{t,x}_s)\dif s}
	w^T(T\wedge\tau_R, X^{t,x}_{T\wedge\tau_R})\right).
	$$
	On the other hand, by Lemma \ref{th:7.2} in appendix, for any $\gamma\geq 0$ and $\alpha\in[0,2)$,
	$$
	\mE\left(\e^{\gamma \sup_{s\in[t,T]}\<X^{t,x}_s\>^\alpha}\right)\leq C(\gamma)\e^{C_2\gamma \<x\>^\alpha}.
	$$
	Since $w^T(t,x)\leq \e^{\lambda\|v\|_{\mL^\infty_T(\rho_\delta)}\<x\>^\delta}$, letting $R\to\infty$ and by the dominated convergence theorem, we get
	\begin{align*}
	\e^{\lambda v(T-t,x)}
	&\leq \mE\left(\e^{\int^{T}_{t} \lambda F_\delta(X^{t,x}_s)\dif s}w^T(T, X^{t,x}_T)\right)
	=\mE\left(\e^{\int^{T}_{t} \lambda F_\delta(X^{t,x}_s)\dif s+\lambda v_0(X^{t,x}_T)}\right)\\
	&\leq\mE\left(\e^{\ell_0\sup_{s\in[t,T]}\<X^{t,x}_s\>^{\delta}}\right)\leq C(\ell_0)\e^{\ell_0 \<x\>^{\delta}},
	\end{align*}
	where $\ell_0:=\lambda(c_2+\|v_0\|_{L^\infty(\rho_{\delta})})$.
	Hence,
	$$
	v(T-t,x)\leq C(\ell_0)\<x\>^{\delta}.
	$$
	By applying the above estimate to $-v$, we obtain the desired estimate.
\end{proof}


\subsection{Subcritical case}\label{sec:4.2}
In this subsection we consider the subcritical case {\bf (H$^{\delta,\zeta}_{\rm sub}$)} and prove some a priori regularity estimate.
For this aim, we prepare the following simple interpolation inequality in weighted spaces, which will play important roles in treating the weights.
\bl
(i) For any $p\geq 2$ and $r,p\in[1,\infty]$ with $\frac{2}{p}=\frac{1}{r}+\frac{1}{q}$,
and $\delta,\delta_1,\delta_2\in\mR$ with $\delta_1+\delta_2=2\delta$,
there is a constant $C=C(p,r,q,\delta,\delta_1,\delta_2)>0$ such that
\begin{align}\label{Es10}
\|\nabla v\rho_\delta\|_{L^p}\lesssim_C\|\nabla^2 v\rho_{\delta_1}\|^{1/2}_{L^q}\|v\rho_{\delta_2}\|^{1/2}_{L^r}
+\|v\rho_{\delta+1}\|_{L^p}.
\end{align}
(ii) For any $p,q\in[2,\infty), r\in [2,\infty]$ with $\frac{q+2}{p}=1+\frac{2}{r}$, and $\delta,\delta_1,\delta_2\in\mR$
with $\delta=\frac{q\delta_1}{q+2}+\frac{2\delta_2}{q+2}$, there is a constant $C=C(p,q,r,\delta,\delta_1,\delta_2)>0$ such that
\begin{align}\label{KJ3}
\|\nabla v\rho_\delta\|_{L^p}\lesssim_C
\left(\int|\nabla^2 v|^2||\nabla v|^{q-2}\rho^q_{\delta_1}\right)^{\frac{1}{q+2}}
\|v\rho_{\delta_2}\|_{L^r}^{\frac{2}{q+2}}+\|v\rho_{\delta+1}\|_{L^p}.
\end{align}
\el
\begin{proof}
	By definition and the integration by parts, we have
	\begin{align}
	\|\nabla v\rho_\delta\|_{L^p}^p&=\int |\nabla v|^p\rho_{\delta p}=\int \<\nabla v,\nabla v|\nabla v|^{p-2}\rho_{\delta p}\>\no\\
	&\lesssim \int|v|\Big(|\nabla^2 v||\nabla v|^{p-2}\rho_{\delta p}+|\nabla v|^{p-1}|\nabla\rho_{\delta p}|\Big).\label{KJ1}
	\end{align}
	(i) By H\"older's inequality we have
	\begin{align*}
	\int|v||\nabla^2 v||\nabla v|^{p-2}\rho_{\delta p}
	&\leq \|v\rho_{\delta_2}\|_{L^r}\|\nabla^2 v\rho_{\delta_1}\|_{L^q}\|\nabla v\rho_\delta\|_{L^p}^{p-2},
	\end{align*}
	and by $|\nabla\rho_{\delta}|\lesssim\rho_{\delta+1}$,
	\begin{align}\label{KJ2}
	\int|v| |\nabla v|^{p-1}|\nabla\rho_{\delta p}|\leq\|\nabla v\rho_\delta\|_{L^p}^{p-1}\|v\rho_{\delta+1}\|_{L^p}.
	\end{align}
	Therefore,
	$$
	\|\nabla v\rho_\delta\|_{L^p}^p\lesssim \|v\rho_{\delta_2}\|_{L^r}\|\nabla^2 v\rho_{\delta_1}\|_{L^q}\|\nabla v\rho_\delta\|_{L^p}^{p-2}
	+\|\nabla v\rho_\delta\|_{L^p}^{p-1}\|v\rho_{\delta+1}\|_{L^p}.
	$$
	Thus by Young's inequality, we obtain \eqref{Es10}.
	\medskip\\
	(ii) On the other hand, by H\"older's inequality we also have
	\begin{align*}
	&\int|v||\nabla^2 v||\nabla v|^{p-2}\rho_{\delta p}
	\leq \left(\int|\nabla^2 v|^2||\nabla v|^{q-2}\rho_{\delta_1 q}\right)^{1/2}\|\nabla v\rho_{\delta}\|_{L^p}^{p-\frac{q}{2}-1}\|v\rho_{\delta_2}\|_{L^r},
	\end{align*}
	which together with \eqref{KJ1} and \eqref{KJ2} yields \eqref{KJ3}.
\end{proof}
We now prove the following a priori regularity estimate.
\bt\label{Th47}
Let $T>0$, $\delta\in (0,2)$ and $\alpha,\delta_1\in(0,1]$.
Suppose {\bf (H$^\alpha_1$)}, $B\in \mL_T^\infty(\rho_{\delta_1})$
 and {\bf (H$^{\delta,\zeta}_{\rm sub}$)}. 
Then for any $\eta>(2\delta_1+\delta)\vee\frac{\zeta\delta}{2-\zeta}$ and $v_0\in \sC^2(\rho_{\delta})$, 
there is a $p_0$ large enough 
so that for all $p>p_0$ and any strong solution $v$ of HJB \eqref{HJB},
$$
\|\p_t (v\rho_\eta)\|_{\mL^{p}_T}+\|v\rho_\eta\|_{\mH^{2,p}_T}\leq C,
$$
where $C=C(\Theta,\eta,p,\delta_1,\|B\|_{\mL^\infty_T(\rho_{\delta_1})},\|v_0\|_{\sC^2(\rho_{\delta})})$.
\et
\begin{proof}
	Multiplying both sides of \eqref{HJB} by $\rho_\eta$, we get
	\begin{align}\label{Eq1}
	\p_t (v\rho_\eta)=\tr(a\cdot \nabla^2(v\rho_\eta))-\Gamma_\rho+(B\cdot\nabla v)\rho_\eta+H(v,\nabla v)\rho_\eta,
	\end{align}
	where
	$$
	\Gamma_\rho=\tr(a\cdot (2\nabla v\otimes\nabla\rho_\eta+v\nabla^2\rho_\eta)).
	$$
	Fix
	$$
	p>\frac{(2-\zeta)d}{(2-\zeta)\eta-\zeta\delta}\vee\frac{d}{\eta-2\delta_1-\delta}=:p_0.
	$$
	By the $L^p$-theory of PDEs (see \cite{K08}), there is a constant $C=C(\Theta, p)$ such that
	$$
	\|\p_t(v\rho_\eta)\|_{\mL^{p}_T}+\|v\rho_\eta\|_{\mH^{2,p}_T}\lesssim_C\|H(v,\nabla v)\rho_\eta
	+(B\cdot\nabla v)\rho_\eta-\Gamma_\rho\|_{\mL^p_T}+\|v_0\rho_\eta\|_{H^{2,p}}.
	$$
Since $p(\eta-\delta)>d$, we have
	$$
	\|v_0\rho_\eta\|_{H^{2,p}}\lesssim\|v_0\rho_{\delta}\|_{\sC^2}\left(\int_{\mR^d}\rho^p_{\eta-\delta}(x)\dif x\right)^{1/p}
	\lesssim \|v_0\|_{\sC^2(\rho_{\delta})},
	$$
and by \eqref{HH},
	\begin{align*}
	\|H(v,\nabla v)\rho_\eta\|_{\mL^p_T}\lesssim \|\rho_{\eta-\delta}\|_{L^p}+\||\nabla v|^\zeta\rho_\eta\|_{\mL^p_T}
	\lesssim 1+\|\nabla v\rho_{\eta/\zeta}\|_{\mL^{\zeta p}_T}^\zeta.
	\end{align*}
	By interpolation inequality \eqref{Es10} and using $|\nabla \rho_\delta|\lesssim\rho_{\delta+1}$, we have
	\begin{align*}
	\|\nabla v\rho_{\eta/\zeta}\|^\zeta_{\mL^{\zeta p}_T}
	&\leq\|\nabla^2 v\rho_\eta\|_{\mL^p_T}^{\zeta/2}\|v\rho_{\eta(2/\zeta-1)}\|_{\mL^q_T}^{\zeta/2}
	+\|v\rho_{\eta/\zeta+1}\|^\zeta_{\mL^{\zeta p}_T},
	\end{align*}
	where $q=p\zeta/(2-\zeta)$.
	Since $p(\eta-\zeta\delta/(2-\zeta))>d$, by \eqref{Max}, we have
	\begin{align*}
	\|v\rho_{2\eta/\zeta-\eta}\|_{\mL^q_T}^q&=\int^T_0\!\!\!\int_{\mR^d} |v(t,x)|^q\rho_{\eta p}(x)\dif x\dif t\\
	&\lesssim\int_{\mR^d}\rho_{\delta}(x)^{-p\zeta/(2-\zeta)}\rho_{\eta p}(x)\dif x\\
	&\lesssim\int_{\mR^d}(1+|x|)^{\frac{p\zeta\delta}{2-\zeta}-\eta p}\dif x\lesssim1,
	\end{align*}
	and also,
	$$
	\|v\rho_{\eta/\zeta+1}\|^\zeta_{\mL^{\zeta p}_T}\lesssim \|\rho_{\eta/\zeta+1-\delta}\|^\zeta_{\mL^{\zeta p}_T}\lesssim 1.
	$$
	Thus, for any $\eps\in(0,1)$, by Young's inequality,
	$$
	\|H(v,\nabla v)\rho_\eta\|_{\mL^p_T}\lesssim\eps\|\nabla^2 v\rho_\eta\|_{\mL^p_T}+1.
	$$
Since $B\in \mL^\infty_T(\rho_{\delta_1})$ and $\eta>2\delta_1+\delta$ and $p(\eta-2\delta_1-\delta)>d$, we also have by \eqref{Es10} and \eqref{Max}
	\begin{align*}
	\|(B\cdot\nabla v)\rho_\eta\|_{\mL^p_T}&\lesssim\|\rho_{\eta-\delta_1}|\nabla v|\|_{\mL^p_T}
	\lesssim\|\nabla^2 v\rho_\eta\|_{\mL^p_T}^{1/2}\|v\rho_{\eta-2\delta_1}\|^{1/2}_{\mL^p_T}+\|v\rho_{\eta+1}\|_{\mL^p_T}
	\\&\lesssim \eps \|\nabla^2 v\rho_\eta\|_{\mL^p_T}+1.
	\end{align*}
	Moreover,  noting that
	\begin{align*}
	|\Gamma_\rho|&\lesssim |\nabla v||\nabla\rho_\eta|+|v||\nabla^2\rho_\eta|\lesssim \rho_{\eta}|\nabla v|+\rho_{\eta}|v|,
	\end{align*}
	we have by \eqref{Es10} and \eqref{Max}
	\begin{align*}
	\|\Gamma_\rho\|_{\mL^p_T}\lesssim \|\nabla v\rho_\eta\|_{\mL^p_T}+\|v\rho_\eta\|_{\mL^p_T}\lesssim
	\|\nabla^2v\rho_\eta\|_{\mL^p_T}^{1/2}+1.
	\end{align*}
	Combining the above calculations, by Young's inequality, we get
	$$
	\|\p_t (v\rho_\eta)\|_{\mL^{p}_T}+\|v\rho_\eta\|_{\mH^{2,p}_T}\lesssim 1.
	$$
	The result now follows.
\end{proof}

\subsection{Critical one dimensional case}\label{sec:4.3}
In this subsection we consider the critical one dimensional case and prove the following a priori estimate.
\bt\label{Th48}
Let $T>0$ and $\alpha,\delta_1\in(0,1], \delta\in(0,2)$. 
Suppose {\bf (H$^\alpha_1$)}, $B\in \mL_T^\infty(\rho_{\delta_1})$
 and {\bf (H$^{\delta,\beta}_{\rm crit}$)}. 
For any $\eta>2\big(\frac{(1+2\beta)\delta}{\beta}\vee(\delta_1+\delta)\big)$
 and $v_0\in \sC^2(\rho_{\delta})$, there is a $p_0$ large enough 
so that for all $p>p_0$ and any strong solution $v$ of HJB \eqref{HJB},
$$
\|\p_t (v\rho_\eta)\|_{\mL^{p}_T}+\|v\rho_\eta\|_{\mH^{2,p}_T}\leq C,
$$
where $C=C(\Theta,\eta,p,\delta_1,\|B\|_{\mL^\infty_T(\rho_{\delta_1})}, \|v_0\|_{\sC^2(\rho_{\delta})})$.
\et

To prove this result, we first show the following lemma.
\bl\label{Le47}
Under the assumptions of Theorem \ref{Th48},
for any $\eta>\frac{(1+2\beta)\delta}{\beta}\vee(\delta_1+\delta)$, there is a $p_0$ large enough so that for all $p>p_0$
 and any strong solution $v$ of HJB \eqref{HJB},
\begin{align}\label{DQ11}
\|\p_x v\rho_\eta\|_{L^\infty_TL^p}+\int^T_0\!\!\!\int |\p^2_x v|^2|\p_x v|^{p-2}\rho^p_{\eta}\leq C.
\end{align}
\el
\begin{proof}
	Let $p\geq 2$ be fixed, whose value will be determined below. Define
	$$
	w(t,x):=\p_x v(t,x),\ \ \mA^w_p:=\int |\p_x w|^2|w|^{p-2}\rho^p_{\eta}.
	$$
	For $q\in[\frac{p}{2}+1,p+2]$ and $\gamma\in\mR$, by \eqref{KJ3} and \eqref{Max} and 
	$|\nabla\rho_\delta|\lesssim \rho_{\delta+1}$ we have
	\begin{align*}
	\left(\int |w|^q\rho_{p\eta+\gamma}\right)^{1/q}
	&\lesssim \left(\int |\p_x w|^2|w|^{p-2}\rho_{p\eta}\right)^{\frac{1}{p+2}}
	\|v\rho_{\delta_2}\|^{\frac{2}{p+2}}_{L^r}+\|v\rho_{\frac{p\eta+\gamma}{q}+1}\|_{L^q}\\
	&\lesssim\left(\mA^w_p\right)^{\frac{1}{p+2}}\|\rho_{\delta_2-\delta}\|^{\frac{2}{p+2}}_{L^r}
	+\|\rho_{\frac{p\eta+\gamma}{q}+1-\delta}\|_{L^q},
	\end{align*}
	where
	$$
	\delta_2:=\tfrac{(p+2-q)p\eta}{2q}+\tfrac{(p+2)\gamma}{2q},\ \ r:=\tfrac{2q}{p+2-q}\in[2,\infty].
	$$
	Recalling $\rho_\delta(x)=\<x\>^{-\delta}$ and $d=1$, we have for $q=p+2$ and $\gamma=2\delta$, or $q\in[\frac{p}{2}+1,p+2)$ and
	$\gamma>\frac{2q\delta}{p+2}+(1-p\eta)(1-\frac{q}{p+2})=:\gamma_0$,
	$$
	\|\rho_{\delta_2-\delta}\|_{L^r}+\|\rho_{\frac{p\eta+\gamma}{q}+1-\delta}\|_{L^q}<\infty.
	$$
	Thus we always have
	\begin{align}
	\int |w|^q\rho_{p\eta+\gamma}
	\lesssim
	\left\{
	\begin{aligned}
	&\mA^w_p+1,\qquad\qquad q=p+2,\gamma=2\delta,\\
	&(\mA^w_p)^{\frac{q}{p+2}}+1,\quad q\in[\tfrac{p}{2}+1,p+2), \gamma>\gamma_0.&\label{DD7}
	\end{aligned}
	\right.
	\end{align}
	Now by \eqref{HJB}, one sees that
	\begin{equation}\label{eq:z}
	\partial_t w=\partial_x\left(a\cdot \partial_xw\right)+\partial_x(Bw)+\partial_xH(v,w).
	\end{equation}
	Since $\eta>\big(\frac{1+2\beta}{\beta}\big)\delta\vee(\delta_1+\delta)$,
	we can choose  $p$ large enough such that
	\begin{align}\label{ESP}
	\eta>\left(\Big[2\tfrac{p+1}{p}+\tfrac{p+2}{\beta p}\Big]\delta+\tfrac{1}{p}\right)\vee\left((1+\tfrac2p)\delta_1+\tfrac1p+\delta\right).
	\end{align}
	Multiplying both sides of \eqref{eq:z} by $w|w|^{p-2}\rho_{p\eta}$ and integrating on $\mR$, we obtain
	\begin{align*}
	\frac{1}{p}\partial_t\int |w\rho_\eta|^p&=-\int a\p_x w\p_x(w|w|^{p-2}\rho_{p\eta})-\int Bw\p_x(w|w|^{p-2}\rho_{p\eta})\\
	&\quad-\int H(v,w)\p_x(w|w|^{p-2}\rho_{p\eta})\\
	&=:I_1+I_2+I_3.
	\end{align*}
	For $I_1$, since $a\geq c_0$ and $\eta>\frac{1}{p}+\delta$, by \eqref{DD7} with $q=p$ and $\gamma=0$, we have
	\begin{align*}
	I_1&\leq -c_0\int |\p_x w|^2|w|^{p-2}\rho_{p\eta}+C\int |\p_xw||w|^{p-1}\rho_{p\eta}\\
	&\leq-\frac{c_0}{2}\mA^w_p+C\int |w|^p\rho_{p\eta}\leq-\frac{c_0}{4}\mA^w_p+C.
	\end{align*}
	For $I_2$, since $|B|\leq \|B\|_{\mL^\infty_T(\rho_{\delta_1})}\rho^{-1}_{\delta_1}$
	and $\eta>(1+\frac{2}{p})\delta_1+\tfrac{1}{p}+\delta$, by \eqref{DD7} with $q=p$ and $\gamma=-2\delta_1$, we have
	\begin{align*}
	I_2&\lesssim \int |\p_x w||w|^{p-1}\rho_{p\eta-\delta_1}+\int |w|^p\rho_{p\eta+1-\delta_1}\\
	&\lesssim \left(\mA^w_p\right)^{1/2}\left(\int |w|^p\rho_{p\eta-2\delta_1}\right)^{1/2}+\int |w|^p\rho_{p\eta}\\
	&\lesssim \left(\mA^w_p\right)^{(p+1)/(p+2)}+1.
	\end{align*}
	Next comes to the hard term $I_3$. Let $\phi_\eps(y)=\eps^{-1}\phi(y/\eps)$, where $\phi\in C^\infty_c((-1,1))$ is 
	a smooth density function. Define for given $t\in[0,T]$ and $v,Q\in\mR$,
	\begin{align}\label{MOD}
	H_\eps(t,x,v,Q)=\int H(t,y,v,Q)\phi_{\eps\rho_{\delta/\beta}(x)}(x-y)\dif y.
	\end{align}
	We make the following decomposition for $I_3$:
	\begin{align*}
	I_3&=\int (H_\eps(v,w)-H(v,w))\p_x(w|w|^{p-2}\rho_{p\eta})
	\\&\quad-(p-1)\int H_\eps(v,w)\p_xw|w|^{p-2}\rho_{p\eta}
	\\&\quad-\int H_\eps(v,w)w |w|^{p-2}\p_x\rho_{p\eta}
	\\&:=I_{31}-I_{32}-I_{33}.
	\end{align*}
	For $I_{31}$, noting that by \eqref{HH1}, \eqref{MOD} and \eqref{Max},
	\begin{align*}
	&|H_\eps(x,v,w)-H(x,v,w)|\leq \int |H(y,v,w)-H(x,v,w)|\phi_{\eps\rho_{\delta/\beta}(x)}(x-y)\dif y
\\&\qquad\qquad\lesssim \eps^\beta\rho_{\delta}(x)\int (\<x\>^\delta+\<y\>^\delta+|v|^2+|w|^2)\phi_{\eps\rho_{\delta/\beta}(x)}(x-y)\dif y
\\&\qquad\qquad\lesssim \eps^\beta\rho_\delta(x)\Big(\<x\>^\delta+\<x\>^{2\delta}+|w|^2\Big)
\lesssim\rho^{-1}_{\delta}(x)+\eps^\beta\rho_{\delta}(x)|w|^2,
	\end{align*}
	we have
		\begin{align*}
	I_{31}&\lesssim \int\rho^{-1}_{\delta}|\p_x(w|w|^{p-2}\rho_{p\eta})|
	+\eps^\beta\int  \rho_{\delta}w^2 |\p_x(w|w|^{p-2}\rho_{p\eta})|=:I_{311}+I_{312}.
	\end{align*}
	For $I_{311}$, noting that by the chain rule and $|\nabla \rho_{p\eta}|\lesssim \rho_{p\eta+1}$,
	\begin{align}\label{HH3}
	|\p_x(w|w|^{p-2}\rho_{p\eta})|\lesssim |w|^{p-2}|\p_x w|\rho_{p\eta}+|w|^{p-1}\rho_{p\eta+1},
	\end{align}
since $\eta>\frac1p+\delta$, we have by \eqref{DD7},
	\begin{align*}
	I_{311}&\lesssim \int|w|^{p-2}|\p_xw|\rho_{p\eta-\delta}+\int|w|^{p-1}\rho_{p\eta+1-\delta}
	\\&\lesssim \left(\mA^w_p\right)^{1/2}\left(\int |w|^{p-2}\rho_{p\eta-2\delta}\right)^{1/2}+\int |w|^{p-1}\rho_{p\eta+1-\delta}
	\\&\lesssim \left(\mA^w_p\right)^{p/(p+2)}+1,
	\end{align*}
	where we used H\"older's inequality. 
	For $I_{312}$, by \eqref{HH3} and \eqref{DD7} and $\eta>\frac1p+\delta$, we have
	\begin{align*}
	I_{312}&\lesssim\eps^\beta\int|w|^{p}|\p_xw|\rho_{p\eta+\delta}+\int \eps^\beta |w|^{p+1}\rho_{p\eta+1+\delta}
	\\&\lesssim\eps^\beta\int(|w|^{p-2}|\p_xw|^2\rho_{p\eta}+|w|^{p+2}\rho_{p\eta+2\delta})+\int |w|^{p+1}\rho_{p\eta+1+\delta}
	\\&\lesssim \eps^\beta \mA^w_p+\left(\mA^w_p\right)^{(p+1)/(p+2)}+1,
	\end{align*}
	where we used H\"older's inequality and Young's inequality. 
	For $I_{32}$, noting that by the chain rule,
\begin{align*}
&H_\eps(v,w)\p_x w|w|^{p-2}=\p_x\Big(\int_0^{w}H_\eps(v,r)|r|^{p-2}\dif r\Big)
\\&\qquad-\int_0^w(\p_xH_\eps(v,r)+\p_vH_\eps(v,r)w)|r|^{p-2}\dif r,
\end{align*}
by the integration by parts, we have
	\begin{align*}
	I_{32}&\lesssim \int\left(\int_0^{w}|H_\eps(v,r)| |r|^{p-2}\dif r\right)|\p_x\rho_{p\eta}|
	\\&\quad+\int\left(\int_0^w |\p_xH_\eps(v,r)| |r|^{p-2}\dif r\right)\rho_{p\eta}
	\\&\quad+\int\left(\int_0^w|\p_vH_\eps(v,r)w| |r|^{p-2}\dif r\right)\rho_{p\eta}
	\\&=:I_{321}+I_{322}+I_{323}.
	\end{align*}
For $I_{321}$, by \eqref{HH2} and \eqref{DD7} we have
\begin{align*}
I_{321}&\lesssim \int\left(\int_0^{w}(\rho^{-1}_{\delta}+|r|^2) |r|^{p-2}\dif r\right)\rho_{p\eta+1}
\\&\lesssim \int(\rho^{-1}_{\delta}|w|^{p-1}+|w|^{p+1})\rho_{p\eta+1}
\\&\lesssim (\mA_p^w)^{\frac{p+1}{p+2}}+1.
\end{align*}
For $I_{322}$, noting that
$$
|\p_xH_\eps(x,v,w)|\lesssim \eps^{-1}\rho^{-1}_{\delta/\beta}(x)(\<x\>^\delta+w^{2}),
$$
we have
\begin{align*}
I_{322}&\lesssim \eps^{-1}\int(\rho^{-1}_{\delta+\delta/\beta}|w|^{p-1}+\rho^{-1}_{\delta/\beta}|w|^{p+1})\rho_{p\eta}\lesssim (\mA_p^w)^{\frac{p+1}{p+2}}+1.
\end{align*}
where we used $\eta>\Big[2\frac{p+1}p+\frac{p+2}{\beta p}\Big]\delta+\frac1p$ and  \eqref{DD7} with $q=p+1$, $\gamma=-\delta/\beta$ and $q=p-1$, $\gamma=-\delta-\delta/\beta$.

For $I_{323}$, by \eqref{HH2}, \eqref{Max} we have
\begin{align*}
I_{323}&\lesssim \int(|w|^{p+1}+\rho^{-1}_{\delta}|w|^{p})\rho_{p\eta}\lesssim1,
\end{align*}
where we used \eqref{DD7} with $q=p$, $\gamma=-\delta$.
Finally, for $I_{33}$, by \eqref{HH2}, we similarly use \eqref{DD7} to have 
\begin{align*}
I_{33}\lesssim \int (|w|^{p-1}\rho^{-1}_{\delta}+|w|^{p+1})\rho_{p\eta+1}\lesssim (\mA_p^w)^{\frac{p+1}{p+2}}+1.
\end{align*}
Combining the above calculations, choosing $\eps$ small enough and by Young's inequality, we obtain
	$$
	\frac{1}{2}\partial_t\|w\rho_\eta\|^p_{L^p}\lesssim -\frac{c_0}{8}\mA^w_p+1.
	$$
	Integrating both sides from $0$ to $T$, we obtain the desired estimate.
\end{proof}
Now we can give the proof of Theorem \ref{Th48}.
\begin{proof}[Proof of Theorem \ref{Th48}]
	We follow the proof of Theorem \ref{Th47}.
	Fix $p>d/(\eta-\delta)$.
	By the $L^p$-theory of PDEs (cf. \cite{K08}), we have
	$$
	\|\p_t(v\rho_\eta)\|_{\mL^{p}_T}+\|v\rho_\eta\|_{\mH^{2,p}_T}\lesssim_C
	\|H(v,\nabla v)\rho_\eta+(B\cdot\nabla v)\rho_\eta-\Gamma_\rho\|_{\mL^p_T}+\|v_0\rho_\eta\|_{H^{2,p}},
	$$
	with $\Gamma_\rho$ defined in the proof of Theorem \ref{Th47}.
	Since $p>d/(\eta-\delta)$, by $|H(v,Q)|\lesssim \<x\>^\delta+|Q|^2$, we have
	\begin{align*}
	\|H(v,\nabla v)\rho_{\eta}\|_{\mL^p_T}\lesssim \|\rho_{\eta-\delta}\|_{L^p}+\||\nabla v|^2\rho_\eta\|_{\mL^p_T}
	\lesssim 1+\|\nabla v \rho_{\eta/2}\|_{\mL^{2p}_T}^2.
	\end{align*}
	We have by H\"older's inequality and Sobolev's embedding,
	\begin{align*}
	\|\nabla v \rho_{\eta/2}\|_{\mL^{2p}_T}\leq& \|\nabla v\rho_\eta\|_{\mL^\infty_T}^{\theta}\|\nabla v\rho_{\eta_0}\|_{L^\infty_TL^r}^{1-\theta} 
	\\
	\lesssim&\|\nabla(\nabla v\rho_\eta)\|_{\mL^p_T}^{\theta}\|\nabla v\rho_{\eta_0}\|_{L^\infty_TL^r}^{1-\theta}
	+\|\nabla v\rho_\eta\|_{\mL^p_T}^{\theta}\|\nabla v\rho_{\eta_0}\|_{L^\infty_TL^r}^{1-\theta},
	\end{align*}
	where $\theta\in(0,1/2)$ and
	$$
	r=2p(1-\theta),\ \ \eta_0=\tfrac{1-2\theta}{2(1-\theta)}\eta.
	$$
	Let $p_0$ be as in Lemma \ref{Le47}.
	Since $\eta>2\Big(\frac{1+2\beta}{\beta}\delta\vee(\delta_1+\delta)\Big)$, one can choose $\theta$ close to zero and $p$ large enough so that
	$$
	\eta_0=\tfrac{1-2\theta}{2(1-\theta)}\eta>\tfrac{1+2\beta}{2\beta}\delta\vee(\delta_1+\delta),\ \ r,p\geq p_0.
	$$
	Thus by \eqref{DQ11}, we obtain
	$$
	\|\nabla v\rho_{\eta_0}\|_{L^\infty_TL^r}+\|\nabla v\rho_{\eta}\|_{\mL^p_T}\leq C,
	$$
	and therefore,
	$$
	\|H(v,\nabla v)\rho_\eta\|_{\mL^p_T}\leq\eps\|\nabla^2(v\rho_\eta)\|_{\mL^p_T}+C.
	$$
Moreover, as in the proof of Theorem \ref{Th47}, one has
	$$
	\|(B\cdot\nabla v)\rho_\eta-\Gamma_\rho\|_{\mL^p_T}\leq C.
	$$
	Thus we obtain the desired estimate as in the proof of Theorem \ref{Th47}.
\end{proof}

\subsection{Proof of Theorem \ref{Th42}}\label{sec:4.4}

In this subsection we prove Theorem \ref{Th42} by the previous a priori estimates.

{\bf (Existence).} Let $T>0$. For fixed $m\in\mN$, let $\chi^m_n(x):=\chi^m(x/n), n\in\mN$ be the cutoff function in $\mR^m$,
 and $\varrho^m_n(x):=n^m\varrho^m(nx), n\in\mN$ be the mollifiers in $\mR^m$, where $\chi^m\in C^\infty_c(\mR^m)$ with
$\chi^m=1$ for $|x|\leq 1$ and $\chi^m=0$ for $|x|>2$, and $\varrho^m\in C^\infty_c(\mR^m)$ is a density function. 
Define
$$
B_n(t,x):=B(t,x)\1_{|x|\leq n},\ \ \varphi_n(x):=v_0(x)\chi^d_n(x).
$$
For nonlinear term $H$, we construct the approximation $H_n$ as follows:
\begin{equation}\label{zz1}
H_n(t,x,v,Q):=((H(t,x,\cdot,\cdot)\chi^{d+1}_n)*\varrho^{d+1}_n)(v,Q)\chi_n^d(x).
\end{equation}
We consider the following approximation equation:
\begin{align}\label{AHJB}
\p_t v_n=\tr(a\cdot\nabla^2 v_n)+B_n\cdot\nabla v_n+H_n(v_n,\nabla v_n),\ \ v_n(0)=\varphi_n.
\end{align}
Note that by the assumptions of Theorem \ref{Th42},
$$
B_n\in\cap_{p\in[1,\infty]}\mL^p_T,\ \ \varphi_n\in \cap_{p\in[1,\infty]}H^{2,p},
$$
and
$$
\|H_n\|_{\mL_T^\infty}+\|\p_v H_n\|_{\mL_T^\infty}+\|\p_Q H_n\|_{\mL_T^\infty}<\infty.
$$
It is well known that the approximation equation \eqref{AHJB} admits 
a unique strong solution $v_n\in\cap_{p\geq 2} \mH^{2,p}_T$ (cf. \cite{K08}).
Moreover, by definition, we have the following uniform estimates:
$$
\|B_n\rho_{\delta_1}\|_{\mL_T^\infty}\leq\|B\rho_{\delta_1}\|_{\mL_T^\infty},
$$
and for some $C$ independent of $n$, in the subscritical case,
$$
|H_n(v,Q)|\lesssim_C \<x\>^\delta+|Q|^\zeta,
$$
and in the critical case $d=1$,
$$
|H_n(t,x,v,Q)|\lesssim_C \<x\>^\delta+|Q|^2,\ \ |\p_vH_n(t,x,v,Q)|\lesssim_C\<x\>^\delta+|v|+|Q|,
$$
$$
|H_n(t,x,v,Q)-H_n(t,y,v,Q)|\lesssim_C |x-y|^\beta(\<x\>^\delta+\<y\>^\delta+|v|^2+|Q|^2),
$$
Thus by Theorems \ref{Th43}, \ref{Th47} and \ref{Th48}, we have the following uniform estimate:
for $\eta$ being as in \eqref{ETA} and $p$ large enough,
$$
\|v_n\rho_{\delta}\|_{\mL^\infty_T}+\|\p_t(v_n\rho_\eta)\|_{\mL^{p}_T}+\|v_n\rho_\eta\|_{\mH^{2,p}_T}\leq C,
$$
where $C$ is independent of $n$.
By Sobolev's embedding (cf. \cite[Lemma 2.3]{CZ16}), for any $\beta'\in(0,2-\frac{2}{p})$ and $\gamma=1-\frac{\beta'}{2}-\frac{1}{p}$,
\begin{align*}
\|v_n\rho_\eta\|_{C^\gamma_T\bC^{\beta'-d/p}}\lesssim&\|v_n\rho_\eta\|_{C^\gamma_TH^{\beta',p}}
\lesssim \|\p_t(v_n\rho_\eta)\|_{\mL^{p}_T}+\|v_n\rho_\eta\|_{\mH^{2,p}_T}+\|v_0\rho_\eta\|_{H^{\beta',p}}\leq C.
\end{align*}
Thus by Ascolli-Arzela's lemma, there are subsequence $n_k$ and $v\in\mL^\infty_T(\rho_{\delta})\cap\mH^{2,p}_T(\rho_\eta)$ such that
for all $t,x$,
\begin{align}\label{WQ1}
\nabla^j v_{n_k}(t,x)\to \nabla^jv(t,x),\ \ j=0,1,
\end{align}
and for any $R>0$,
\begin{align}\label{WQ2}
\nabla^2v_n\to \nabla^2v\mbox{ weakly in $L^2([0,T]\times B_R)$.}
\end{align}
By taking limits for \eqref{AHJB}, one finds that $v$ is a strong solution to \eqref{HJB} in the sense of Definition \ref{Def41}.
Indeed, for any $\psi\in C^\infty_c(\mR^d)$,  by \eqref{WQ2} we have
$$
\lim_{n\to\infty}\int^t_0 \<\tr(a\cdot\nabla^2 v_n),\psi\>\dif s=\int^t_0 \<\tr(a\cdot\nabla^2 v),\psi\>\dif s
$$
and by \eqref{WQ1} and the dominated convergence theorem,
$$
\lim_{n\to\infty}\int^t_0 \<B_n\cdot\nabla v_n,\psi\>\dif s=\int^t_0 \<B\cdot\nabla v,\psi\>\dif s
$$
Moreover, since for each $(t,x)\in[0,T]\times\mR^d$ and $R>0$,
$$
\lim_{n\to\infty}\sup_{|(v,Q)|\leq R}|H_n(t,x,v,Q)-H(t,x,v,Q)|=0,
$$
by \eqref{WQ1} and the dominated convergence theorem, we also have
$$
\lim_{n\to\infty}\int^t_0 \<H_n(s,\cdot,v_n,\nabla v_n),\psi\>\dif s=\int^t_0 \<H(s,\cdot,v,\nabla v),\psi\>\dif s.
$$
Thus we obtain the existence of a strong solution.

{\bf (Uniqueness).} We prove the uniqueness on the time interval $[0,1]$.
Let $v_1, v_2$ be two strong solutions of HJB \eqref{Def41}with the same initial value $v_0$.
By \eqref{DW1}, we have
\begin{align}\label{WQ4}
v_1, v_2\in \mL^\infty_1(\rho_{\delta})\cap L^\infty_1\sC^1(\rho_\eta).
\end{align}
Let $V:=v_1-v_2$.
Then $V$ is a strong solution of the following linear PDE:
$$
\p_t V=\tr(a\cdot\nabla^2 V)+B\cdot\nabla V+G\cdot\nabla V+K\cdot V,\ V(0)=0,
$$
where
$$
G:=\int^1_0\p_Q H(v_1, \nabla v_1+\theta\nabla(v_2-v_1))\dif\theta,
$$
and
$$
K:=\int^1_0\p_v H(v_1+\theta(v_2-v_1), \nabla v_2)\dif\theta.
$$
By \eqref{WQ4} and \eqref{WQ5}, there is a constant $C_0>0$ such that for all $(t,x)\in[0,1]\times\mR^d$,
\begin{align}\label{DDS}
|G(t,x)|\lesssim_{C_0}\<x\>, \ \ |K(t,x)|\lesssim_{C_0}\<x\>^2.
\end{align}
Let $T\in(0,1]$ be fixed and determined below. For a space-time function $F$, let
$$
F^T(t,x):=F(T-t,x).
$$
Thus under {\bf (H$^\alpha_1$)} and $B\in \mL^\infty_1(\rho_{\delta_1})$, for each $(t,x)\in[0,T]\times\mR^d$,
the following SDE admits a unique weak solution starting from $x$ at time $t$ (see \cite{Kry80}):
$$
X^{t,x}_s=x+\int^s_t\sqrt{2a^T}(r,X^{t,x}_r)\dif W_r+\int^s_t(B^T+G^T)(r,X^{t,x}_r)\dif r,\ \ \forall s\in[t,T].
$$
As in the proof of Theorem \ref{Th43}, by It\^o's formula, we have
\begin{align*}
\e^{\int^{t'}_{t} K^T(s, X^{t,x}_s)\dif s} V^T(t',X^{t,x}_{t'})=V^T(t,x)+M_{t'},\ t'\in[t,T],
\end{align*}
where $M_{t'}$ is a continuous local martingale. Note that by \eqref{DDS} and \cite[Lemma 2.2]{Z10}, for 
$T=T(C_0,d,c_0,\|B\|_{\mL^\infty_1(\rho_{\delta_1})})$ small enough,
$$
\mE\e^{2\int^T_{t} K^T(s, X^{t,x}_s)\dif s}\leq\mE\e^{2C_0\sup_{s\in[t,T]}|X^{t,x}_s|^2}<\infty.
$$
By using stopping time technique as in the proof of Theorem \ref{Th43} and taking expectations, we find that
for $T$ being small enough, $0\leq t\leq T$
$$
V^T(t,x)=\mE\e^{\int^T_{t} K^T(s, X^{t,x}_s)\dif s} V(0,X^{t,x}_T)\equiv 0.
$$
Thus we obtain the uniqueness on small time interval $[0,T]$. We can proceed to consider $[T,2T]$ and so on.
The proof is complete.

\section{HJB equations with distribution-valued coefficients}\label{Zvonkin}

In this section we fix $\alpha\in(\frac{1}{2},\frac{2}{3})$ and $\kappa\in(0,1)$ being small enough so that
\begin{align}\label{DE9}
\bar\alpha:=\alpha+\widetilde{\kappa}\in(\tfrac{1}{2},\tfrac{2}{3}),\ \ \delta:=2(\tfrac{9}{2-3\alpha}+1)\kappa<1,
\end{align}
where $\widetilde\kappa=\kappa^{1/4}$. 
We consider the following singular HJB equation:
\begin{equation} \label{HJB0}
\sL u=\left( \partial_t - \Delta \right) u = b \cdot \nabla u + H(u,\nabla u) + f, \quad u (0) =\varphi, 
\end{equation}
where $(b,f)\in\cap_{T>0}\mB^\alpha_T({\rho_\kappa})$ and 
$$H(t,x,u,Q):\mR^+\times\mR^d\times\mR\times\mR^d\to\mR$$ satisfies 
{\bf (H$^{\delta,\zeta}_{\rm sub}$)} or {\bf (H$^{\delta,\beta}_{\rm crit}$)} with $\zeta\in[0,2)$, $\beta\in(0,1]$ and for some $C>0$,
\begin{align}\label{DDA}
|\p_u H(t,x,u,Q)|+|\p_Q H(t,x,u,Q)|\lesssim_C \<x\>^\delta+|u|+|Q|.
\end{align}

To understand HJB equation \eqref{HJB0}, we write it in the paracontrolled sense:
\begin{equation}\label{eq:ansatz}
u=\nabla u\Prec \sI b+\sI f+u^\sharp,
\end{equation}
where $u^\sharp$ solves the following equation
\begin{equation}\label{eq:sharp}
\left\{
\aligned
\sL u^\sharp&=\nabla u\prec b-\nabla u\Prec b+\nabla u\succ b+b\circ\nabla u
\\&\quad+H(u,\nabla u)-[\sL,\nabla u\Prec ]\sI b,
\\u^\sharp(0)&=u_0,
\endaligned
\right.\end{equation}
with $b\circ\nabla u$ being defined by \eqref{FQ2} for $\lambda=0$.

Our aim of this section is to prove the following result.
\bt\label{existence}  
Let $T>0$, $\beta\in(0,1-\bar\alpha]$, $\zeta\in[0,2)$ and $\alpha,\bar\alpha,\kappa,\delta$ be as in \eqref{DE9}. 
Suppose that $(b,f)\in\mB^\alpha_T({\rho_\kappa})$ and {\bf (H$^{\delta,\zeta}_{\rm sub}$)} or {\bf (H$^{\delta,\beta}_{\rm crit}$)} as well as  \eqref{DDA} 
hold. Let $\eps\in(0,1-\alpha)$ and
\begin{align*}
\left\{
\begin{aligned}
&\eta>\tfrac{2\zeta\delta}{2-\zeta}\vee[2\widetilde\kappa+2\delta],\ &\mbox{\rm under {\bf (H$^{\delta,\zeta}_{\rm sub}$)}};\\
&\eta>2\left[\tfrac{2(1+2\beta)\delta}{\beta}\vee(\widetilde\kappa+2\delta)\right],&\mbox{\rm under {\bf (H$^{\delta,\beta}_{\rm crit}$)}}.
\end{aligned}
\right.
\end{align*}
For any initial value $\varphi\in \bC^{1+\alpha+\eps}(\rho_{\eps\delta})$ for some $\eps>0$,
there is a  paracontrolled solution $(u, u^\sharp)$ solving \eqref{eq:ansatz} and \eqref{eq:sharp} with regularity
$$
u\in \mathbb{S}_T^{2-\bar\alpha}(\rho_\eta)\cap \mL^\infty_T(\rho_{2\delta}),
\quad u^\sharp\in \mathbb{S}_T^{3-2\bar\alpha}(\rho_{2\eta})\cap \mL_T^\infty(\rho_{2\delta+\kappa}).
$$ 
Furthermore, suppose  $\eta<\frac{1-\alpha}2$, the paracontrolled solution $(u,u^\sharp)$ is unique. 
\et

\br 
Since $\kappa$ is arbitary small, $\eta$ could be arbitary small.
\er

To show the existence of a paracontrolled solution, we use the approximation method. More precisely, since $(b,f)\in\mB^\alpha_T(\rho_\kappa)$, 
by the very definition, there is a sequence of $(b_n,f_n)\in L^\infty_T\sC^\infty(\rho_\kappa)$ 
	with 
	$$
	\sup_n\Big(\ell_T^{b_n}(\rho_\kappa)+\mA^{b_n,f_n}_{T,\infty}(\rho_\kappa)\Big)\leq c_0,
	$$
	and such that for  $\lambda\geq 0$,
	\begin{align}\label{LL9}
	\left\{
	\begin{aligned}
&		\lim_{n\to\infty}\Big(\|b_n-b\|_{L_T^\infty\bC^{-\alpha}(\rho_\kappa)}+\|f_n-f\|_{L_T^\infty\bC^{-\alpha}(\rho_\kappa)}\Big)=0,\\
&	\lim_{n\to\infty}\|b_n\circ\nabla\sI_\lambda b_n-b\circ\nabla\sI_\lambda b\|_{L_T^\infty\bC^{1-2\alpha}(\rho_\kappa)}=0,\\
&	\lim_{n\to\infty}\|b_n\circ\nabla\sI_\lambda f_n-b\circ\nabla\sI_\lambda f\|_{L_T^\infty\bC^{1-2\alpha}(\rho_\kappa)}=0.
\end{aligned}
\right.
	\end{align}
	Moreover, let $\varphi_n$ be the mollifying approximation of $\varphi$ so that
	$$
	\sup_n\|\varphi_n\|_{\bC^{1+\alpha+\eps}(\rho_{\eps\delta})}\lesssim\|\varphi\|_{\bC^{1+\alpha+\eps}(\rho_{\eps\delta})}.
	$$
We consider the following approximation equation:
\begin{align}\label{APP}
\sL u_n = b_n \cdot \nabla u_n + H(u_n,\nabla u_n) + f_n, \quad u_n(0) = \varphi_n. 	
\end{align}
By Theorem \ref{Th42}, it is well known that approximation equation \eqref{APP} admits a unique strong solution $u_n$ with
$$
\|u_n\|_{\mL^\infty_T(\rho_{\delta})}+\|\p_tu_n\|_{\mL^{p}_T(\rho_\eta)}+\|u_n\|_{\mH^{2,p}_T(\rho_\eta)}\leq C_n.
$$
Our aim is of course to establish the following uniform estimate:
\begin{align}\label{UNI}
\sup_n\Big(\|u_n\|_{\mathbb{S}_T^{2-\bar\alpha}(\rho_\eta)}+\|u_n\|_{\mL^\infty_T(\rho_{2\delta})}
+\|u^\sharp_n\|_{\mS_T^{3-2\bar\alpha}(\rho_{2\eta})}+\|u^\sharp_n\|_{\mL_T^\infty(\rho_{2\delta+\kappa})}\Big)\leq C,
\end{align}
where $u^\sharp_n$ is defined by \eqref{eq:ansatz} with $(b,f)$ being replaced by $(b_n,f_n)$.

To show the uniform estimate \eqref{UNI}, our approach is to transform \eqref{APP} into  HJB equation 
studied in Section \ref{s:HJB}. In the following, for simplicity, we shall drop the subscript $n$ and use the convention that
all the constants appearing below only depend on the parameter set
$$
\Theta:=(T, d,\alpha,\beta,\eta,\zeta,\kappa,c_0,\eps,\|\varphi\|_{\bC^{1+\alpha+\eps}(\rho_{\eps\delta})}).
$$ 
First of all, by Lemma \ref{lem:local2}, one can make the following 
decomposition for the initial value $\varphi\in \bC^{1+\alpha+\eps}(\rho_{\eps\delta})$: for $\eps_0\in(0,\eps/4)$,
$$
\varphi=\varphi_1+\varphi_2,\ \ \varphi_1\in \bC^{1+\alpha+\eps_0},\ \varphi_2\in\sC^2(\rho_\delta).
$$
Next we make the following decomposition for $u$:
$$
u=u_1+u_2,
$$
where $u_1$ solves the following linear equation with non-homeogeneous term $f$
\begin{equation} \label{L1}
\sL u_1 = b \cdot \nabla u_1  + f, \quad u_1(0)=\varphi_1,
\end{equation}
while $u_2$ solves the following HJB equation 
\begin{equation} \label{L2}
\sL u_2 = b \cdot \nabla u_2  + H(u_1+u_2,\nabla u_1+\nabla u_2), \quad u_2(0)=\varphi_2.
\end{equation}
Clearly, the linear equation \eqref{L1} can be uniquely solved by Theorem \ref{Th33} with the solution $u_1\in \mS_T^{2-\alpha}(\rho_\delta)$.
Thus it remains to solve \eqref{L2}. However, since $b$ is a distribution, to obtain the a priori estimate, we can not directly use Theorem \ref{Th42}.
We shall use \eqref{dec} and Zvonkin's transformation to kill the bad part of $b$. 

\subsection{Zvonkin's transformation for HJB equations}\label{sec:sch}

In this subsection we introduce a transformation of phase space to kill the distributional part in the drift of HJB equation  \eqref{L2}
so that we are in the situation of Section \ref{s:HJB}.
Such a transformation was first used by Zvonkin in \cite{Z74} to study the SDE with bad drifts. 
In the literature, it is also called Zvonkin's transformation.
Below we always assume 
\begin{align}\label{MN2}
b\in L^\infty_T(\sC^\infty(\rho_\kappa)),\ \ell^b_T(\rho_\kappa)\leq c_0.
\end{align}
Let us first recall the following decomposition introduced in \eqref{dec}:
$$
b = b_{>} + b_{\leqslant}:=\VV_>b+\VV_\leq b.
$$ 
Furthermore, we define
\begin{align}\label{b1}
\bar{b}:=b_{>}\circ \nabla \sI_\lambda(b_>),\ \ \bar{b}_>&:=\VV_>\bar{b},\quad \bar{b}_\leq:=\VV_\leq \bar{b}.
\end{align}
\bl\label{Le53}
For any $m\in\mN$ and $\eps>0$, it holds that
$$
b_>\in L^\infty_T\sC^m,\ \ \bar b_\leq\in L^\infty_T\sC^m(\rho_{2\kappa+\eps}).
$$
For some $C=C(d,\alpha,\kappa)>0$, it holds that
\begin{align}\label{eq:b1}
\| b_{>}\|_{L^\infty_T \bC^{-\alpha-\widetilde\kappa}}+\|b_{\leqslant} \|_{\mL^\infty_T (\rho_{\widetilde\kappa})}\lesssim_C \sqrt{\ell_T^b(\rho_{\kappa})},
\end{align}
where $\tilde\kappa=\kappa^{1/4}$, and
\begin{align}\label{e:xi}
\|\bar{b}\|_{L^\infty_T\bC^{1-2\alpha}(\rho_{\widetilde\kappa})}+
\|\bar{b}_>\|_{L^\infty_T\bC^{1-2\alpha-\widetilde\kappa}}
+\|\bar{b}_\leq\|_{\mL^\infty_T(\rho_{\widetilde\kappa})}\lesssim_C \ell_T^b(\rho_\kappa).
\end{align}
\el
\begin{proof}
	(i) The first result follows by Lemma \ref{lem:local2}.
	\\
	\\
	(ii) We use Lemma \ref{lem:local2} for weight $\rho_{\kappa^{1/2}}$ to conclude
	$$
	\| b_{>}\|_{L^\infty_T \bC^{-\alpha-\widetilde\kappa}}
	\lesssim \| b_{>}\|_{L^\infty_T \bC^{-\alpha-\kappa^{1/2}}}
	\lesssim \| b \|_{L^\infty_T\bC^{-\alpha} (\rho_\kappa)}\leq \sqrt{\ell_T^b(\rho_{\kappa})}.
	$$
	Since $\alpha<1$, we can choose $\eps>0$ being small enough so that
	$$
	\bar\kappa:=\kappa+\kappa^{1/2}(\alpha+\eps)\leq \kappa^{1/2}-\kappa<\tfrac23\widetilde{\kappa}-\kappa.
	$$
	Noting that 
	$$
	\rho_{\bar\kappa}(x)=\<x\>^{-\kappa^{1/2}(\kappa^{1/2}+\alpha+\eps)}=\rho^{\kappa^{1/2}+\alpha+\eps}_{\kappa^{1/2}}(x),
	$$
	by Lemma \ref{lem:local2} again, we have
	\begin{align*}
	\|b_{\leqslant} \|_{\mL^\infty_T (\rho_{\widetilde\kappa})}\leq\|b_{\leqslant} \|_{\mL^\infty_T (\rho_{\bar\kappa})}
	&=\|b_{\leqslant} \|_{\mL^\infty_T (\rho^{\kappa^{1/2}+\alpha+\eps}_{\kappa^{1/2}})}
	\\&\lesssim \| b \|_{L^\infty_T\bC^{-\alpha} (\rho^{\kappa^{1/2}}_{\kappa^{1/2}})}=\| b \|_{L^\infty_T\bC^{-\alpha} (\rho_\kappa)}.
	\end{align*}
	(iii)  Note that by definition \eqref{b1},
	$$
	\bar b=b\circ\nabla \sI_\lambda b- b\circ \nabla \sI_\lambda(b_{\leqslant}) - b_{\leqslant}\circ\nabla\sI_\lambda(b_>)
	$$
	and
	$$
	\|b\circ\nabla \sI_\lambda b\|_{L^\infty_T\bC^{1-2\alpha}(\rho_{2\kappa})}\leq \ell_T^b(\rho_\kappa).
	$$
	By \eqref{GZ2}, \eqref{EG01} and \eqref{eq:b1}, we have for $\eps\in(0,1-\alpha)$,
	\begin{align*}
	\|b\circ\nabla \sI_\lambda(b_{\leqslant})\|_{L^\infty_T\bC^0(\rho_{\kappa+\bar\kappa})}
	&\lesssim\|b\|_{L^\infty_T\bC^{-\alpha}(\rho_\kappa)}\|b_{\leqslant}\|_{L_T^\infty\bC^{\alpha+\eps-1}(\rho_{\bar\kappa})} 
	\lesssim \ell_T^b(\rho_\kappa),
	\end{align*}
	and 
	\begin{align*}
	\|b_\leq\circ\nabla \sI_\lambda(b_{>}) \|_{L^\infty_T\bC^{1-\alpha-\widetilde\kappa}(\rho_{\bar\kappa})}
	&\lesssim\|b_\leq\|_{\mL_T^\infty(\rho_{\bar\kappa})}\|b_>\|_{L^\infty_T\bC^{-\alpha-\widetilde\kappa}}
	\lesssim \ell_T^b(\rho_\kappa).
	\end{align*}
	Combining the above estimate we get
	$$
	\|\bar{b}\|_{L^\infty_T\bC^{1-2\alpha}(\rho_{\widetilde\kappa})}\lesssim\|\bar{b}\|_{L^\infty_T\bC^{1-2\alpha}(\rho_{\kappa+\bar\kappa})}\lesssim \ell_T^b(\rho_\kappa).
	$$
	(iii) As for the other two estimates in  \eqref{e:xi}, we use Lemma \ref{lem:local2} for weight $\rho_{\widetilde\kappa}$ to have
	\begin{align*}
	\|\bar b_>\|_{L_T^\infty\bC^{1-2\alpha-\widetilde\kappa}}\leq
		\|\bar b_>\|_{L_T^\infty\bC^{1-2\alpha-\frac{\bar\kappa+\kappa}{\widetilde\kappa}}}\lesssim\|\bar{b}\|_{L^\infty_T\bC^{1-2\alpha}(\rho_{\kappa+\bar\kappa})}\lesssim \ell_T^b(\rho_\kappa),
	\end{align*}
	and for $\eps>0$ small enough
	\begin{align*}
	\|\bar b_\leq\|_{\mL_T^\infty(\rho_{\widetilde\kappa})}\leq
	\|\bar b_\leq\|_{\mL_T^\infty(\rho_{{\bar\kappa+\kappa+\widetilde\kappa(2\alpha-1+\eps)}})}\lesssim\|\bar{b}\|_{L^\infty_T\bC^{1-2\alpha}(\rho_{\kappa+\bar\kappa})}\lesssim \ell_T^b(\rho_\kappa).
	\end{align*}
	Now we complete the proof. 
\end{proof}

Now we consider the following vector-valued parabolic equation:
\begin{align}\label{e:phi}
\sL_\lambda \u= (b_{>} -\bar{b}_\leq) \cdot(\nabla \u+\mI),\quad \u(0)={\bf 0}\in\mR^d.
\end{align}
\br\label{re}\rm
The reason of considering $b_{>} -\bar{b}_\leq$ rather than $b_>$ is the following: 
in order to use Lemma \ref{Le11a} to construct a $C^1$-diffemorphism, we have to require $\ell^{b_>}_T(1)<\infty$.
However, by \eqref{e:xi}, $\bar b=b_{>}\circ\nabla \sI_\lambda(b_>)$ only stays a priorily in a weight space.
Thus the term $\bar{b}_\leq$ will be used to cancel the weight term in renormalizing $b_>\circ\nabla\u$. 
\er

Notice that by (i) of Lemma \ref{Le53}, the above equation admits a unique smooth solution $\u$.
Here our aim is to show the following a priori regularity estimate for $\u$ so that $\u$ stays in an unweighted Besov space. 

\bl\label{lem:phi} 
Let $\alpha\in(\frac12,\frac23)$ and $\kappa\in(0,(\frac23-\alpha)^4)$.
Under \eqref{MN2}, for $\bar\alpha=\alpha+\widetilde\kappa$,
there exist $\lambda=\lambda(\Theta)$ large enough and $C=C(\Theta)>0$ such that 
\begin{align}\label{L5}
\|\u\|_{\sC^1}\leq 1/2,\ \ \| \u \|_{\mathbb{S}^{2-\bar\alpha}_T}\leq C.
\end{align}
\el
\begin{proof}
	We use the paracontrolled ansatz as  in \eqref{DT11} and write
	$$\u=\nabla \u\Prec \sI_\lambda(b_>)+\sI_\lambda(b_>)+\u^\sharp,$$
	where
	$$
	\u^\sharp:=\sI_\lambda\big(\nabla \u\prec b_>-\nabla \u\Prec b_>+\nabla \u\succ b_>
	+\Gamma^b_\u-[\sL_\lambda,\nabla \u \Prec]\sI_\lambda (b_>)\big)
	$$ 
	with
	$$
	\Gamma^b_\u:=b_>\circ \nabla \u-\bar{b}_\leq\cdot(\nabla \u+\mI).
	$$
	Note that as in \eqref{FQ2}, 
	\begin{align*}
	\Gamma^b_\u&=b_>\circ(\nabla^2 \u\prec \sI_\lambda(b_>))+\mathrm{com}(\nabla \u,\nabla \sI_\lambda(b_>),b_>)
	\\&\quad+\mathrm{com}_1+b_>\circ \nabla \u^\sharp+\bar{b}_>\cdot(\nabla \u+\mI),
	\end{align*}
	where
	$$\mathrm{com}_1:=b_>\circ\nabla[\nabla \u\Prec \sI_\lambda (b_>)-\nabla \u\prec \sI_\lambda (b_>)].$$
	Let 
	$$
	\gamma,\beta\in(\bar\alpha,2-2\bar\alpha].
	$$
	Except for the last term $\bar{b}_>\cdot(\nabla \u+\mI)$, we estimate each term of $\Gamma^b_\u$ as in Lemma \ref{Le32} and obtain
	\begin{align*}
	\|\Gamma^b_\u\|_{L_T^\infty\bC^{1-2\bar\alpha}}
	&\lesssim \|b_>\|^2_{L^\infty_T\bC^{-\bar\alpha}}\|\u\|_{\mS_T^{\bar\alpha+\gamma}}
	+\|b_>\|_{L^\infty_T\bC^{-\bar\alpha}}\|\nabla\u^\sharp\|_{L^\infty_T\bC^\beta}
	\\&\quad	+\|\bar{b}_>\cdot(\nabla \u+\mI)\|_{L_T^\infty\bC^{1-2\bar\alpha}}
	\\&\lesssim\ell_T^b(\rho_\kappa)\Big(\|\u\|_{\mathbb{S}_T^{\bar\alpha+\gamma}}+1\Big)
	+\sqrt{\ell_T^b(\rho_\kappa)}\|\u^\sharp\|_{L_T^\infty\bC^{\beta+1}}+1,
	\end{align*}
	where we use \eqref{eq:b1}, \eqref{e:xi} and \eqref{GZ3}. 
	As in Lemma \ref{Le11a}, for any $\theta\in(1+\frac{3\bar\alpha}{2}, 2)$, there is a $C>0$ independent of $\lambda$ such that
	for all $\lambda\geq 1$,
	\begin{align*}
	\lambda^{1-\frac{\theta}{2}}\|\u\|_{\mathbb{S}_T^{\theta-\bar\alpha}}\leq C.
	\end{align*}
	Taking $\lambda$ being large enough, we get the first desired estimate. Then as in Lemma \ref{Le11a}  we obtain the second estimate. 
\end{proof}

Now, let us define
$$
\Phi(t,x):=x+\u(t,x).
$$
By Lemma \ref{lem:phi}, it is easy to see that for each $t\in[0,T]$ and $x,y\in\mR^d$,
\begin{align}\label{L4}
\tfrac12|x-y|\leq |\Phi(t,x)-\Phi(t,y)|\leq\tfrac32|x-y|
\end{align}
and
\begin{align}\label{L3}
\p_t\Phi=\Delta\Phi-\lambda\u+ (b_{>} -\bar{b}_\leq)\cdot\nabla\Phi.
\end{align}
In particular,
$$
x\mapsto\Phi(t,x)\mbox{ is a $C^1$-diffeomorphism.}
$$
Let $\Phi^{-1}(t,x)$ be the inverse of $x\mapsto \Phi(t,x)$ and define
$$
v(t,x):=u_2(t,\Phi^{-1}(t,x))\Rightarrow v(t,\Phi(t,x))=u_2(t,x),
$$
where $u_2$ solves HJB equation \eqref{L2}. 

In the rest of this subsection, with a little of confused notations, we also use $\circ$ to denote the composition of two functions.  
By the chain rule, we have
$$
\p_t v\circ\Phi+\p_t\Phi\cdot(\nabla v\circ\Phi)=\p_t u_2,\ \ \nabla u_2=\nabla\Phi\cdot(\nabla v\circ\Phi)
$$
and
$$
\Delta u_2=\Delta\Phi\cdot(\nabla v\circ\Phi)+\tr(\widetilde a\cdot\nabla^2 v\circ\Phi),
$$
where $\widetilde a_{ij}:=\sum_{k=1}^d(\p_k\Phi^i\p_k\Phi^j)$, which implies by \eqref{L2} and \eqref{L3} that
\begin{align*}
(\p_t v)\circ\Phi&=\tr(\widetilde a\cdot\nabla^2 v\circ\Phi)+H(u_1+u_2,\nabla u_1+\nabla u_2)
\\&\quad+((b_\leq+\bar b_\leq)\cdot\nabla\Phi+\lambda\u)\cdot(\nabla v\circ\Phi).
\end{align*}
Thus we obtain the following key lemma for solving HJB equation \eqref{L2}.
\bl
\label{Zvon} 
The $v$ defined above solves the following HJB equation:
\begin{equation}\label{eq:51}\aligned
\partial_t v=\mathrm{tr}\left(a\cdot \nabla^2v\right)+B\cdot\nabla v+\widetilde H(v,\nabla v),\quad v(0)=\varphi_2,
\endaligned\end{equation}
where $a_{ij}:=\sum_{k=1}^d(\p_k\Phi^i\p_k\Phi^j)\circ \Phi^{-1}$ and
$$
B:=((b_\leq +\bar{b}_\leq)\cdot\nabla \Phi+\lambda\u)\circ\Phi^{-1},
$$
and for $(t,x,v,Q)\in[0,T]\times\mR^d\times\mR\times\mR^d$,
$$
\widetilde H(t,x,v,Q):=H\big(t,\cdot, u_1(t,\cdot)+v, \nabla u_1(t,\cdot)+\nabla\Phi(t,\cdot)\cdot Q\big)\circ\Phi^{-1}(t,x).
$$
Moreover, $a$ satisfies {\bf (H$^{1-\bar\alpha}_a$)}, $B\in \mL^\infty_T(\rho_{\widetilde\kappa})$, and
under {\bf (H$^{\delta,\zeta}_{\rm sub}$)} or {\bf (H$^{\delta,\beta}_{\rm crit}$)} for $\beta\leq 1-\bar\alpha$, $\widetilde H$ still satisfies 
{\bf (H$^{2\delta,\zeta}_{\rm sub}$)} or {\bf (H$^{2\delta,\beta}_{\rm crit}$)}.
\el
\begin{proof}
	(i) By \eqref{L4} and \eqref{L5}, we have $\frac{1}{2}\mI\leq\widetilde a\leq 2\mI$ and
	\begin{align*}
	| a(t,x)- a(t,y)|&\lesssim|\nabla\u(t,\Phi^{-1}(t,x))-\nabla\u(t,\Phi^{-1}(t,y))|
	\\&\lesssim |\Phi^{-1}(t,x)-\Phi^{-1}(t,y)|^{1-\bar\alpha}\lesssim|x-y|^{1-\bar\alpha}.
	\end{align*}
	(ii) Note that for some $C\geq 1$,
	\begin{align}\label{L6}
	C^{-1}\<x\>\leq \<\Phi(t,x)\>\leq C\<x\>,\ \forall t\in[0,T].
	\end{align}
	The assertion $B\in \mL^\infty_T(\rho_{\widetilde\kappa})$ follows by \eqref{L5} and Lemma \ref{Le53}.
	\medskip\\
	(iii) We only check that under {\bf (H$^{\delta,\beta}_{\rm crit}$)},
	$\widetilde H$ satisfies {\bf (H$^{2\delta,\beta}_{\rm crit}$)}. For simplicity, we drop the time variable.
	By \eqref{HH2}, we have
	\begin{align*}
&	|H\big(x, u_1(x)+v, \nabla u_1(x)+\nabla\Phi(x)\cdot Q\big)|
\\	&\quad\leq c_2\<x\>^\delta+c'_3(|Q|^2+|\nabla u_1(x)|^2)
\leq c'_2\<x\>^{2\delta}+c'_3|Q|^2,
	\end{align*}
	where we used $u_1\in \mS^{2-\alpha}_T(\rho_\delta)$.
	By \eqref{HH1} and \eqref{DDA}, we have for $|x-y|\leq 1$, $\beta\leq 1-\bar\alpha$
	\begin{align*}
	&|H\big(x, u_1(x)+v, \nabla u_1(x)+\nabla\Phi(x)\cdot Q\big)-H\big(y, u_1(y)+v, \nabla u_1(y)+\nabla\Phi(y)\cdot Q\big)|
	\\&\quad\lesssim |x-y|^\beta\Big(\<x\>^\delta+\<y\>^\delta+|u_1(x)+v|^2+|\nabla u_1(x)+\nabla\Phi(x)\cdot Q|^2\Big)
	\\&\qquad+|u_1(x)-u_1(y)|\Big(\<y\>^\delta+|v|+|u_1(x)|+|u_1(y)|+|\nabla u_1(x)|+|Q|\Big)
	\\&\quad\qquad+(|\nabla u_1(x)-\nabla u_1(y)|+|\nabla\Phi(x)-\nabla\Phi(y)||Q|)
	\\&\quad\qquad\times (\<y\>^\delta+|u_1(y)|+|v|+|\nabla u_1(x)|+|\nabla u_1(y)|+|Q|)
	\\&\quad\lesssim |x-y|^{\beta}(\<x\>^{2\delta}+\<y\>^{2\delta}+|v|^2+|Q|^2).
	\end{align*}
	Furthermore, we have
	\begin{align*}
	&|\p_v H\big(x, u_1(x)+v, \nabla u_1(x)+\nabla\Phi(x)\cdot Q\big)|
\\&\quad	\lesssim \<x\>^\delta+ |u_1(x)|+|v|+|\nabla u_1(x)|+|Q|\lesssim\<x\>^\delta+|v|+|Q|.
	\end{align*}
Therefore, $\widetilde H$ satisfies {\bf (H$^{2\delta,\beta}_{\rm crit}$)} by definition and \eqref{L4}, \eqref{L6}.
\end{proof}	
	
\subsection{Proof of Theorem \ref{existence}}\label{sec:5.2}


We first use Lemma \ref{Zvon} and Theorem \ref{Th42} to derive the following a priori estimate.
\bl
Under \eqref{MN2}, there is a constant $C=C(\Theta)>0$ such that
\begin{align}\label{BZ11}
\|u\|_{\mL^\infty_T(\rho_{2\delta})}+\|u\|_{\mathbb{S}_T^{2-\bar\alpha}(\rho_\eta)}\leq C.
\end{align}
\el
\begin{proof}
By \eqref{L1}, \eqref{L2} and Theorem \ref{Th33}, it suffices to prove that
\begin{align}\label{BZ1}
\|u_2\|_{\mL^\infty_T(\rho_{2\delta})}+\|u_2\|_{\mathbb{S}_T^{2-\bar\alpha}(\rho_\eta)}\lesssim 1.
\end{align}
By Lemma \ref{Zvon} and Theorem \ref{Th42}, for any $p$ large enough and $\eta$ depending on $\kappa, \alpha,\delta$
\begin{align}\label{bd:v}
\|v\|_{\mL_T^\infty(\rho_{2\delta})}+
\|\p_tv\|_{\mL_T^p(\rho_\eta)}+\|v\|_{\mH^{2,p}_T(\rho_\eta)}\lesssim 1.
\end{align}
which implies by  \cite[Lemma 2.3]{CZ16},
\begin{align}\label{bd:v0}
\|v\|_{C_T^{(2-\bar\alpha)/2}L^\infty(\rho_\eta)}\lesssim1.
\end{align}
By \eqref{L6}, we have
	\begin{align*}
	\|u_2\|_{\mL^\infty_T(\rho_{2\delta})}=\|v(\Phi)\rho_{2\delta}\|_{\mL^\infty_T}
	\asymp\|v(\Phi)\rho_{2\delta}(\Phi)\|_{\mL^\infty_T}=\|v\|_{\mL^\infty_T(\rho_{2\delta})},
	\end{align*}
	and by \eqref{GZ3}, \eqref{bd:v} and \eqref{L5},
	\begin{align*}
	&\|\nabla u_2\|_{L^\infty_T\bC^{1-\bar\alpha}(\rho_\eta)}
	=\|\nabla v\circ\Phi \cdot\nabla\Phi\|_{L^\infty_T\bC^{1-\bar\alpha}(\rho_\eta)}
	\\&\qquad\lesssim\|\nabla v(\Phi)\|_{L^\infty_T\bC^{1-\bar\alpha}(\rho_\eta)}
	\|\nabla\Phi\|_{L^\infty_T\bC^{1-\bar\alpha}}
	\\&\qquad\lesssim\|\nabla v\|_{L^\infty_T\bC^{1-\bar\alpha}(\rho_\eta)}
	\|\u\|_{L^\infty_T\bC^{2-\bar\alpha}}\lesssim 1.
	\end{align*}
	Here we used \eqref{eq} and \eqref{L6}, \eqref{L4} to conclude that for $|x-y|\leq 1$,
	\begin{align*} \rho_\eta(x)|\nabla v(\Phi(x))-\nabla v(\Phi(y))|&\lesssim\rho_\eta(\Phi(x))|\nabla v(\Phi(x))-\nabla v(\Phi(y))|
	\\&\lesssim |\Phi(x)-\Phi(y)|^{1-\bar\alpha}\|\nabla v\|_{L^\infty_T\bC^{1-\bar\alpha}(\rho_\eta)}.
	\end{align*}
	Moreover, note that by \eqref{L6},
	\begin{align*}
	\|u_2(t)-u_2(s)\|_{L^\infty(\rho_\eta)}
	&\lesssim \|v(t,\Phi(t))-v(t,\Phi(s))\|_{L^\infty(\rho_\eta)}+\|v(t)-v(s)\|_{L^\infty(\rho_\eta)}\\
	&\leq \|\Phi(t)-\Phi(s)\|_{L^\infty}\int^1_0\|\nabla v(t,\Gamma^{t,s}_r)\|_{L^\infty(\rho_\eta)}\dif r\\
	&\quad+\|v(t)-v(s)\|_{L^\infty(\rho_\eta)},
	\end{align*}
	where $\Gamma^{t,s}_r(x):=r\Phi(t,x)+(1-r)\Phi(s,x)$.
	Since for any $r\in[0,1]$ and $t,s\in[0,T]$,
	$$
	\Gamma^{t,s}_r(x)=x+r\u(t,x)+(1-r)\u(s,x),
	$$
	by \eqref{L5}, we have
	$$
	\rho_\eta(\Gamma^{t,s}_r(x))\asymp\rho_\eta(x).
	$$
	Hence, by \eqref{L5} and \eqref{bd:v0},
	$$
	\frac{\|u_2(t)-u_2(s)\|_{L^\infty(\rho_\eta)}}{|t-s|^{(1-\bar\alpha)/2}}\lesssim 1.
	$$
	Combining the above estimates, we obtain \eqref{BZ1}. The proof is complete.
\end{proof}
Next we apply \eqref{BZ11}, \eqref{eq:ansatz} and \eqref{eq:sharp} to derive the following a priori estimate for $u^\sharp$ as done in Lemma \ref{Le32}.
\bl
Under \eqref{MN2}, there is a constant $C=C(\Theta)>0$ such that
\begin{align}\label{BZ2}
\|u^\sharp\|_{\mL_T^\infty (\rho_{2\delta+\kappa})}+\|u^\sharp\|_{\mathbb{S}_T^{3-2\bar\alpha}(\rho_{2\eta})}\leq C.
\end{align}
\el
\begin{proof}
	First of all, by \eqref{eq:ansatz} and \eqref{BZ11}, we have
	\begin{equation}\label{b:sharp}
	\|u^\sharp\|_{\mL_T^\infty (\rho_{2\delta+\kappa})}
	+\|u^\sharp\|_{L^\infty_T\bC^{2-\bar\alpha}(\rho_{\eta+\kappa})}\lesssim1.
	\end{equation}
	Next we estimate each term on the right hand side of \eqref{eq:sharp} by using Lemma \ref{lem:para}.
	\begin{enumerate}[$\bullet$]
		\item By \eqref{GA4}, \eqref{*}, 
		and $\bar\alpha=\alpha+\widetilde\kappa$, we have
		$$
		\|\nabla u\prec b-\nabla u\Prec b\|_{L^\infty_T\bC^{1-2\bar\alpha}(\rho_{\eta+\kappa})}\lesssim 
		\|u\|_{\mS_T^{2-\bar\alpha}(\rho_{\eta})}\|b\|_{L^\infty_T\bC^{-\alpha}(\rho_{\kappa})}\lesssim 1.
		$$
		\item By \eqref{GZ1} we have
		$$
		\|\nabla u\succ b\|_{L_T^\infty\bC^{1-2\bar\alpha}(\rho_{\eta+\kappa})}\lesssim 
		\|u\|_{L_T^\infty\bC^{2-\bar\alpha}(\rho_{\eta})}\|b\|_{L_T^\infty\bC^{-\alpha}(\rho_\kappa)}\lesssim1.
		$$
		\item By \eqref{GA3} and \eqref{EG01} we have
		$$
		\|[\sL,\nabla u\Prec]\sI b\|_{L_T^\infty\bC^{1-2\bar\alpha}(\rho_{\eta+\kappa})}\lesssim 
		\|u\|_{\mS_T^{2-\bar\alpha}(\rho_{\eta})}\|b\|_{L_T^\infty\bC^{-\alpha}(\rho_\kappa)}\lesssim 1.
		$$
		\item By the growth of $H$ and \eqref{BZ11}, we have
		\begin{align*}
		\|H(u,\nabla u)\|_{\mL^\infty_T(\rho_{2\eta})}
		\lesssim 1+\|\nabla u\|_{\mL^\infty_T(\rho_{\eta})}^2\lesssim 1.
		\end{align*}
		\item By Lemma \ref{Le32} with $\gamma=2-2\bar\alpha$, $\beta\in(\bar\alpha,2-2\bar\alpha)$, we have
		\begin{align*}
		\|b\circ\nabla u\|_{L_T^\infty\bC^{1-2\bar\alpha}(\rho_{2\eta})}\lesssim \|u\|_{\mS_T^{2-\bar\alpha}(\rho_{2\eta-2\kappa})}
		+\|u^\sharp\|_{L_T^\infty\bC^{\beta+1}(\rho_{2\eta-\kappa})}+1,
		\end{align*}
		and by interpolation inequality \eqref{DQ1} with $\theta=\frac{\eta-2\kappa}{\eta-\kappa}$, \eqref{b:sharp} and Young's inequality,
		\begin{align*}
		\|u^\sharp\|_{L_T^\infty\bC^{\beta+1}(\rho_{2\eta-\kappa})}
		&\lesssim \|u^\sharp\|_{L_T^\infty\bC^{3-2\bar\alpha}(\rho_{2\eta})}^{\theta}
		\|u^\sharp\|_{L^\infty_T\bC^{2-\bar\alpha}(\rho_{\eta+\kappa})}^{1-\theta}
		\\&\lesssim \eps\|u^\sharp\|_{L_T^\infty\bC^{3-2\bar\alpha}(\rho_{2\eta})}+1,
		\end{align*}
		where we choose $\beta$ such that $\beta\leq (1-\bar\alpha)(\theta+1)$ since $\kappa$ is small enough.
	\end{enumerate}
	Combining the above calculations and by \eqref{EG1} with $\theta=2$ and $q=\infty$, we obtain
	\begin{align*}
	\|u^\sharp\|_{\mathbb{S}_T^{3-2\bar\alpha}(\rho_{2\eta})}\lesssim \eps\|u^\sharp\|_{L_T^\infty\bC^{3-2\bar\alpha}(\rho_{2\eta})}+1,
	\end{align*}
	which in turn implies the desired estimate.
\end{proof}

Now we are in a position to give
\begin{proof}[Proof of Theorem \ref{existence}]
{\bf (Existence)} 
	By \eqref{BZ11} and \eqref{BZ2}, we obtain the uniform estimate \eqref{UNI}. Now by Ascoli-Arzel\`{a}'s lemma, there are a subsequence
	still denoted by $n$ and 
	$$
	(u, u^\sharp)\in \mathbb{S}_T^{2-\bar\alpha}(\rho_{\eta})	\times \mathbb{S}_T^{3-2\bar\alpha}(\rho_{2\eta})
	$$
	such that for each $\gamma>0$,
	$$
	(u_n, u^\sharp_n)\to (u, u^\sharp)\ \mbox{\rm in }\mathbb{S}_T^{2-\bar\alpha-\gamma}(\rho_{\eta+\gamma})
	\times \mathbb{S}_T^{3-2\bar\alpha-\gamma}(\rho_{2\eta+\gamma}).
	$$
By \eqref{LL9} and taking weak limits for approximation equation 
\eqref{eq:ansatz} and \eqref{eq:sharp} with $(b,f)$ being replaced by $(b_n,f_n)$, one sees that $(u, u^\sharp)$ 
solves \eqref{eq:ansatz} and \eqref{eq:sharp} (see \cite{GH18} for more details).
	\medskip\\
	{\bf (Uniqueness)} 
	 Let $u, \bar u$ be two  paracontrolled solutions to \eqref{HJB0} in the sense of Theorem \ref{existence} starting from the same initial value.
	Let $U:=u-\bar u$. It is easy to see that $U$ is a paracontrolled solution to the following linear equation
	\begin{equation}\label{eq:U}
	\partial_tU=\Delta U+(b+R)\cdot \nabla U+K\cdot U,\quad U(0)=0,
	\end{equation}
where
	$$R:=\int_0^1\nabla_QH(u,\nabla u+s\nabla(\bar u-u))\dif s,$$
	$$K:=\int_0^1\p_uH(u+s(\bar u-u),\nabla \bar u)\dif s.$$
	Note that by \eqref{DDA} and $u,\bar u\in \mathbb{S}_T^{2-\bar\alpha}(\rho_\eta)$,
	$$
	|R|+|K|\lesssim\rho^{-1}_\delta+|u|+|\bar u|+|\nabla\bar u|+|\nabla u|\lesssim \rho^{-1}_\eta.
	$$Then uniqueness follows from Theorem \ref{Th72}.
\end{proof}

\section{Application to KPZ equations}\label{sec:kpz}

Consider the following KPZ equation:
\begin{equation}\label{eq:kpz}
\sL h=(\p_x h)^{\diamond2}+\xi,\quad h(0)=h_0
\end{equation}
where  $\xi$ is a space-time white noise on $\mathbb{R}^+\times \R$ 
on some stochastic basis $(\Omega,\sF,(\sF_t)_{t\geq 0}, \mP)$. Here the nonlinear term 
$(\p_x h)^{\diamond2}=``(\p_x h)^{2}-\infty"$ with $\infty=\lim_{n\to\infty}\tone{c_n}$ for $\tone{c_n}$ and the approximation $\xi_n$ below. 
We define the $2n$ periodization of $\xi$ by 
$$\tilde{\xi}_n(\psi)=\xi(\psi_n) \textrm{ where } \psi_n(t,x)=\mathbf{1}_{ [-n,n)}(x)\sum_{y\in 2n\mathbb{Z}}\psi(t,x+y).$$ 
Let $\varphi\in C^\infty_c(\R)$ be even and such that 
$\varphi(0)=1$ and define the spatial regularization of $\tilde{\xi}_n$
$$\xi_n=\varphi(n^{-1}\p_x)\tilde{\xi}_n=\mathcal{F}^{-1}(\varphi(n^{-1}\cdot)\mathcal{F}\tilde{\xi}_n).$$
The regularity of the space-time white noise $\xi$ is more rough than the coefficient $f$ given in \eqref{eq:1}. To apply Theorem \ref{existence} we need to introduce the following random fields and use  Schauder estimate to decompose \eqref{eq:kpz} into \eqref{eq:1} and the following equations.  This is the usual way for KPZ equation (cf. \cite{Hai13,GP17,PR18}).

Define 
\begin{equation}
\begin{aligned}
\sL Y_{n} & = \xi_{n} & &  \sL Y=\xi
\\\sL\tone{Y_{n}} & =  (\p_xY_n)^2-\tone{c_n}
& &\sL\ttwo{Y_{n}}  =  2\p_xY_n\p_x\tone{Y}_n 
\\
\sL\tthree{Y_{n}} &=  2\p_x\ttwo{Y_n}\circ\p_x{Y}_n+\tfour{c_n} &  &
\sL\tfour{Y_{n}}  =  (\p_x\tone{Y}_n)^2-\tfour{c_n}\\
\sL\tze{Y_n}&=\p_xY_n,&
\end{aligned} \label{e:trees-def}
\end{equation}
all with zero initial conditions except  $Y(0)(x)=Cx+B(x)$ and $Y_n(0)$ defined similarly as $\xi_n$ with $\xi$ replaced by $Cx+B(x)$, 
where $B$ is a two sided Brownian motion, which is independent of space-time white noise $\xi$, and $C\in \R$. The choice of the initial condition is due to our interest in the KPZ equation starting from its invariant measure (cf. \cite[Section 1.4]{QS15} and \cite{FQ15}). Here $\tone{c_n}$ and $\tfour{c_n}$ are renormalization constants. 
We also set
$$
X_n=\p_xY_n, \quad X=\p_x Y, \quad X^{(\cdot)}=\p_xY^{(\cdot)},
$$
where $(\cdot)$ stands for the above tree. In the following we draw a table 
for the regularity of each $Y^{(\cdot)}$.
For $\gamma>0$ the homogeneities $\alpha_\tau \in \mathbb{R}$ are
given by
$$ \begin{array}{|c|c|c|c|c|c|c|c|c|c|c|c|}
\hline
{\tau} & = & {Y} &
\tone{Y} & \ttwo{Y} &
\tthree{Y} &
\tfour{Y}\\
\hline
\alpha_\tau & = & \frac12-\gamma & 1-\gamma & \frac32-\gamma & 2-\gamma & 2-\gamma \\
\hline
{\tau} & =& X &\tze{Y}&\p_x\tze{Y}\circ \p_xY &\sL \tthree{Y}&\sL \tfour{Y}\\
\hline
\alpha_\tau & =& -\frac12-\gamma&\frac32-\gamma &-\gamma&-\gamma&-\gamma\\
\hline
\end{array} $$

\bl\label{lem:re}
With the above notations, there exist random distributions 
$$
\sY:=\Big\{\tone{Y}, \ttwo{Y},\tthree{Y},\tfour{Y},X,\tze{Y},\p_x\tze{Y}\circ \p_x Y,\sL \tthree{Y},\sL \tfour{Y} \Big\}
$$
and divergence constants $\tone{c_n}$, $\tfour{c_n}$ such that for every $\tau\in \sY$,
$$
\tau\in \cap_{\kappa>0}\mathbb{S}_T^{\alpha_\tau}(\rho_\kappa),
$$
for $\alpha_\tau$ given in the above table. Moreover, for $\tau_n$ defined in \eqref{e:trees-def} $\tau_n\to \tau$ in $L^p(\Omega,\mathbb{S}_T^{\alpha_\tau}(\rho_\kappa))$ for every $p\in [1,\infty)$ and every $\kappa>0$. Furthermore, $Y_n\to Y$ in $L^p(\Omega,\mathbb{S}_T^{\frac12-\gamma}(\rho_{1+\kappa})$ for every $p\in [1,\infty)$. Moreover, there exist random distribution $\nabla \sI^{t}_s(X)\circ X$ such that 
$$\sup_{0\leq s\leq t\leq T}\|\nabla \sI^{t}_s(X_n)\circ X_n(t)-\nabla \sI^{t}_s(X)\circ X(t)\|_{\bC^{-\gamma}(\rho_\kappa)}\to0 \text{ in }L^p(\Omega).$$
\el
\begin{proof}
	Most terms except $\sL \tthree{Y},\sL \tfour{Y}$ in \eqref{e:trees-def} have been considered in \cite[Theorem 3.6]{PR18}. These two terms can also been obtained by similar calculation as in \cite[Theorem 9.3]{GP17} (see also \cite[Section 3.3.1, Section A.2]{ZZ15}). The last convergence result for $\nabla \sI^{t}_s(X)\circ X(t)$ can be obtained similarly as in \cite[Lemma C.1]{PR18}. 
\end{proof}

We make the following decomposition 
$$
h=Y+\tone{Y}+\ttwo{Y}+\widetilde h,
$$
where $\widetilde h$ satisfies the following equation
\begin{equation}\label{e:h1}
\left\{\aligned
\sL \widetilde h&=2\p_x\widetilde h(X+\tone{X}+\ttwo{X})+(\p_x\widetilde h)^2+\sL\tthree{Y}+\sL\tfour{Y}
\\&\quad+(\ttwo{X})^2+2\ttwo{X}\tone{X}+2(X\ttwo{X}-X\circ\ttwo{X}),
\\\widetilde h(0)&=h_0-Y(0).\endaligned
\right.
\end{equation}
Here we use \eqref{e:trees-def}.

Using Lemma \ref{lem:re}, we obtain
\bl\label{lem:z1} There exists a measurable set $\Omega_0$ with $\mP(\Omega_0)=1$ such that for every  $\kappa>0$, $\gamma>0$  and $\omega\in \Omega_0$
$$b:=2\p_x(Y+\tone{Y}+\ttwo{Y})\in L_T^\infty\bC^{-\frac{1}{2}-\gamma}(\rho_\kappa),$$
$$f:=\sL\tthree{Y}+\sL\tfour{Y}+(\ttwo{X})^2
+2\ttwo{X}\tone{X}+2(X\ttwo{X}-X\circ\ttwo{X})\in L_T^\infty\bC^{-\frac{1}{2}-\gamma}(\rho_\kappa).$$
\el
\begin{proof}
	We use Lemma \ref{lem:para} and \eqref{EG1} to have that 
	$$\|(\ttwo{X})^2\|_{\bC^{\frac12-\gamma}(\rho_\kappa)}\lesssim
	\|\ttwo{X}\|_{\bC^{\frac12-\gamma}(\rho_{\kappa/2})}^2,
	$$ $$\|\ttwo{X}\tone{X}\|_{\bC^{-\gamma}(\rho_\kappa)}\lesssim
	\|\ttwo{X}\|_{\bC^{\frac12-\gamma}(\rho_{\kappa/2})}
	\|\tone{X}\|_{\bC^{-\gamma}(\rho_{\kappa/2})},$$
	and 
	$$X\ttwo{X}-X\circ\ttwo{X}=X\succ\ttwo{X}-X\prec\ttwo{X},$$
	to have 
	$$\|X\ttwo{X}-X\circ\ttwo{X}\|_{\bC^{-\frac12-\gamma}(\rho_{\kappa})}
	\lesssim \|X\|_{\bC^{-\frac12-\gamma}(\rho_{\kappa/2})}
	\|\ttwo{X}\|_{\bC^{\frac12-\gamma}(\rho_{\kappa/2})}.$$
	Other terms follows directly from Lemma \ref{lem:re}. 
\end{proof}

As a result $\widetilde h$ satisfies \eqref{eq:1} with $b, f$ given above.  We say that $h$ is a paracontrolled solution to \eqref{eq:kpz} if $\widetilde h$ is a paracontrolled solution to \eqref{e:h1} in the sense of \eqref{eq:ansatz} and \eqref{eq:sharp}.  

Since $\gamma$ can be arbitrary small,
we apply Theorem \ref{main} to obtain
\bt\label{th:k} For every initial condition $\widetilde h(0)\in \bC^{\frac32+2\eps}(\rho_{\eps\delta})$ 
for $\eps>0, 0<\delta:=40\kappa<1$, 
there exists a unique paracontrolled solution 
$$(\widetilde h,\widetilde h^{\sharp})\in(\mathbb{S}^{\frac32-2\kappa^{1/4}}_T(\rho_\eta)\cap \mL^\infty_T(\rho_{2\delta}),\mS_T^{2-3\kappa^{1/4}}(\rho_{2\eta})\cap\mL^\infty_T(\rho_{2\delta+\kappa}))$$   to \eqref{e:h1},
where $$2\Big[(100\kappa)\vee ({\kappa^{1/4}}+80\kappa)\Big]<\eta<\frac14.$$
\et
\begin{proof}
	
	In the following we check other conditions of Theorem \ref{main}. The condition for $H$ is satisfied easily. In the following we prove $(b,f)\in \mathbb{B}^\alpha_T(\rho_\kappa)$. The approximation $\{(b_n,f_n)\}_n$ for $(b,f)$ is given as in  Lemma   \ref{lem:z1} with the corresponding tree $\tau$ replaced by $\tau_n$ in Lemma \ref{lem:re}. In the following we prove that for every $\kappa>0$
	\begin{align}\label{eq:k}
	\sup_n(\ell_T^{b_n}(\rho_\kappa)+\mA_{T,\infty}^{b_n,f_n}(\rho_\kappa))<\infty,
	\end{align}
	with $\ell_T^{b_n}(\rho_\kappa)$ and $\mA_{T,\infty}^{b_n,f_n}(\rho_\kappa))$ defined in \eqref{AA9} and \eqref{AAb}, respectively. In the following we omit the subscript $n$ for simplicity and all the following bounds are uniform in $n$ and $\lambda$. 
	We  first consider
	\begin{align*}\frac1{4}\nabla \sI_\lambda(b)\circ b=&\nabla\sI_\lambda(\p_x(Y+\tone{Y}+\ttwo{Y}))\circ \p_x(Y+\tone{Y}+\ttwo{Y})
	.\end{align*}
	By the last result in Lemma \ref{lem:re} and  Lemma \ref{lem:lambda} we deduce the first term 
	$$\|\nabla\sI_\lambda(\p_x Y)\circ \p_xY\|_{ L_T^\infty\bC^{-\gamma}(\rho_{\kappa})}\lesssim1.$$
	Other terms on the right hand side can be calculated by 
	Lemma \ref{lem:para} and \eqref{EG1} to have
	\begin{align*}
	&\|\nabla\sI_\lambda(\p_x(\tone{Y}+\ttwo{Y}))\circ b\|_{L_T^\infty\bC^{-\gamma}(\rho_{2\kappa})}
	\\\lesssim& (\|\tone{Y}\|_{L_T^\infty\bC^{1-\gamma}(\rho_{\kappa})}+\|\ttwo{Y}\|_{L_T^\infty\bC^{\frac32-\gamma}(\rho_{\kappa})})
	\|b\|_{L_T^\infty\bC^{-\frac12-\gamma}(\rho_{\kappa})}\lesssim1,
	\end{align*}
	and 
	\begin{align*}
	&\|\nabla \sI_\lambda(\p_x Y)\circ \p_x(\tone{Y}+\ttwo{Y})\|_{L_T^\infty\bC^{-\gamma}(\rho_{2\kappa})}
	\\\lesssim& \|Y\|_{L_T^\infty\bC^{\frac12-\gamma}(\rho_{\kappa})}
	(\|\tone{Y}\|_{L_T^\infty\bC^{1-\gamma}(\rho_{\kappa})}
	+\|\ttwo{Y}\|_{L_T^\infty\bC^{\frac32-\gamma}(\rho_{\kappa})})\lesssim1.
	\end{align*}

	On the other hand, we know
	\begin{align*}\nabla \sI_\lambda(f)\circ b=&\nabla\sI_\lambda(f_1)\circ b+\nabla\sI_\lambda(\ttwo{X}\prec X)\circ 2(X+\tone{X}+\ttwo{X}),
	\end{align*}
	with $f_1=f-\ttwo{X}\prec X\in L_T^\infty\bC^{-2\gamma}(\rho_\kappa)$. By Lemma \ref{lem:para} and \eqref{EG1} we know 
	\begin{align*}\|\nabla \sI_\lambda(f_1)\circ b\|_{L_T^\infty\bC^{-\gamma}(\rho_{2\kappa})}\lesssim
	\|f_1\|_{L_T^\infty\bC^{-2\gamma}(\rho_{\kappa})}
	\|b\|_{L_T^\infty\bC^{-\frac12-\gamma}(\rho_{\kappa})}\lesssim1
	,
	\end{align*}
	\begin{align*}
	&\|\nabla\sI_\lambda(\ttwo{X}\prec X)\circ(\tone{X}+\ttwo{X})\|_{L_T^\infty\bC^{-\gamma}(\rho_{2\kappa})}\\\lesssim& \|\ttwo{X}\|_{L_T^\infty\bC^{\frac12-\gamma}(\rho_{\kappa})} \|X\|_{L_T^\infty\bC^{-\frac12-\gamma}(\rho_{\kappa})}(\|\tone{X}\|_{L_T^\infty\bC^{-\gamma}(\rho_{\kappa})}+\|\ttwo{X}\|_{L_T^\infty\bC^{\frac12-\gamma}(\rho_{\kappa})})\lesssim1.
	\end{align*}
	It remains to consider the term $\nabla\sI_\lambda(\ttwo{X}\prec X)\circ X$ and we use the commutator introduced in Lemma \ref{lem:com2} and Lemma \ref{lem:5.1} to have
	\begin{align*}
	\nabla\sI_\lambda(\ttwo{X}\prec X)\circ X=&([\nabla \sI_\lambda,\ttwo{X}\prec]X)\circ X\\&+\mathrm{com}(\ttwo{X},\nabla\sI_\lambda(X),X)
	+\ttwo{X}(\nabla\sI_\lambda(X)\circ X).
	\end{align*}
	By Lemmas \ref{lem:5.1}, \ref{lem:com2} and Lemma \ref{lem:re} 
	$$\|\nabla\sI_\lambda(\ttwo{X}\prec X)\circ X\|_{L_T^\infty\bC^{-\gamma}(\rho_\kappa)}\lesssim1,$$
	where we used time regularity of $\ttwo{X}$, which follows from \eqref{*}. Combining all the above estimates, we deduce \eqref{eq:k} follows. Furthermore, we know that the convergence in Definition \ref{Def216} also holds by using Lemma \ref{lem:re} and Lemma \ref{lem:lambda}, which gives that $(b,f)\in \mB_T^\alpha(\rho_\kappa)$. 	Then the result follows from Theorem \ref{main}.	
\end{proof}

\br\rm
The exponent $\eta$ of the weight could be arbitrary small since $\kappa$ is arbitrary small. This result improves the weight  for the solution of the KPZ equation obtained in \cite{PR18}.
\er

 \appendix
\renewcommand{\appendixname}{Appendix~\Alph{section}}
\renewcommand{\theequation}{A.\arabic{equation}}
\section{Uniqueness of paracontrolled solutions}\label{sec:7.1}
In this subsection we use Hairer and Labb\'e's argument \cite{HL18} to show the uniqueness of paracontrolled solutions. For this aim, 
we use the following time-dependent exponential weight: for $\ell\in(0,1)$,
$$ 
{\bf e}^\ell_t(x):=\exp(-(1+t)\<x\>^\ell),\ t\geq0,\ x\in\mR^d.
$$
We can similarly define the H\"older space with weight ${\bf e}^\ell$ (see \cite{PR18}). For instance,
$$
\|f\|_{L^\infty_T\bC^\alpha({\bf e}^\ell)}:=\sup_{t\in[0,T]}\|f(t,\cdot)\|_{\bC^\alpha({\bf e}^\ell_t)},
$$
and for $\alpha\in(0,1)$,
\begin{align*}
\|f\|_{C^\alpha_TL^\infty({\bf e}^\ell)}&:=
\sup_{0\leq t\leq T} \|f(t){\bf e}^\ell_t\|_{L^\infty}
+\sup_{0\leq s\neq t\leq T} \frac{\|f(t)-f(s)\|_{L^\infty({\bf e}^\ell_{t\vee s})}}{|t-s|^{\alpha}}.
\end{align*}
In particular, for $\alpha\in(0,2)$, we also set
$$
\mS^\alpha_T({\bf e}^\ell):=\|f\|_{L^\infty_T\bC^\alpha({\bf e}^\ell)}+\|f\|_{C^{\alpha/2}_TL^\infty({\bf e}^\ell)}.
$$
By \cite[Lemma 2.10]{MW17}, for any $T>0$, there is a $C=C(T,\ell,d)>0$ such that for all $s,t\in[0,T]$ and $j\geq -1$,
\begin{align}\label{DK1}
\|\Delta_jP_tf\|_{L^\infty({\bf e}^\ell_s)}\lesssim \e^{-2^{2j}t}\|\Delta_jf\|_{L^\infty({\bf e}^\ell_s)}.
\end{align}
Moreover, Lemmas \ref{lem:2.8}, \ref{lem:para}, \ref{lem:com2} and \ref{lem:5.1} still hold for exponential 
weight ${\bf e}^\ell_t$ (see \cite{PR18}).
The following result corresponds to Lemma \ref{Le11}.

\bl\label{graa}
Let $\alpha,\ell\in(0,1)$, $\kappa\in(0,(1-\frac{\alpha}2)\ell)$. 
For any $q\in(\frac{1}{1-\alpha/2-\kappa/\ell},\infty]$ and $T>0$, there  is a constant $C=C(T, d,\alpha,\ell,\theta,\kappa,q)>0$ such that
$$
\|\sI f\|_{\mS_T^{2-\frac2q-\frac{2\kappa}\ell-\alpha}({\bf e}^\ell)}\lesssim_C \|f\|_{L^q_T\bC^{-\alpha}(\rho_{\kappa}{\bf e}^\ell)}.
$$
\el
\begin{proof} 
First of all we have the following simple observation:
\begin{align}\label{GM1}
{\bf e}^\ell_t(x)\lesssim \<x\>^{-\kappa} {\bf e}^\ell_s(x)/|t-s|^{\kappa/\ell},\ \ 0\leq s<t<\infty.
\end{align}
Let $\frac{1}{p}+\frac{1}{q}=1$ and $t\in(0,T]$.  By \eqref{DK1}  and H\"older's inequality, we have for $j\geq-1$,
	\begin{align*}
	\|\Delta_j\sI f(t)\|_{L^\infty({\bf e}^\ell_t)}
&\lesssim\int^t_0\e^{-2^{2j}(t-s)}\|\Delta_jf(s)\|_{L^\infty({\bf e}^\ell_t)}\dif s
	\\&\lesssim \int^t_0\frac{\e^{-2^{2j}(t-s)}}{|t-s|^{\kappa/\ell}}\|\Delta_jf(s)\|_{L^\infty(\rho_\kappa{\bf e}^\ell_s)}\dif s
	\\&\lesssim 2^{\alpha j}\left(\int^t_0\frac{\e^{-p2^{2j}(t-s)}}{|t-s|^{p\kappa/\ell}}\dif s\right)^{1/p}
	\|f\|_{L^q_t\bC^{-\alpha}(\rho_\kappa{\bf e}^\ell)}\\
	&\lesssim 2^{-(\frac2p-\frac{2\kappa}\ell-\alpha) j}\|f\|_{L^q_t\bC^{-\alpha}(\rho_\kappa{\bf e}^\ell)},
	\end{align*}
	which in turn gives that
\begin{align}\label{GFM}
	\|\sI f\|_{L^\infty_T\bC^{2-\frac2q-\frac{2\kappa}\ell-\alpha}({\bf e}^\ell)}
\lesssim \|f\|_{L^q_T\bC^{-\alpha}(\rho_{\kappa}{\bf e}^\ell)}.
\end{align}
On the other hand, for $0\leq t_1<t_2\leq T$, we have
	\begin{align*} 
	\|\sI f(t_2)- \sI f(t_1)\|_{L^\infty({\bf e}^\ell_{t_2})}
	&\leq\|(P_{t_2-t_1}-I)\sI f(t_1)\|_{L^\infty({\bf e}^\ell_{t_2})}\\
	&\quad+\left\|\int_{t_1}^{t_2}P_{t_2-s}f(s) \dif s\right\|_{L^\infty({\bf e}^\ell_{t_2})}=:I_1+I_2.
	\end{align*}
	For $I_1$, by \eqref{E2} and \eqref{GFM} we have
\begin{align*}
I_1&\lesssim(t_2-t_1)^{1-\frac\alpha 2-\frac{1}{q}- \frac{\kappa}{\ell}}
	\|\sI f(t_1)\|_{\bC^{2-\alpha-\frac{2}{q}- \frac{2\kappa}{\ell}}({\bf e}^\ell_{t_2})}
\\&\lesssim(t_2-t_1)^{1-\frac\alpha 2-\frac{1}{q}- \frac{\kappa}{\ell}}\|f\|_{L^q_T\bC^{-\alpha}(\rho_{\kappa}{\bf e}^\ell)}.
\end{align*}
For $I_2$, by \eqref{E1}, \eqref{GM1} and H\"older's inequality, we have
\begin{align*}
I_2&\lesssim\int_{t_1}^{t_2}(t_2-s)^{-\frac\alpha2}\|f(s)\|_{\bC^{-\alpha}({\bf e}^\ell_{t_2})}\dif s
	\\&\lesssim\int_{t_1}^{t_2}(t_2-s)^{-\frac\alpha2-\frac\kappa\ell}\|f(s)\|_{\bC^{-\alpha}(\rho_\kappa{\bf e}^\ell_s)}\dif s
	\\&\lesssim(t_2-t_1)^{1-\frac\alpha 2-\frac{1}{q}- \frac{\kappa}{\ell}}\|f\|_{L^q_T\bC^{-\alpha}(\rho_{\kappa}{\bf e}^\ell)}.
	\end{align*}
Combining the above estimates, we obtain the desired estimate. 
\end{proof}

Now we consider the following linear equation:
\begin{align}\label{PDE9}
\sL u=(b+\bar b)\cdot\nabla u+hu,\ \ u(0)\equiv 0,
\end{align}
where $b\in \cap_{T>0}\mB^\alpha_T(\rho_\kappa)$ and $\bar b,h\in \cap_{T>0}L^\infty_T(\rho_\eta)$.
Let 
$$
(u,u^\sharp)\in \cap_{T>0}\mS^{2-\alpha}_T(\rho_\eta)\times \mS^{3-2\alpha}_T(\rho_{2\eta})
$$ 
be the paracontrolled solution of PDE \eqref{PDE9}. That is,
\begin{align}\label{DT01}
u=\nabla u\Prec  \sI b+u^\sharp,
\end{align}
with $u^\sharp$ solving the following PDE in weak sense
\begin{align}\label{DT010}
\sL u^\sharp&=\nabla u\prec b-\nabla u\Prec b+\nabla u\succ b+b\circ\nabla u\no
\\&\quad+\bar b\cdot\nabla u+hu-[\sL, \nabla u\Prec]\sI b,
\end{align}
where
\begin{align}
b\circ\nabla u&=b\circ(\nabla^2 u\prec \sI b)+(b\circ\nabla \sI b)\cdot \nabla u+\textrm{com}\no\\
&\quad+\textrm{com}_1+b\circ\nabla u^\sharp,\label{FQ22}
\end{align}
and
$$\textrm{com}_1:=b\circ\nabla [\nabla u\Prec \sI b-\nabla u\prec \sI b]$$
and
$$
\textrm{com}:=\mathrm{com}(\nabla u, \nabla \sI b,b).
$$

\bt\label{Th72}
Let $\ell\in(0,1)$ and $\kappa\in(0,\tfrac{(2-3\alpha)\ell}{6})$, $\eta\in(0,\frac{(1-\alpha)\ell}2)$. Suppose that
\begin{align*}
&b\in \cap_{T>0}\mB^\alpha_T(\rho_\kappa),\ \ \bar b,h\in \cap_{T>0}\mL^\infty_T(\rho_\eta),\\ &\beta\in(\alpha,(2-2\alpha-\tfrac{6\kappa}{\ell})\wedge(1-\frac{2\eta}{\ell})),\ \ \gamma\in(\alpha,2-2\alpha-\tfrac{4\kappa}{\ell}).
\end{align*}
The unique paracontrolled solution to PDE \eqref{PDE9} in the sense of Definition \ref{def:para1} with 
$$
(u,u^\sharp)\in\mS_T^{\gamma+\alpha}({\bf e}^\ell)\times L_T^\infty\bC^{\beta+1}({\bf e}^\ell)
$$
is zero.
\et
\begin{proof}
Let $T>0$. Choose $q$ large enough such that
 $$\alpha<\gamma\leq 2-2\alpha-\frac2q-\frac{4\kappa}\ell,\ \
 \alpha<\beta\leq (2-2\alpha-\frac2q-\frac{6\kappa}\ell)\wedge(1-\frac{2\eta}\ell).
 $$
First of all, by Lemmas \ref{graa} and \ref{lem:para}, we have
\begin{align*} 
\begin{split}
\|u\|_{\mS_T^{2-\alpha-\frac2q-\frac{4\kappa}\ell}({\bf e}^\ell)}
	&\lesssim\|b\prec\nabla u+b\succ\nabla u+b\circ\nabla u\|_{L^q_T\bC^{-\alpha}(\rho_{2\kappa}{\bf e}^\ell)}+\|\bar b\cdot\nabla u+hu\|_{L_T^qL^\infty(\rho_{\eta}{\bf e}^\ell)}
	\\&\lesssim \|b\|_{L^\infty_T\bC^{-\alpha}(\rho_\kappa)}\|\nabla u\|_{L^q_TL^\infty({\bf e}^\ell)}
	+\|b\circ\nabla u\|_{L^q_T\bC^{-\alpha}(\rho_{2\kappa}{\bf e}^\ell)}
	\\&\quad+\|\bar b\|_{\mL_T^\infty(\rho_\eta)}\|\nabla u\|_{L_T^qL^\infty({\bf e}^\ell)}+\|h\|_{\mL_T^\infty(\rho_\eta)}\| u\|_{L_T^qL^\infty({\bf e}^\ell)},
	\end{split}
\end{align*}
and by the corresponding version of Lemma \ref{lem:5.1} for exponential weight ${\bf e}^\ell$ (see \cite[Lemma 2.10]{PR18}),  
$$\aligned
\|u^\sharp\|_{L_T^\infty\bC^{\beta+1}({\bf e}^\ell)}
&\lesssim \|\nabla u\prec b-\nabla u\Prec b+\nabla u\succ b-[\sL, \nabla u\Prec]\sI b\|_{L^\infty_T\bC^{1-2\alpha-\frac2q-\frac{4\kappa}\ell}(\rho_\kappa{\bf e}^\ell)}\\
&\quad+\|b\circ\nabla u\|_{L^q_T\bC^{1-2\alpha}(\rho_{2\kappa}{\bf e}^\ell)}+\|\bar b\cdot\nabla u+hu\|_{\mL_T^\infty(\rho_{\eta}{\bf e}^\ell)}
\\
&\lesssim \|b\circ\nabla u\|_{L^q_T\bC^{1-2\alpha}(\rho_{2\kappa}{\bf e}^\ell)}+\|u\|_{\mS_T^{2-\alpha-\frac2q-\frac{4\kappa}\ell}({\bf e}^\ell)}
\\&\quad+\|\bar b\|_{\mL_T^\infty(\rho_\eta)}\|\nabla u\|_{\mL_T^\infty({\bf e}^\ell)}+\|h\|_{\mL_T^\infty(\rho_\eta)}\| u\|_{\mL_T^\infty({\bf e}^\ell)}
\\&\lesssim \|u\|_{\mS_T^{2-\alpha-\frac2q-\frac{4\kappa}\ell}({\bf e}^\ell)}+\|b\circ\nabla u\|_{L^q_T\bC^{1-2\alpha}(\rho_{2\kappa}{\bf e}^\ell)}.
\endaligned$$
Moreover, by Lemma \ref{Le32} with $(\rho,\bar\rho)=(\rho_\kappa,{\bf e}^\ell_t)$, 
\begin{align*}
\|(b\circ\nabla u)(t)\|_{\bC^{1-2\alpha}(\rho_{2\kappa}{\bf e}^\ell_t)}
\lesssim \|u\|_{\mS^{\gamma+\alpha}_t({\bf e}^\ell)}+\|u^\sharp(t)\|_{\bC^{\beta+1}(\rho_\kappa{\bf e}^\ell_t)}.
\end{align*}
Combining the above three estimates, we obtain
\begin{align*}
&\|u\|_{\mS_T^{\gamma+\alpha}({\bf e}^\ell)}+\|u^\sharp\|_{L_T^\infty\bC^{\beta+1}({\bf e}^\ell)}
	\\\lesssim&\|\nabla u\|_{L^q_TL^\infty({\bf e}^\ell)}
	+\| u\|_{L_T^qL^\infty({\bf e}^\ell)}+\|b\circ\nabla u\|_{L^q_T\bC^{1-2\alpha}(\rho_{2\kappa}{\bf e}^\ell)}
	\\\lesssim&\left(\int^T_0\Big(\|u\|^q_{\mS^{\gamma+\alpha}_t({\bf e}^\ell)}+\|u^\sharp(t)\|^q_{\bC^{\beta+1}(\rho_\kappa{\bf e}^\ell_t)}\Big)\dif t\right)^{1/q},
\end{align*}
which implies $u\equiv 0$ by Gronwall's inequality.
\end{proof}
	
\renewcommand{\theequation}{B.\arabic{equation}}

\section{Exponential moment estimates for SDEs}\label{sec:7.2}

In this section we consider the following SDE:
$$
\dif X_t=b(t,X_t)\dif t+\sigma(t,X_t)\dif W_t, \  X_0=x.
$$
We have the following  exponential moment estimates for $X_t$.
\bl\label{th:7.2}
Suppose that $\sigma$ is bounded and $b$ is linear growth. Then for any $\alpha\in[0,2)$ and $T, \gamma>0$, there is a constant $C>0$
such that for all $x\in\mR^d$,
$$
\mE\e^{\gamma\sup_{t\in[0,T]} \<X_t\>^\alpha}\leq C\e^{\<x\>^\alpha}.
$$
\el
\begin{proof}
	Let $\beta\in(\alpha,2)$. Recall $\<x\>^\beta=(1+|x|^2)^{\beta/2}$. By It\^o's formula, we have
	\begin{align*}
	M_t:=\e^{-\lambda t}\<X_t\>^\beta&=\<x\>^\beta+\int^t_0 \eta_s\dif s+\int^t_0\xi_s\dif W_s,
	\end{align*}
	where
	\begin{align*}
	\eta_s&:=\e^{-\lambda s}\beta \Big[X_s\cdot b(s,X_s)+ \tr(\sigma\sigma^*)(s,X_s)/2\Big]\<X_s\>^{\beta-2}\\
	&\quad+\beta(\tfrac{\beta}{2}-1)\e^{-\lambda s}|\sigma^*(s,X_s) X_s|^2 \<X_s\>^{\beta-4}-\lambda\e^{-\lambda s} \<X_s\>^\beta,
	\end{align*}
	and
	$$
	\xi_s:=\beta\e^{-\lambda s}\sigma^*(s,X_s)X_s\<X_s\>^{\beta-2}.
	$$
	By the linear growth of $b$ and the boundedness of $\sigma$, there is a $\lambda$ large enough so that
	$$
	\eta_s\leq 0
	$$
	and
	$$
	|\xi_s|^2\leq C\e^{-\lambda s}\<X_s\>^{2(\beta-1)}\leq C M^{2-\frac{2}{\beta}}_s.
	$$
	Now by \cite[Theorem 1.1]{Hu09}, we obtain the desired estimate.
\end{proof}	

\section*{Acknowlegement}
 We are very grateful to Nicolas Perkowski for proposing this problem to us and 
 sharing his idea on this problem with us (especially the idea mentioned in Remark \ref{re}), where we benefit a lot.   
 X. Zhang is partially supported by NSFC (No. 11731009). R. Zhu and X. Zhu are grateful to
the financial supports of the NSFC (No. 11671035, 11771037,
11922103) and the financial support by the DFG through the CRC 1283 “Taming uncertainty
and profiting from randomness and low regularity in analysis, stochastics and their applications” and the support by key Lab of Random Complex Structures and Data Science, Chinese Academy of Science.

\end{document}